\documentclass{mscs}
\usepackage{etex}
\usepackage{scalefnt}
\usepackage{latexsym}
\usepackage{amsmath}
\usepackage{amssymb}
\usepackage{stmaryrd}
\usepackage{mathrsfs}
\usepackage[sans]{dsfont}
\usepackage{proof}
\usepackage{cmll}
\usepackage[pdftex]{hyperref}
\usepackage[all]{xy}

\def\limp{\Rightarrow}
\def\liff{\Leftrightarrow}

\def\Nat{\mathsf{N}}

\def\Lam#1#2{\lambda#1\,{.}\,#2}
\def\cc{\mathtt{c\!c}}
\def\Fork{{\pitchfork}}
\def\Por{\mathbf{p}\text{-}\mathbf{or}}
\def\<{\langle}
\def\>{\rangle}
\def\FV{\mathit{FV}}

\def\redd{\twoheadrightarrow}
\def\dom{\mathrm{dom}}
\def\SN{\mathrm{SN}}
\def\SAT{\mathbf{SAT}}

\def\A{\mathscr{A}}
\def\B{\mathscr{B}}
\def\cle{\preccurlyeq}
\def\cge{\succcurlyeq}
\def\meet{\curlywedge}
\def\join{\curlyvee}
\def\bigmeet{\bigcurlywedge}
\def\bigjoin{\bigcurlyvee}
\def\slambda{\boldsymbol{\lambda}}
\def\pto{\rightharpoonup}

\def\onto{\twoheadrightarrow}
\def\inonto{\mathbin{\tilde{\to}}}
\def\PER{\mathrm{PER}}
\def\defin{\mathrel{{\downarrow}{\in}}}
\def\Set{\ensuremath{\mathbf{Set}}}
\def\Pos{\ensuremath{\mathbf{Pos}}}
\def\HA{\ensuremath{\mathbf{HA}}}

\def\op{\mathrm{op}}
\def\id{\mathrm{id}}
\def\C{\mathbf{C}}
\def\P{\mathsf{P}}

\def\simto{\mathbin{\mathop{\to}\limits^{\sim}}}
\def\pullbackbox#1{\hbox to#1{\vbox to#1{\hbox to#1{\hfil}\vfil\hrule}\hskip-0.4pt\vrule}}
\def\pullback#1{\hbox to 0pt{\vbox to 0pt{\pullbackbox{#1}\vss}\hss}}
\def\Prop{\textit{Prop}}
\def\Tr{\mathsf{Tr}}
\def\Int#1{\llbracket#1\rrbracket}

\newcommand\ScaleForall[1]{\vcenter{\hbox{\scalefont{#1}$\forall$}}}
\newcommand\ScaleExists[1]{\vcenter{\hbox{\scalefont{#1}$\exists$}}}
\DeclareMathOperator*\bigforall{%
  \vphantom\sum\mathchoice{\ScaleForall{2}}{\ScaleForall{1.4}}{\ScaleForall{1}}{\ScaleForall{0.75}}}
\DeclareMathOperator*\bigexists{%
  \vphantom\sum\mathchoice{\ScaleExists{2}}{\ScaleExists{1.4}}{\ScaleExists{1}}{\ScaleExists{0.75}}}

\def\M{\mathscr{M}}

\def\Pow{\mathfrak{P}}

\def\N{\mathds{N}}

\let\ds=\displaystyle
\outer\long\def\COUIC#1{}
\def\Sep{\mathrm{Sep}}
\def\SJ{\mathit{S}^0_{\!J}}
\def\SK{\mathit{S}^0_{\!K}}
\def\ent{\vdash}
\def\tnent{\mathrel{{\dashv}{\vdash}}}

\newtheorem{theorem}{Theorem}[section]
\newtheorem{proposition}[theorem]{Proposition}
\newtheorem{fact}[theorem]{Fact}
\newtheorem{lemma}[theorem]{Lemma}
\newtheorem{corollary}[theorem]{Corollary}
\newtheorem{definition}[theorem]{Definition}
\newtheorem{remark}[theorem]{Remark}
\newtheorem{remarks}[theorem]{Remarks}
\newtheorem{example}[theorem]{Example}

\title{Implicative algebras: a new foundation
  for realizability and forcing}
\author[Alexandre Miquel]{%
  A\ls l\ls e\ls x\ls a\ls n\ls d\ls r\ls e\ns
  M\ls I\ls Q\ls U\ls E\ls L$^1$\\
  $^1$ Instituto de Matem{\'a}tica y Estad{\'i}stica
  Prof. Ing. Rafael Laguardia\addressbreak
  Facultad de Ingenier{\'ia} --
  Universidad de la Rep{\'u}blica\addressbreak
  Julio Herrera y Reissig 565 -- Montevideo C.P. 11300 -- URUGUAY%
  \thanks{This work was partly supported by the Uruguayan National
    Research \& Innovation Agency (ANII) under the project
    ``Realizability, Forcing and Quantum Computing'',
    FCE\_1\_2014\_1\_104800.}
}

\begin{document}
\maketitle

\begin{abstract}
  We introduce the notion of implicative algebra, a simple algebraic
  structure intended to factorize the model-theoretic constructions
  underlying forcing and realizability (both in intuitionistic and
  classical logic).
  The salient feature of this structure is that its elements can be
  seen both as truth values and as (generalized) realizers, thus
  blurring the frontier between proofs and types.
  We show that each implicative algebra induces a ($\Set$-based)
  tripos, using a construction that is reminiscent from the
  construction of a realizability tripos from a partial combinatory
  algebra.
  Relating this construction with the corresponding constructions in
  forcing and realizability, we conclude that the class of implicative
  triposes encompass all forcing triposes (both intuitionistic and
  classical), all classical realizability triposes (in the sense of
  Krivine) and all intuitionistic realizability triposes built from
  partial combinatory algebras.
\end{abstract}
\nocite{Ruy06}

\setcounter{tocdepth}{3}

\section{Introduction}
\label{ss:Intro}

In this paper, we introduce the notion of implicative algebra, a
simple algebraic structure that is intended to factorize the
model-theoretic constructions underlying forcing and realizability,
both in intuitionistic and classical logic.

Historically, the method of forcing was introduced by
Cohen~\cite{Coh63,Coh64} to prove the relative consistency of the
negation of the continuum hypothesis w.r.t.\ the axioms of set theory.
Since then, forcing has been widely investigated---both from a
proof-theoretic point of view and from a model-theoretic point of
view---, and it now constitutes a standard item in the toolbox of set
theorists~\cite{Jec02}.
From a model-theoretic point of view, the method of forcing can be
understood as a particular way to construct Boolean-valued models of
the considered theory (typically: set theory or higher-order
arithmetic), in which each formula~$\phi$ is interpreted as an element
$$\Int{\phi}~\in~B$$
of a given complete Boolean algebra~$B$.
If one is only interested in interpreting intuitionistic theories,
one can replace complete Boolean algebras by complete Heyting
algebras, in which case similar construction methods give us
Heyting-valued models, that are essentially equivalent to Kripke
(i.e.\ intuitionistic) forcing or Beth forcing.

As observed by Scott~\cite{Oos02}, there is a strong similarity
between (intuitionistic or classical) forcing and the method of
realizability, that was introduced by Kleene~\cite{Kle45} to give a
constructive semantics to Heyting (i.e.\ intuitionistic) arithmetic.
From a model-theoretic point of view, the method of realizability
interprets each closed formula~$\phi$ as a set of realizers
$$\Int{\phi}~\in~\Pow(P)$$
where~$P$ is a suitable algebra of ``programs'' (typically: a partial
combinatory algebra), following the Brouwer-Heyting-Kolmogorov
semantics for intuitionistic logic.
(Here, the symbol $\Pow$ stands for the set-theoretic powerset
operator.)
Although the method of realizability was initially introduced for
intuitionistic first-order arithmetic, it extends to intuitionistic
higher-order arithmetic and even to intuitionistic Zermelo-Fraenkel
set theory~\cite{Myh73,Fri73,McCPhd}.

For a long time, the method of realizability was limited to
intuitionistic logic.
However, from the mid-90's, Krivine reformulated~\cite{Kri09} the
principles of realizability to make them compatible with classical
logic, using the correspondence between classical reasoning and
control operators discovered by Griffin~\cite{Gri90}.
Technically, classical realizability departs from intuitionistic
realizability by interpreting each formula~$\phi$ not as a set of
realizers, but as a set of counter-realizers (a.k.a.\ a falsity value)
$$\Int{\phi}~\in~\Pow(\Pi)$$
where~$\Pi$ is the set of stacks associated to an algebra of classical
programs~$\Lambda$~\cite{Kri11,Str13}.
The corresponding set of realizers (or truth value) is then defined
indirectly, as the orthogonal $\Int{\phi}^{\Bot}\subseteq\Lambda$ of
the falsity value $\Int{\phi}\subseteq\Pi$ with respect to a
particular set of processes $\Bot\subseteq\Lambda\times\Pi$ ---the
pole of the model--- that parameterizes the construction.
As for intuitionistic realizability, classical realizability extends
to higher-order arithmetic and even to (classical) Zermelo-Fraenkel
set theory~\cite{Kri01,Kri12}, possibly enriched with some weak forms
of the axiom of choice.

In spite of their similarity, there is a fundamental difference
between forcing and realizability, regarding the treatment of
connectives and quantifiers.
In forcing, conjunction and disjunction are interpreted as binary
meets and joins
$$\Int{\phi\land\psi}~=~\Int{\phi}\meet\Int{\psi}
\qquad\text{and}\qquad
\Int{\phi\lor\psi}~=~\Int{\phi}\join\Int{\psi}$$
(respectively writing $b\meet c$ and $b\join c$ the meet and the join
of two elements $b,c\in B$), whereas universal and existential
quantifications are interpreted by
$$\Int{\forall x\,\phi(x)}~=~\bigmeet_{v\in\M}\Int{\phi(v)}
\qquad\text{and}\qquad
\Int{\exists x\,\phi(x)}~=~\bigjoin_{v\in\M}\Int{\phi(v)}\,.$$
So that from the point of view of (intuitionistic or classical)
forcing, conjunction and disjunction are just finite forms of
universal and existential quantifications.
This is definitely not the case in intuitionistic realizability, where
conjunctions and disjunctions are interpreted as Cartesian products
and direct sums
$$\Int{\phi\land\psi}~=~\Int{\phi}\times\Int{\psi}
\qquad\text{and}\qquad
\Int{\phi\lor\psi}~=~\Int{\phi}+\Int{\psi}$$
whereas universal and existential quantifications are still
interpreted uniformly%
\footnote{Here, we put aside the case of numeric (or arithmetic)
  quantifications, that can always be decomposed as a uniform
  quantification followed by a relativization:
  $(\forall x\,{\in}\,\N)\phi(x)\equiv
  \forall x\,(x\in\N\limp\phi(x))$ and
  $(\exists x\,{\in}\,\N)\phi(x)\equiv
  \exists x\,(x\in\N\land\phi(x))$.}:
$$\Int{\forall x\,\phi(x)}~=~\bigcap_{v\in\M}\Int{\phi(v)}
\qquad\text{and}\qquad
\Int{\exists x\,\phi(x)}~=~\bigcup_{v\in\M}\Int{\phi(v)}\,.$$
(The situation is slightly more complex in classical realizability, in
which existential quantification and disjunction have to be
interpreted negatively.
But the above picture still holds for conjunctions and universal
quantifications.)
In some sense, realizability is more faithful to proof-theory, in
which proving a universal quantification
$$\infer{\vdash\forall x\,\phi(x)}{\vdash\phi(x)}$$
(that is: providing a generic proof that holds of all instances
of the variable~$x$) is much stronger than proving a (finitary or
infinitary) conjunction:
$$\infer{\vdash\phi(t_0)\land\phi(t_1)\land\phi(t_2)\land
  \cdots\land\phi(t_n)\land\cdots}{
  \vdash\phi(t_0)&\vdash\phi(t_1)&\vdash\phi(t_2)&\cdots&
  \vdash\phi(t_n)&\cdots
}$$
(that is: providing a distinct proof for each instance of the
variable~$x$)%
\footnote{The distinction between uniform constructions
  (e.g.\ intersection and union types) and non-uniform constructions
  (Cartesian product and direct sum) has always been overlooked in
  model theory, although it is at the core of the phenomenon of
  incompleteness in logic.
  Indeed, G{\"o}del's undecidable sentence is a $\Pi^0_1$-formula
  $G\equiv\forall x\,\phi(x)$ that is built from a particular
  $\Delta^0_0$-predicate $\phi(x)$ that has no generic proof, although
  each closed instance $\phi(n)$ ($n\in\N$) has.}.

But what do have in common an element of a complete Heyting (or
Boolean) algebra, a set of realizers (taken in a combinatory
algebra~$P$) or a set of counter-realizers (taken in a set of
stacks~$\Pi$)?
The aim of this paper is to show that all these notions of `truth
value' pertain to implicative algebras, a surprisingly simple
algebraic structure whose most remarkable feature is to use the same
set to represent truth values and realizers, thus blurring the
frontier between proofs and types.
As a matter of fact (Section~\ref{ss:Manifesto}), implicative algebras
offer a fresh semantic reading of \emph{typing} and \emph{definitional
  ordering} in terms of \emph{subtyping}, that is now the primitive
notion.

However, implicative algebras do not only encompass the various
notions of `truth value' underlying forcing and realizability, but
they also allow us to factorize the corresponding model-theoretic
constructions.
For that, we shall place ourselves in the categorical framework of
\emph{triposes}~\cite{HJP80}, that was introduced precisely to compare
forcing and realizability in the perspective of constructing
categorical models of higher-order logic.
Intuitively, a tripos is a $\Set$-indexed Heyting algebra of
`predicates' $\P:\Set^{\op}\to\HA$ (see Def.~\ref{d:Tripos}
p.~\pageref{d:Tripos}) that constitutes a (categorical) model of
higher-order logic.
Triposes can be built from a variety of algebraic structures, such as
complete Heyting (or Boolean) algebras, Partial Combinatory Algebras,
Ordered Combinatory Algebras~\cite{Oos08} and even Abstract Krivine
Structures~\cite{Str13}.
And each tripos can be turned into a topos (i.e.\ a `Set-like
category') via the standard tripos-to-topos
construction~\cite{HJP80}.

As we shall see in Section~\ref{s:Tripos}, all the above tripos
constructions (as well as the corresponding topos constructions) can
be factored through a unique construction, namely: the construction of
an \emph{implicative tripos} from a given implicative algebra.
Thanks to this factorization, we will be able to characterize forcing
in terms of non-determinism (from the point of view of generalized
realizers), and we shall prove that classical implicative triposes are
equivalent to Krivine's classical realizability triposes.

\subsection*{Sources of inspiration \& related works}
The notion of implicative algebra emerged from so many sources of
inspirations that it is almost impossible to list them all here.
Basically, implicative algebras were designed from a close analysis of
the algebraic structure underlying falsity values in Krivine
realizability~\cite{Kri09}, noticing that this structure is very
similar to the one of reducibility
candidates~\cite{Tai67,Gir89,WerPhD,Par97}.
Other sources of inspiration are the notion of semantic type in
coherence spaces~\cite{Miq00} as well as the notion of fact (or
behavior) in phase semantics~\cite{Gir87}.

The idea of reconstructing $\lambda$-terms from implication and
infinitary meets came from filter models~\cite{BCD83}, that are
strongly related to implicative algebras from a technical point of
view, although they are not implicative algebras.
The same idea appeared implicitly in Streicher's reconstruction of
Krivine's tripos~\cite{Str13} and more explicitly in~\cite{FFGMM17},
that introduced many of the ideas that are presented here, but in a
slightly different framework, closer to Streicher's.
Similar ideas were developed independently by Ruyer, whose
applicative lattices~\cite[p.~29]{Ruy06} are equivalent to a
particular case of implicative structures, namely: to the implicative
structures that are compatible with joins (Section~\ref{ss:ExJoin}).

\subsection*{Outline of the paper}
In Section~\ref{s:ImpStruct}, we introduce the notion of
\emph{implicative structure} (as a natural generalization of the
notion of complete Heyting algebra), and show how the elements of such
a structure can be used to represent both truth values (or types) and
realizers.
In Section~\ref{s:Separation}, we introduce the fundamental notion of
\emph{separator} (that generalizes the usual notion of filter) as well
as the accompanying notion of \emph{implicative algebra}.
We show how each separator induces a particular Heyting algebra
(intuitively: the corresponding algebra of propositions), and give a
first account on the relationship between forcing and non-determinism
(Prop.~\ref{p:CharacPrincFilter} p.~\pageref{p:CharacPrincFilter}).
Section~\ref{s:Tripos} is devoted to the construction of the
\emph{implicative tripos} induced by a particular implicative algebra.
We show that implicative triposes encompass many well-known triposes,
namely:
(intuitionistic and classical) forcing triposes,
classical realizability triposes~\cite{Str13},
intuitionistic realizability triposes induced by (total) combinatory
algebras, and even
intuitionistic realizability triposes induced by partial combinatory
algebras (Section~\ref{ss:TriposRealizJ}).
We also characterize forcing triposes as the non-deterministic
implicative triposes (Theorem~\ref{th:CharacForcingTriposes}
p.~\pageref{th:CharacForcingTriposes}), and show that classical
implicative triposes are equivalent to classical realizability
triposes (Theorem~\ref{th:UnivAKS} p.~\pageref{th:UnivAKS}).

\section{Implicative structures}
\label{s:ImpStruct}

\subsection{Definition}

\begin{definition}[Implicative structure]\label{d:ImpStruct}
  An \emph{implicative structure} is a complete meet-semilattice
  $(\A,{\cle})$ equipped with a binary operation
  $(a,b)\mapsto(a\to b)$, called the \emph{implication of~$\A$},
  that fulfills the following two axioms:
  \begin{enumerate}[(99)]
  \item[(1)] Implication is anti-monotonic w.r.t.\ its first operand
    and monotonic w.r.t.\ its second operand:
    $$\text{if}~~a'\cle a~~\text{and}~~b\cle b',
    \quad\text{then}\quad(a\to b)\cle(a'\to b')
    \eqno(a,a',b,b'\in\A)$$
  \item[(2)] Implication commutes with arbitrary meets on its second
    operand:
    $$a\to\bigmeet_{b\in B}\!\!b~=~\bigmeet_{b\in B}\!(a\to b)
    \eqno(a\in\A,~B\subseteq\A)$$
  \end{enumerate}
\end{definition}

\begin{remarks}\label{r:ImpStruct}
  (1)~~By saying that $(\A,\cle)$ is a complete meet-semilattice, we
  mean that every subset of~$\A$ has a greatest lower bound (i.e.\ a
  \emph{meet}).
  Such a poset has always a smallest element $\bot=\bigmeet\A$ and a
  largest element $\top=\bigmeet\varnothing$.
  More generally, every subset of~$\A$ has also a least upper bound
  (i.e.\ a \emph{join}), so that a complete meet-semilattice is
  actually the same as a complete lattice.
  However, in what follows, we shall mainly be interested in the
  meet-semilattice structure of implicative structures, so that it is
  convenient to think that implicative structures are (complete)
  lattices only by accident.\par
  (2)~~In the particular case where $B=\varnothing$, axiom~(2) states
  that $(a\to\top)=\top$ for all $a\in\A$. 
  (Recall that $\top=\bigmeet\varnothing$.)
  In some circumstances, it is desirable to relax this equality, by
  requiring that axiom~(2) hold only for the nonempty subsets~$B$
  of~$\A$.
  Formally, we call a \emph{quasi-implicative structure} any complete
  meet-semilattice~$\A$ equipped with a binary operation
  $(a,b)\mapsto(a\to b)$ that fulfills both axioms~(1) and~(2) of
  Def.~\ref{d:ImpStruct}, the latter being restricted to the case where
  $B\neq\varnothing$.
  From this definition, we easily check that a quasi-implicative
  structure is an implicative structure if and only if
  $(\top\to\top)=\top$.
\end{remarks}

\subsection{Examples of implicative and quasi-implicative structures}
\label{ss:ExamplesImpStruct}

\subsubsection{Complete Heyting algebras}\label{sss:ComplHA}
The most obvious examples of implicative structures are given by
complete Heyting algebras.
Recall that a \emph{Heyting algebra} is a bounded lattice $(H,{\cle})$
equipped with a binary operation $(a,b)\mapsto(a\to b)$ (Heyting's
implication) characterized by the adjunction
$$(c\meet a)\cle b\quad\text{iff}\quad c\cle(a\to b)
\eqno(a,b,c\in H)$$

Historically, Heyting algebras have been introduced as the
intuitionistic counterpart of Boolean algebras, and they can be used
to interpret intuitionistic provability the same way as Boolean
algebras can be used to interpret classical provability.
In this framework, conjunction and disjunction are interpreted by
binary meets and joins, whereas implication is interpreted by the
operation $a\to b$.
This interpretation validates all reasoning principles of
intuitionistic propositional logic, in the sense that every
propositional formula that is intuitionistically valid is denoted
by the truth value~$\top$.

\emph{Boolean algebras} are the Heyting algebras $(H,{\cle})$ in which
negation is involutive, that is: $\lnot\lnot a=a$ for all
$a\in H$, where negation is defined by $\lnot a:=(a\to\bot)$.
Boolean algebras more generally validate all classical reasoning
principles, such as the law of excluded middle ($a\join\lnot a=\top$)
or Peirce's law ($(((a\to b)\to a)\to a)=\top$).

A Heyting (or Boolean) algebra is \emph{complete} when the underlying
lattice is complete.
In a complete Heyting algebra, the interpretation depicted above
naturally extends to all formulas of predicate logic, by interpreting
universal and existential quantifications as meets and joins of
families of truth values indexed over a fixed nonempty set.
Again, this (extended) interpretation validates all reasoning
principles of intuitionistic predicate logic.
It is easy to check that in a complete Heyting algebra, Heyting's
implication fulfills both axioms~(1) and~(2) of Def.~\ref{d:ImpStruct},
so that:
\begin{fact}
  Every complete Heyting algebra is an implicative structure.
\end{fact}

In what follows, we shall say that an implicative structure
$(\A,{\cle},{\to})$ is a \emph{complete Heyting algebra} when the
underlying lattice $(\A,{\cle})$ is a (complete) Heyting algebra, and
when the accompanying implication $(a,b)\mapsto(a\to b)$ is Heyting's
implication.

\subsubsection{Dummy implicative structures}
Unlike Heyting's implication, the implication of an implicative
structure~$\A$ is in general not determined by the ordering of~$\A$, and
several implicative structures can be defined upon the very same
complete lattice structure:
\begin{example}[Dummy implicative structures]\label{ex:DummyImpStruct}
  Let~$(L,{\cle})$ be a complete lattice.
  There are at least two distinct ways to define a dummy implication
  $a\to b$ on~$L$ that fulfills the axioms (1) and (2) of
  Def.~\ref{d:ImpStruct}:
  \begin{enumerate}[(99)]
  \item[(1)] Put $(a\to b):=b$\quad for all $a,b\in L$.
  \item[(2)] Put $(a\to b):=\top$\quad for all $a,b\in L$.
  \end{enumerate}
\end{example}
Each of these two definitions induces an implicative structure on the
top of the complete lattice $(L,{\cle})$.
From the point of view of logic, these two examples are definitely
meaningless, but they will be useful as a source of counter-examples.

\subsubsection{Quasi-implicative structures induced by partial
  applicative structures}\label{sss:ImpStructRealizJ}
Another important source of examples is given by the structures
underlying intuitionistic realizability~\cite{Oos08}.
Recall that a \emph{partial applicative structure} (PAS) is a
nonempty set~$P$ equipped with a partial binary operation
$({\cdot}):P\times P\pto P$, called \emph{application}.
Such an operation naturally induces a (total) binary operation
$(a,b)\mapsto(a\to b)$ on the subsets of~$P$, called \emph{Kleene's
  implication}, that is defined for all $a,b\subseteq P$ by:
$$a\to b~:=~\{z\in P:\forall x\,{\in}\,a,~z\cdot x\defin b\}$$
(where $z\cdot x\defin b$ means that $z\cdot x$ is defined and
belongs to~$b$).
We easily check that:
\begin{fact}\label{f:PAS}
  Given a partial applicative structure $(P,{\,\cdot\,})$:
  \begin{enumerate}[(99)]
  \item[(1)] The complete lattice $(\Pow(P),{\subseteq})$ equipped
    with Kleene's implication $a\to b$ is a quasi-implicative
    structure (in the sense of Remark~\ref{r:ImpStruct}~(2)).
  \item[(2)] The quasi-implicative structure
    $(\Pow(P),{\subseteq},{\to})$ is an implicative structure if and
    only if the underlying operation of application
    $(x,y)\mapsto x\cdot y$ is total.
  \end{enumerate}
\end{fact}
We shall come back to this example in Section~\ref{sss:CaseRealizJ}.

\medbreak
A variant of the above construction consists to replace the subsets
of~$P$ by the \emph{partial equivalence relations} (PER) over~$P$,
that is, by the binary relations on~$P$ that are both symmetric and
transitive---but not reflexive in general.
The set of partial equivalence relations over~$P$, written $\PER(P)$,
is clearly closed under arbitrary intersection (in the sense of
relations), so that the poset $(\PER(P),{\subseteq})$ is a complete
meet-semilattice.
Kleene's implication naturally extends to partial equivalence
relations, by associating to all $a,b\in\PER(P)$ the relation
$(a\to_2b)\in\PER(P)$ defined by:
$$a\to_2b~:=~\{(z_1,z_2)\in P^2:
\forall(x_1,x_2)\,{\in}\,a,~
(z_1\cdot x_1,z_2\cdot x_2)\defin b\}\,.$$
Again:
\begin{fact}
  Given a partial applicative structure $(P,{\,\cdot\,})$:
  \begin{enumerate}[(99)]
  \item[(1)] The complete lattice $(\PER(P),{\subseteq})$ equipped
    with Kleene's implication $a\to_2b$ is a quasi-implicative structure
    (in the sense of Remark~\ref{r:ImpStruct}~(2)).
  \item[(2)] The quasi-implicative structure
    $(\PER(P),{\subseteq},{\to_2})$ is an implicative structure if and
    only if the underlying operation of application
    $(x,y)\mapsto x\cdot y$ is total.
  \end{enumerate}
\end{fact}

\begin{remark}
  The reader is invited to check that the last two examples of
  (quasi-) implicative structures fulfill the following additional
  axiom
  $$\biggl(\bigjoin_{a\in A}\!\!a\biggr)\to b~=~
  \bigmeet_{\!\!a\in A\!\!}(a\to b)
  \eqno(\text{for all}~A\subseteq\A~\text{and}~b\in\A)$$
  In what follows, we shall see that this axiom---that already holds
  in complete Heyting algebras---is characteristic from the
  implicative structures coming from intuitionistic realizability or
  from (intuitionistic or classical) forcing.
  (On the other hand, this axiom does not hold in the implicative
  structures coming from classical realizability, except in the
  degenerate case of forcing.)
  We shall come back to this point in
  Section~\ref{ss:ExJoin}.
\end{remark}

\subsubsection{Quasi-implicative structures of reducibility candidates}
\label{sss:QImpRedCandidates}

Other examples of quasi-implicative structures are given by the various
notions of \emph{reducibility
  candidates}~\cite{Tai67,Gir89,WerPhD,Par97} that are used to prove
strong normalization.
Let us consider for instance the case of Tait's saturated
sets~\cite{Tai67}.

Recall that a set~$S$ of (possibly open) $\lambda$-terms is
\emph{saturated} (in the sense of Tait) when it fulfills the following
three criteria:
\begin{enumerate}[$(iiii)$]
\item[$(i)$] $S\subseteq\SN$, where $\SN$ is the set of all
  strongly normalizing terms.
\item[$(ii)$] If $x$ is a variable and if $u_1,\ldots,u_n\in\SN$,
  then $xu_1\cdots u_n\in S$.
\item[$(iii)$] If $t\{x:=u_0\}u_1\cdots u_n\in S$ and $u_0\in\SN$,
  then $(\Lam{x}{t})u_0u_1\cdots u_n\in S$.
\end{enumerate}
The set of all saturated sets, written $\SAT$, is closed under
Kleene's implication, in the sense that for all $S,T\in\SAT$ one has
$S\to T=\{t:\forall u\,{\in}\,S,~tu\in T\}\in\SAT$.
Again:
\begin{fact}
  The triple $(\SAT,{\subseteq},{\to})$ is a quasi-implicative
  structure.
\end{fact}

The reader is invited to check that the same holds if we replace
Tait's saturated sets by other notions of reducibility candidates,
such as Girard's reducibility candidates~\cite{Gir89} or Parigot's
reducibility candidates~\cite{Par97}.
Let us mention that in each case, we only get a
\emph{quasi-}implicative structure, in which we have
$(\top\to\top)\neq\top$.
The reason is that full implicative structures (which come with the
equation $(\top\to\top)=\top$) are actually expressive enough to
interpret the full $\lambda$-calculus (see Section~\ref{ss:AppLam}),
so that they are incompatible with the notion of (weak or strong)
normalization.

\subsubsection{Implicative structures of classical
  realizability}\label{sss:ImpStructRealizK}

The final example---which is the main motivation of this work---is
given by classical realizability, as introduced by
Krivine~\cite{Kri01,Kri03,Kri09,Kri11,Kri12}.
Basically, classical realizability takes place in a structure of the
form $(\Lambda,\Pi,{\,\cdot\,},\Bot)$ where:
\begin{itemize}
\item[$\bullet$] $\Lambda$ is a set whose elements are called
  \emph{terms}, or \emph{realizers};
\item[$\bullet$] $\Pi$ is a set whose elements are called
  \emph{stacks}, or \emph{counter-realizers};
\item[$\bullet$] $({\,\cdot\,}):\Lambda\times\Pi\to\Pi$ is a binary
  operation for \emph{pushing} a term onto a stack;
\item[$\bullet$] $\Bot\subseteq\Lambda\times\Pi$ is a binary relation
  between $\Lambda\times\Pi$, called the \emph{pole}.
\end{itemize}
(Krivine's classical realizability structures actually contain many
other ingredients---cf Section~\ref{sss:CaseRealizK}---that we do not
need for now.)
From such a quadruple $(\Lambda,\Pi,{\,\cdot\,},\Bot)$, we let:
\begin{itemize}
\item[$\bullet$] $\A~:=~\Pow(\Pi)$;
\item[$\bullet$] $a\cle b~:\liff~a\supseteq b$\hfill
  (for all $a,b\in\A$)
\item[$\bullet$] $a\to b~:=~a^{\Bot}\cdot b~=~
  \{t\cdot\pi:t\in a^{\Bot},~\pi\in b\}$\hfill
  (for all $a,b\in\A$)
\end{itemize}
writing
$a^{\Bot}:=\{t\in\Lambda:\forall\pi\in a,~(t,\pi)\in\Bot\}
\in\Pow(\Lambda)$ the \emph{orthogonal} of the set $a\in\Pow(\Pi)$
w.r.t.\ the pole $\Bot\subseteq\Lambda\times\Pi$.
Again, it is easy to check that:
\begin{fact}\label{f:KrivineImpAlg}
  The triple $(\A,{\cle},{\to})$ is an implicative structure.
\end{fact}

\begin{remark}
  The reader is invited to check that Krivine's implication $a\to
  b=a^{\Bot}\cdot b$ fulfills the two additional axioms
  $$\biggl(\bigmeet_{a\in A}\!\!a\biggr)\to b~=~
  \bigjoin_{\!\!a\in A\!\!}(a\to b)\qquad\text{and}\qquad
  a\to\biggl(\bigjoin_{b\in B}\!\!b\biggr)~=~
  \bigjoin_{b\in B}(a\to b)$$
  for all $a,b\in\A$, $A,B\subseteq\A$, $A,B\neq\varnothing$.
  It is worth to notice that these extra properties are almost
  never used in classical realizability, thus confirming that only
  the properties of meets really matter in such a structure.
\end{remark}

We shall come back to this example in Section~\ref{sss:CaseRealizK}.

\subsection{Viewing truth values as generalized realizers: a
  manifesto}\label{ss:Manifesto}

Intuitively, an implicative structure $(\A,{\cle},{\to})$ represents a
semantic type system in which the ordering $a\cle b$ expresses the
relation of subtyping, whereas the operation $a\to b$ represents the
arrow type construction.
From the point of view of logic, it is convenient to think of the
elements of~$\A$ as truth values according to some notion of
realizability, that is: as sets of realizers enjoying particular
closure properties.

Following this intuition, we can always view an actual realizer~$t$ as
a truth value, namely: as the smallest truth value that contains~$t$.
This truth value, written $[t]$ and called the \emph{principal type}
of the realizer~$t$, is naturally defined as the meet of all truth
value containing~$t$ as an element.
Through the correspondence $t\mapsto[t]$%
\footnote{Note that this correspondence automatically identifies
  realizers that have the same principal type.
  But since such realizers are clearly interchangeable in the `logic'
  of~$\A$, this identification is harmless.},
the membership relation $t\in a$ rephrases in term of subtyping as
$[t]\cle a$, so that we can actually manipulate realizers as if they
were truth values.

But the distinctive feature of implicative structures is that they
allow us to proceed the other way around.
That is: to manipulate \emph{all} truth values as if they were
realizers.
Technically, this is due to the fact that the two fundamental
operations of the $\lambda$-calculus---application and
$\lambda$-abstraction---can be lifted to the level of truth values
(Section~\ref{ss:AppLam}).
Of course, such a possibility definitely blurs the distinction between
the particular truth values that represent actual realizers (the
principal types) and the other ones.
So that the framework of implicative structures actually leads us to
perform a surprising identification, between the notion of truth value
and the notion of realizer, now using the latter notion in a
generalized sense.

Conceptually, this identification relies on the idea that every
element $a\in\A$ may also be viewed as a generalized realizer, namely:
as the realizer whose principal type is $a$ itself (by convention).
In this way, the element~$a$, when viewed as a generalized realizer,
is not only a realizer of~$a$, but it is more generally a realizer of
any truth value~$b$ such that $a\cle b$.
Of course, there is something puzzling in the idea that truth values
are their own (generalized) realizers, since this implies that any
truth value is realized, at least by itself.
In particular, the bottom truth value $\bot\in\A$, when viewed as a
generalized realizer, is so strong that it actually realizes any truth
value.
But this paradox only illustrates another aspect of implicative
structures, which is that they do not come with an absolute criterion
of consistency.
To introduce such a `criterion of consistency', we shall need to
introduce the notion of \emph{separator} (Section~\ref{s:Separation}),
which plays the very same role as the notion of filter in Heyting
(or Boolean) algebras.

Due to the identification between truth values and (generalized)
realizers, the partial ordering $a\cle b$ can be given different
meanings depending on whether we consider the elements~$a$ and~$b$ as
truth values or as generalized realizers.
For instance, if we think of~$a$ and~$b$ both as truth values, then
the ordering $a\cle b$ is simply the relation of subtyping.
And if we think of~$a$ as a generalized realizer and of~$b$ as a truth
value, then the relation $a\cle b$ is nothing but the realizability
relation (`$a$ realizes~$b$').
But if we now think of both elements~$a$ and~$b$ as generalized
realizers, then the relation $a\cle b$ means that the (generalized)
realizer~$a$ is at least as powerful as~$b$, in the sense that~$a$
realizes any truth value~$c$ that is realized by~$b$.
In forcing, we would express it by saying that~$a$ is a \emph{stronger
condition} than~$b$.
And in domain theory, we would naturally say that~$a$ is \emph{more
  defined} than~$b$, which we would write $a\sqsupseteq b$.

The latter example is important, since it shows that when thinking of
the elements of~$\A$ as generalized realizers rather than as truth
values, then the reverse ordering $a\cge b$ is conceptually similar to
the definitional ordering in the sense of Scott.
Note that this point of view is consistent with the fact that the
theory of implicative structures (see Def.~\ref{d:ImpStruct} and
Remark~\ref{r:ImpStruct}~(1)) is built around meets, that precisely
correspond to joins from the point of view of definitional
(i.e.\ Scott) ordering.
In what follows, we shall refer to the relation $a\cle b$ as the
\emph{logical ordering}, whereas the symmetric relation $b\cge a$
(which we shall sometimes write $b\sqsubseteq a$) will be called the
\emph{definitional ordering}.

Using these intuitions as guidelines, it is now easy to lift all the
constructions of the $\lambda$-calculus to the level of truth values
in an arbitrary implicative structure.

\subsection{Interpreting $\lambda$-terms}
\label{ss:AppLam}

From now on, $\A=(\A,{\cle},{\to})$ denotes an arbitrary implicative
structure.

\begin{definition}[Application]\label{d:App}
  Given two points $a,b\in\A$, we call the
  \emph{application of~$a$ to~$b$} and write $ab$ the element of~$\A$
  that is defined by
  $$ab~:=~\bigmeet\bigl\{c\in\A:a\cle(b\to c)\bigr\}\,.$$
  As usual, we write $ab_1b_2\cdots b_n:=((ab_1)b_2)\cdots b_n$
  for all $a,b_1,b_2,\ldots,b_n\in\A$.
\end{definition}

Thinking in terms of definitional ordering rather than in terms of
logical ordering, this definition expresses that $ab$ is the join of
all $c\in\A$ such that the implication $b\to c$ (which is analogous to
a \emph{step function}) is a lower approximation of~$a$:
$$ab~:=~\bigsqcup\bigl\{c\in\A:(b\to c)\sqsubseteq a\bigr\}\,.$$

\begin{proposition}[Properties of application]\label{p:PropApp}
  For all $a,a',b,b'\in\A$:
  \begin{enumerate}[(99)]
  \item[(1)] If $a\cle a'$ and $b\cle b'$, then $ab\cle a'b'$\hfill
    (Monotonicity)
  \item[(2)] $(a\to b)a\cle b$\hfill($\beta$-reduction)
  \item[(3)] $a\cle(b\to ab)$\hfill($\eta$-expansion)
  \item[(4)] $ab=\min\bigl\{c\in\A:a\cle(b\to c)\bigr\}$\hfill
    (Minimum)
  \item[(5)] $ab\cle c$\quad iff\quad $a\cle(b\to c)$\hfill(Adjunction)
  \end{enumerate}
\end{proposition}

\begin{proof}
  For all $a,b\in\A$, we write $U_{a,b}=\{c\in\A:a\cle(b\to c)\}$,
  so that $ab:=\bigmeet U_{a,b}$.
  (The set $U_{a,b}$ is upwards closed, from the variance of
  implication.)
  \par\noindent
  (1)~~If $a\cle a'$ and $b\cle b'$, then $U_{a',b'}\subseteq U_{a,b}$
  (from the variance of implication), hence we get
  $ab=\bigmeet U_{a,b}\cle\bigmeet U_{a',b'}=a'b'$. 
  \par\noindent
  (2)~~It is clear that $b\in U_{a\to b,a}$, hence
  $(a\to b)a=\bigmeet U_{a\to b,a}\cle b$.
  \par\noindent
  (3)~~We have $(b\to ab)=(b\to\bigmeet{U_{a,b}})
  =\bigmeet_{c\in U_{a,b}}(b\to c)\cge a$, from the def. of
  $U_{a,b}$.
  \par\noindent
  (4)~~From (3), it is clear that $ab\in U_{a,b}$, hence
  $ab=\min(U_{a,b})$.
  \par\noindent
  (5)~~Assuming that $ab\cle c$, we get
  $a\cle(b\to ab)\cle(b\to c)$ from (3).
  Conversely, assuming that $a\cle(b\to c)$, we have
  $c\in U_{a,b}$ and thus $ab=\bigmeet U_{a,b}\cle c$.
\end{proof}

\begin{corollary}[Application in a complete Heyting algebra]%
  \label{c:HeytingApp}
  In a complete Heyting algebra $(H,{\cle},{\to})$, application is
  characterized by $ab=a\meet b$ for all $a,b\in H$.
\end{corollary}

\begin{proof}
  For all $c\in\A$, we have $ab\cle c$ iff $a\cle(b\to c)$
  by Prop.~\ref{p:PropApp}~(5).
  But from Heyting's adjunction, we also have $a\cle(b\to c)$
  iff $a\meet b\cle c$.
  Hence $ab\cle c$ iff $a\meet b\cle c$ for all $c\in\A$,
  and thus $ab=a\meet b$.
\end{proof}

\begin{corollary}[Application in a total applicative structure]%
  \label{c:TotalApp}
  In the implicative structure $(\Pow(P),{\subseteq},{\to})$
  induced by a total applicative structure $(P,{\,\cdot\,})$
  (cf Fact~\ref{f:PAS} p.~\pageref{f:PAS}), application is
  characterized by $ab=\{x\cdot y:x\in a,~y\in b\}$ for all
  $a,b\in\Pow(P)$.
\end{corollary}

\begin{proof}
  Let $a\cdot b=\{x\cdot y:x\in a,~y\in b\}$.
  It is clear that for all $c\in\Pow(P)$, we have
  $a\cdot b\subseteq c$ iff $a\subseteq(b\to c)$.
  Therefore: $a\cdot b=ab$, by adjunction.
\end{proof}

\begin{definition}[Abstraction]\label{d:Abs}
  Given an arbitrary function $f:\A\to\A$, we write
  $\slambda{f}$ the element of~$\A$ defined by:
  $$\slambda{f}~:=~\bigmeet_{a\in\A}(a\to f(a))\,.$$
\end{definition}
(Note that we do not assume that the function~$f$ is monotonic.)

Again, if we think in terms of definitional ordering rather than in
terms of logical ordering, then it is clear that this definition
expresses that $\slambda{f}$ is the join of all the step functions of
the form $a\to f(a)$, where~$a\in\A$:
$$\slambda{f}~:=~\bigsqcup_{a\in\A}(a\to f(a))\,.$$

\begin{proposition}[Properties of abstraction]\label{p:PropLam}
  For all $f,g:\A\to\A$ and $a\in\A$:
  \begin{enumerate}[(99)]
  \item[(1)] If $f(a)\cle g(a)$ for all $a\in\A$, then
    $\slambda{f}\cle\slambda{g}$\hfill(Monotonicity)
  \item[(2)] $(\slambda{f})a\cle f(a)$\hfill ($\beta$-reduction)
  \item[(3)] $a\cle\slambda(b\mapsto ab)$\hfill ($\eta$-expansion)
  \end{enumerate}
\end{proposition}

\begin{proof}
  (1)~~Obvious from the variance of implication.\par\noindent
  (2)~~From the definition of~$\slambda{f}$, we have
  $\slambda{f}\cle(a\to f(a))$.
  Applying Prop.~\ref{p:PropApp}~(5), we get
  $(\slambda{f})a\cle f(a)$.\par\noindent
  (3)~~Follows from Prop.~\ref{p:PropApp}~(3), taking the
  meet for all $b\in\A$.
\end{proof}

We call a \emph{$\lambda$-term with parameters in~$\A$} any
$\lambda$-term (possibly) enriched with constants taken in the
set~$\A$---the `parameters'.
Such enriched $\lambda$-terms are equipped with the usual notions of
$\beta$- and $\eta$-reduction, considering parameters as inert
constants.

To every closed $\lambda$-term~$t$ with parameters in~$\A$, we
associate an element of~$\A$, written $t^{\A}$ and defined by
induction on the size of~$t$ by:
$$\begin{array}{r@{~~}c@{~~}l}
  a^{\A} &:=& a \\
  (tu)^{\A} &:=& t^{\A}u^{\A} \\
  (\Lam{x}{t})^{\A} &:=&
  \slambda(a\mapsto(t\{x:=a\})^{\A})\\
\end{array}\eqno\begin{tabular}{r@{}}
(if $a\in\A$)\\
(application in~$\A$)\\
(abstraction in~$\A$)\\
\end{tabular}$$

\begin{proposition}[Monotonicity of substitution]\label{p:SubstMono}
  For each $\lambda$-term~$t$ with free variables~$x_1,\ldots,x_k$
  and for all parameters $a_1\cle a'_1$, \ldots, $a_k\cle a'_k$, we
  have:
  $$(t\{x_1:=a_1,\ldots,x_k:=a_k\})^{\A}~\cle~
  (t\{x_1:=a'_1,\ldots,x_k:=a'_k\})^{\A}$$
  (where $t\{x_1:=a_1,\ldots,x_k:=a_k\}$ denotes a simultaneous
  substitution).
\end{proposition}

\begin{proof}
  By induction on~$t$, using Prop.~\ref{p:PropApp}~(1) and
  Prop.~\ref{p:PropLam}~(1).
\end{proof}

\begin{proposition}[$\beta$ and $\eta$]\label{p:BetaEta}
  For all closed $\lambda$-terms~$t$ and~$u$ with parameters
  in~$\A$:
  \begin{enumerate}[(99)]
  \item[(1)] If $t\redd_{\beta}u$, then $t^{\A}\cle u^{\A}$
  \item[(2)] If $t\redd_{\eta}u$, then $t^{\A}\cge u^{\A}$
  \end{enumerate}
\end{proposition}

\begin{proof}
  Obvious from Prop.~\ref{p:PropLam}~(2), (3) and
  Prop.~\ref{p:SubstMono}.
\end{proof}

\begin{remark}
  It is important to observe that an implicative structure is in
  general \emph{not} a denotational model of the $\lambda$-calculus,
  since the inequalities of Prop.~\ref{p:BetaEta} are in general not
  equalities, as shown in Example~\ref{ex:DummyImpStruct2} below.
  Let us recall that in a denotational model~$D$ of the
  $\lambda$-calculus (where $t=_{\beta\eta}u$ implies $t^D=u^D$), the
  interpretation function $t\mapsto t^D$ is either trivial, or
  injective on $\beta\eta$-normal forms.
  This is no longer the case in implicative structures, where some
  $\beta\eta$-normal terms may collapse, while others do not.
  We shall come back to this problem in Section~\ref{ss:Consistency}.
\end{remark}

\begin{example}[Dummy implicative structure]\label{ex:DummyImpStruct2}
  Let us consider the dummy implicative structure (cf
  Example~\ref{ex:DummyImpStruct}~(2)) constructed on the top of a
  complete lattice $(L,{\cle})$ by putting $a\to b:=\top$ for all
  $a,b\in\A$.
  In this structure, we observe that:
  \begin{itemize}
  \item[$\bullet$]
    $ab=\bigmeet\{c\in\A:a\cle(b\to c)\}=\bigmeet\A=\bot$\hfill
    for all $a,b\in\A$;
  \item[$\bullet$]
    $\slambda{f}=\bigmeet_{a\in\A}(a\to f(a))=\top$\hfill
    for all functions $f:\A\to\A$.
  \end{itemize}
  So that for any closed $\lambda$-term~$t$, we immediately get:
  $$t^{\A}~=~\begin{cases}
    \top&\text{if}~t~\text{is an abstraction}\\
    \bot&\text{if}~t~\text{is an application}\\
  \end{cases}$$
  (The reader is invited to check that the above characterization is
  consistent with the inequalities of Prop.~\ref{p:BetaEta}.)
  In particular, letting $\mathbf{I}:=\Lam{x}{x}$, we observe that:
  \begin{itemize}
  \item[$\bullet$] $\mathbf{I}\,\mathbf{I}\to_{\beta}\mathbf{I}$,\quad
    but\quad $(\mathbf{I}\,\mathbf{I})^{\A}\,({=}\,\bot)~\neq~
    \mathbf{I}^{\A}\,({=}\,\top)$;
  \item[$\bullet$] $\Lam{x}{\mathbf{I}\,\mathbf{I}\,x}\to_{\eta}
    \mathbf{I}\,\mathbf{I}$,\quad but\quad
    $(\Lam{x}{\mathbf{I}\,\mathbf{I}\,x})^{\A}\,({=}\,\top)~\neq~
    (\mathbf{I}\,\mathbf{I})^{\A}\,({=}\,\bot)$.
  \end{itemize}
\end{example}

\begin{proposition}[$\lambda$-terms in a complete Heyting
    algebra]\label{p:HeytingLam}
  If $(\A,{\cle},{\to})$ is a complete Heyting algebra, then for
  all (pure) $\lambda$-terms with free variables $x_1,\ldots,x_k$ and
  for all parameters $a_1,\ldots,a_k\in\A$, we have:
  $$(t\{x_1:=a_1,\ldots,x_k:=a_k\})^{\A}~\cge~
  a_1\meet\cdots\meet a_k\,.$$
  In particular, for all closed $\lambda$-terms~$t$, we have:\quad
  $t^{\A}=\top$.
\end{proposition}

\begin{proof*}
  Let us write $\vec{x}=x_1,\ldots,x_k$ and
  $\vec{a}=a_1,\ldots,a_k$.
  We reason by induction on~$t$, distinguishing the following cases:
  \begin{itemize}
  \item[$\bullet$] $t=x$ (variable).\quad
    This case is obvious.
  \item[$\bullet$] $t=t_1t_2$ (application).\quad
    In this case, we have:
    $$\begin{array}{rcl}
      (t\{\vec{x}:=\vec{a}\})^{\A}
      &=& (t_1\{\vec{x}:=\vec{a}\})^{\A}(t_2\{\vec{x}:=\vec{a}\})^{\A}\\
      &=& (t_1\{\vec{x}:=\vec{a}\})^{\A}\meet
      (t_2\{\vec{x}:=\vec{a}\})^{\A}\\
      &\cge& a_1\meet\cdots\meet a_k \\
    \end{array}\eqno\begin{tabular}{r@{}}
      \\(by Coro.~\ref{c:HeytingApp})\\(by IH)\\
    \end{tabular}$$
  \item[$\bullet$] $t=\Lam{x_0}{t_0}$ (abstraction).\quad
    In this case, we have:
    $$\begin{array}{rcl}
      (t\{\vec{x}:=\vec{a}\})^{\A}
      &=&\ds\bigmeet_{a_0\in\A}
      \bigl(a_0\to(t_0\{x_0:=a_0,\vec{x}:=\vec{a}\})^{\A}\bigr)\\
      &\cge&\ds\bigmeet_{a_0\in\A}
      \bigl(a_0\to a_0\meet a_1\meet\cdots\meet a_k\bigr) \\
      &\cge&a_1\meet\cdots\meet a_k\\
    \end{array}\eqno\begin{tabular}{r@{}}
      \\[3pt](by IH)\\\\
    \end{tabular}$$
    using the relation $b\cle(a\to a\meet b)$
    (for all $a,b\in\A$) in the last inequality.\hfill
    \usebox{\proofbox}
  \end{itemize}
\end{proof*}

\begin{remark}
  The above result is reminiscent from the fact that in forcing
  (in the sense of Kripke or Cohen), all (intuitionistic or classical)
  tautologies are interpreted by the top element (i.e.\ the weakest
  condition).
  This is clearly no longer the case in (intuitionistic or classical)
  realizability, as well as in implicative structures more generally.
\end{remark}

\subsection{Semantic typing}
\label{ss:SemTyping}

Any implicative structure $\A=(\A,{\cle},{\to})$ naturally induces a
\emph{semantic type system} whose types are the elements of~$\A$.

In this framework, a \emph{typing context} is a finite (unordered)
list $\Gamma=x_1:a_1,\ldots,x_n:a_n$, where $x_1,\ldots,x_n$ are
pairwise distinct $\lambda$-variables and where
$a_1,\ldots,a_n\in\A$.
Thinking of the elements of~$\A$ as realizers rather than as types, we
may also view every typing context
$\Gamma=x_1:a_1,\ldots,x_n:a_n$ as the substitution
$\Gamma=x_1:=a_1,\ldots,x_n:=a_n$.

Given a typing context $\Gamma=x_1:a_1,\ldots,x_n:a_n$, we write
$\dom(\Gamma)=\{x_1,\ldots,x_n\}$ its \emph{domain}, and the
concatenation $\Gamma,\Gamma'$ of two typing contexts~$\Gamma$
and~$\Gamma'$ is defined as expected, provided
$\dom(\Gamma)\cap\dom(\Gamma')=\varnothing$.
Given two typing contexts~$\Gamma$ and~$\Gamma'$, we write
$\Gamma'\cle\Gamma$ when for every declaration $(x:a)\in\Gamma$,
there is a type $b\cle a$ such that $(x:b)\in\Gamma'$.
(Note that the relation $\Gamma'\cle\Gamma$ implies that
$\dom(\Gamma')\supseteq\dom(\Gamma)$.)

Given a typing context~$\Gamma$, a $\lambda$-term~$t$ with parameters
in~$\A$ and an element $a\in\A$, we define the (semantic) typing
judgment $\Gamma\vdash t:a$ as the following shorthand:
$$\Gamma\vdash t:a\quad:\liff\quad
\FV(t)\subseteq\dom(\Gamma)~~\text{and}~~
(t[\Gamma])^{\A}\cle a$$
(using~$\Gamma$ as a substitution in the right-hand side inequality).
From this semantic definition of typing, we easily deduce that:
\begin{proposition}[Semantic typing rules]\label{p:SemTypingRules}
  For all typing contexts $\Gamma$, $\Gamma'$, for all
  $\lambda$-terms $t$, $u$ with parameters in~$\A$ and for all
  $a,a',b\in\A$, the following `semantic typing rules' are valid:
  \begin{itemize}
  \item If\quad $(x:a)\in\Gamma$,\quad then\quad
    $\Gamma\vdash x:a$\hfill(Axiom)
  \item $\Gamma\vdash a:a$\hfill(Parameter)
  \item If\quad $\Gamma\vdash t:a$\quad and\quad
    $a\cle a'$,\quad then\quad $\Gamma\vdash t:a'$\hfill(Subsumption)
  \item If\quad $\Gamma'\cle\Gamma$\quad and\quad
    $\Gamma\vdash t:a$,\quad then\quad
    $\Gamma'\vdash t:a$\hfill(Context subsumption)
  \item If\quad $\FV(t)\subseteq\dom(\Gamma)$,\quad then\quad
    $\Gamma\vdash t:\top$\hfill($\top$-intro)
  \item If\quad $\Gamma,x:a\vdash t:b$,\quad then\quad
    $\Gamma\vdash\Lam{x}{t}:a\to b$\hfill(${\to}$-intro)
  \item If\quad $\Gamma\vdash t:a\to b$\quad and\quad
    $\Gamma\vdash u:a$,\quad then\quad
    $\Gamma\vdash tu:b$\hfill(${\to}$-elim)
  \end{itemize}
  Moreover, for every family $(a_i)_{i\in I}$ of elements of~$\A$
  indexed by a set (or a class)~$I$:
  \begin{itemize}
  \item If\quad $\Gamma\vdash t:a_i$\quad
    (for all $i\in I$),\quad then\quad
    $\ds\Gamma\vdash t:\bigmeet_{i\in I}a_i$\hfill(Generalization)
  \end{itemize}
\end{proposition}

\begin{proof}
  \textit{Axiom, Parameter, Subsumption, $\top$-intro}:
  Obvious.
  \smallbreak\noindent
  \textit{Context subsumption}: Follows from
  Prop.~\ref{p:SubstMono} (monotonicity of substitution).
  \smallbreak\noindent
  \textit{$\to$-intro}: Let us assume that
  $\FV(t)\subseteq\dom(\Gamma,x:=a)$ and
  $(t[\Gamma,x:=a])^{\A}\cle b$.
  It is clear that $\FV(\Lam{x}{t})\subseteq\dom(\Gamma)$ and
  $x\notin\dom(\Gamma)$, so that:
  $$\begin{array}{rcl}
    \bigl((\Lam{x}{t})[\Gamma]\bigr)^{\A}
    &=&\ds\bigl(\Lam{x}{t[\Gamma]}\bigr)^{\A}
    ~=~\bigmeet_{\!\!a_0\in\A\!\!}
    \Bigl(a_0\to\bigl(t[\Gamma,x:=a_0]\bigr)^{\A}\Bigr)\\
    \noalign{\smallskip}
    &\cle& a\to\bigl(t[\Gamma,x:=a]\bigr)^{\A}~\cle~a\to b\,.\\
  \end{array}$$
  \textit{$\to$-elim}:  Let us assume that
  $\FV(t),\FV(u)\subseteq\dom(\Gamma)$,
  $(t[\Gamma])^{\A}\!\cle a\to b$ and
  $(u[\Gamma])^{\A}\!\cle a$.
  It is clear that $\FV(tu)\subseteq\dom(\Gamma)$, and
  from Prop.~\ref{p:PropApp}~(2) we get:
  $$\bigl((tu)[\Gamma]\bigr)^{\A}
  ~=~\bigl(t[\Gamma]\bigr)^{\A}\bigl(u[\Gamma]\bigr)^{\A}
  ~\cle~(a\to b)a~\cle~b\,.$$
  \textit{Generalization:} Obvious, by taking the meet.
\end{proof}

\subsection{Some combinators}
\label{ss:Comb}

Let us now consider the following combinators (using Curry's
notation):
$$\begin{array}{r@{~~}c@{~~}l@{\qquad\qquad}r@{~~}c@{~~}l}
  \mathbf{I}&=&\Lam{x}{x} &
  \mathbf{K}&=&\Lam{xy}{x} \\
  \mathbf{B}&=&\Lam{xyz}{x(yz)} &
  \mathbf{W}&=&\Lam{xy}{xyy} \\
  \mathbf{C}&=&\Lam{xyz}{xzy} &
  \mathbf{S}&=&\Lam{xyz}{xz(yz)} \\
\end{array}$$
It is well-known that in any polymorphic type assignment system,
the above $\lambda$-terms can be given the following (principal)
types:
$$\begin{array}{r@{~~}c@{~~}l}
  \mathbf{I}&:&\forall\alpha\,(\alpha\to\alpha)\\
  \mathbf{B}&:&\forall\alpha\,\forall\beta\,\forall\gamma\,
  ((\alpha\to\beta)\to(\gamma\to\alpha)\to\gamma\to\beta)\\
  \mathbf{K}&:&\forall\alpha\,\forall\beta\,
  (\alpha\to\beta\to\alpha)\\
  \mathbf{C}&:&\forall\alpha\,\forall\beta\,\forall\gamma\,
  ((\alpha\to\beta\to\gamma)\to\beta\to\alpha\to\gamma)\\
  \mathbf{W}&:&\forall\alpha\,\forall\beta\,
  ((\alpha\to\alpha\to\beta)\to\alpha\to\beta)\\
  \mathbf{S}&:&\forall\alpha\,\forall\beta\,\forall\gamma\,
  ((\alpha\to\beta\to\gamma)\to(\alpha\to\beta)\to\alpha\to\gamma)\\
\end{array}$$
Turning the above syntactic type judgments into semantic type
judgments (Section~\ref{ss:SemTyping}) using the typing rules of
Prop.~\ref{p:SemTypingRules}, it is clear that in any implicative
structure $\A=(\A,{\cle},{\to})$, we have the following inequalities:
$$\begin{array}{l}
  \ds\mathbf{I}^{\A}~\cle~\bigmeet_{\!\!a\in\A\!\!}(a\to a),\qquad
  \mathbf{K}^{\A}~\cle~\bigmeet_{\!\!\!\!a,b\in\A\!\!\!\!}(a\to b\to a),\\
  \ds\mathbf{S}^{\A}~\cle~\bigmeet_{\!\!\!\!\!\!a,b,c\in\A\!\!\!\!\!\!}
  ((a\to b\to c)\to(a\to b)\to a\to c),\quad\text{etc.}\\
\end{array}$$
A remarkable property of implicative structures is that the above
inequalities are actually equalities, for each one of the six
combinators $\mathbf{I}$, $\mathbf{B}$,  $\mathbf{K}$,  $\mathbf{C}$,
$\mathbf{W}$ and $\mathbf{S}$:
\begin{proposition}\label{p:CombTypes}
  In any implicative structure $(\A,{\cle},{\to})$, we have:
  $$\begin{array}{@{}r@{~~}c@{~~}l@{\qquad}r@{~~}c@{~~}l@{}}
    \mathbf{I}^{\A}&=&\ds\bigmeet_{\!\!a\in\A\!\!}(a\to a)&
    \mathbf{B}^{\A}&=&\ds\bigmeet_{\!\!\!\!\!\!a,b,c\in\A\!\!\!\!\!\!}
    ((a\to b)\to(c\to a)\to c\to b)\\
    \mathbf{K}^{\A}&=&\ds\bigmeet_{\!\!\!\!a,b\in\A\!\!\!\!}(a\to b\to a)&
    \mathbf{C}^{\A}&=&\ds\bigmeet_{\!\!\!\!\!\!a,b,c\in\A\!\!\!\!\!\!}
    ((a\to b\to c)\to b\to a\to c)\\
    \mathbf{W}^{\A}&=&\ds\bigmeet_{\!\!\!\!a,b\in\A\!\!\!\!}
    ((a\to a\to b)\to a\to b)&
    \mathbf{S}^{\A}&=&\ds\bigmeet_{\!\!\!\!\!\!a,b,c\in\A\!\!\!\!\!\!}
    ((a\to b\to c)\to b\to a\to c)\\
  \end{array}$$
\end{proposition}

\begin{proof}
  Indeed, we have:
  \begin{itemize}
  \item[$\bullet$] $\ds\mathbf{I}^{\A}=(\Lam{x}{x})^{\A}=
    \bigmeet_{\!\!a\in\A\!\!}(a\to a)$\hfill(by definition)
  \item[$\bullet$] $\ds\mathbf{K}^{\A}=(\Lam{xy}{x})^{\A}=
    \bigmeet_{\!\!a\in\A\!\!}\biggl(a\to
    \bigmeet_{\!\!b\in\A\!\!}(b\to a)\biggr)
    =\bigmeet_{\!\!\!\!a,b\in\A\!\!\!\!}(a\to b\to a)$\hfill
    (by axiom~(2))
  \item[$\bullet$] By semantic typing, it is clear that:\medbreak
    $\ds\mathbf{S}^{\A}~=~\bigl(\Lam{xyz}{xz(yz)}\bigr)^{\A}
    ~\cle~\bigmeet_{\!\!\!\!\!\!a,b,c\in\A\!\!\!\!\!\!}
    ((a\to b\to c)\to(a\to b)\to a\to c)$.\medbreak
    Conversely, we have:\medbreak
    $\begin{array}[t]{c@{~~}l}
      &\ds\bigmeet_{\!\!\!\!\!\!a,b,c\in\A\!\!\!\!\!\!}
      ((a\to b\to c)\to(a\to b)\to a\to c)\\
      \cle&\ds\bigmeet_{\!\!\!\!\!\!a,d,e\in\A\!\!\!\!\!\!}
      \bigl((a\to ea\to da(ea))\to(a\to ea)\to a\to da(ea)\bigr)\\
      \cle&\ds\bigmeet_{\!\!\!\!\!\!a,d,e\in\A\!\!\!\!\!\!}
      \bigl((a\to da)\to e\to a\to da(ea)\bigr)\\
      \cle&\ds\bigmeet_{\!\!\!\!\!\!a,d,e\in\A\!\!\!\!\!\!}
      (d\to e\to a\to da(ea))\\[-3pt]
      &\ds{=}~~\bigmeet_{\!\!d\in\A\!\!}\biggl(d\to
      \bigmeet_{\!\!e\in\A\!\!}\biggl(e\to
      \bigmeet_{\!\!a\in\A\!\!}\bigl(a\to da(ea)\bigr)\biggr)\biggr)
      ~=~\bigl(\Lam{xyz}{xz(yz)}\bigr)^{\A}~=~\mathbf{S}^{\A}\\
    \end{array}$\hfill\hskip-50mm\begin{tabular}[t]{r@{}}
      \\[12.5pt]\\[12.5pt]
      (using Prop.~\ref{p:PropApp}~(3) twice)\\[12.5pt]
      (using Prop.~\ref{p:PropApp}~(3) again)\\[12.5pt]\\
    \end{tabular}\smallbreak
  \item[$\bullet$] The proofs for $\mathbf{B}$, $\mathbf{W}$ and
    $\mathbf{C}$ proceed similarly.
  \end{itemize}
\end{proof}

\begin{remarks}
  (1)~~The above property does not generalize to typable terms that are
  not in $\beta$-normal form.
  For instance, the term
  $\mathbf{I}\mathbf{I}=(\Lam{x}{x})(\Lam{x}{x})$ has the
  principal polymorphic type $\forall\alpha\,(\alpha\to\alpha)$, but
  in the dummy implicative structure used in
  Example~\ref{ex:DummyImpStruct2} (where $a\to b=\top$ for all
  $a,b\in\A$), we have seen that 
  $$\mathbf{I}\mathbf{I}~({=}\,\bot)\quad\neq\quad
  \bigmeet_{a\in\A}(a\to a)\quad({=}~\mathbf{I}=\top)\,.$$
  However, we conjecture that in any implicative structure
  $(\A,{\cle},{\to})$, the interpretation of each closed
  $\lambda$-term in $\beta$-normal form is equal to the interpretation
  of its principal type in a polymorphic type system with binary
  intersections~\cite{CDV80,RRV84}.\par
  (2)~~Combining the definitions of $\mathbf{K}^{\A}$
  and $\mathbf{S}^{\A}$ with Prop.~\ref{p:BetaEta}, it is clear that
  $$\mathbf{K}^{\A}ab~\cle~a\qquad\text{and}\qquad
  \mathbf{S}^{\A}abc~\cle~ac(bc)\eqno(\text{for all}~a,b,c\in\A)$$
  These inequalities actually mean that each implicative structure~$\A$
  is also an Ordered Combinatory Algebra (OCA)~\cite{Oos08}, and even
  an Intuitionistic Ordered Combinatory Algebra (${}^I$OCA) in the
  sense of~\cite{FFGMM17}.
\end{remarks}

\subsubsection{Interpreting call/cc}
Since Griffin's seminal work~\cite{Gri90}, it is well-known that the
control operator~$\cc$ (`call/cc', for: \emph{call with current
  continuation}) can be given the type
$((\alpha\to\beta)\to\alpha)\to\alpha$
that corresponds to Peirce's law.
In classical realizability~\cite{Kri09}, the control operator~$\cc$
(that naturally realizes Peirce's law) is the key ingredient to bring
the full expressiveness of classical logic into the realm of
realizability.

By analogy with Prop.~\ref{p:CombTypes}, it is possible to interpret
the control operator~$\cc$ in any implicative structure
$\A=(\A,{\cle},{\to})$ by \emph{identifying} it with Peirce's law,
thus letting
$$\begin{array}[t]{rrl}
  \cc^{\A}
  &:=&\ds\bigmeet_{\!\!\!\!a,b\in\A\!\!\!\!}(((a\to b)\to a)\to a)\\
  &=&\ds\bigmeet_{\!\!a\in\A\!\!}((\lnot a\to a)\to a)\\
\end{array}\eqno\hskip-30mm\text{(Peirce's law)}$$
where negation is defined by $\lnot a:=(a\to\bot)$ for all $a\in\A$.
(The second equality easily follows from the properties of meets and
from the variance of implication.)

Of course, the fact that it is possible to interpret the control
operator~$\cc$ in any implicative structure does not mean that any
implicative structure is well suited for classical logic, since it may
be the case that $\cc^{\A}=\bot$, as shown in the following example:
\begin{example}[Dummy implicative structure]\label{ex:DummyImpStruct3}
  Let us consider the dummy implicative structure (cf
  Example~\ref{ex:DummyImpStruct}~(1)) constructed on the top of a
  complete lattice $(L,{\cle})$ by putting $a\to b:=b$ for all
  $a,b\in L$.
  In this structure, we have:
  $$\cc^{\A}~=~
  \bigmeet_{\!\!\!\!a,b\in L\!\!\!\!}(((a\to b)\to a)\to a)\\
  ~=~\bigmeet_{\!\!a\in L\!\!}a~=~\bot\,.$$
\end{example}

The interpretation $t\mapsto t^{\A}$ of pure $\lambda$-terms naturally
extends to all $\lambda$-terms containing the constant~$\cc$, by
interpreting the latter by $\cc^{\A}$.
\begin{proposition}[$\cc$ in a complete Heyting
    algebra]\label{p:BooleanLam}
  Let $(\A,{\cle},{\to})$ be a complete Heyting algebra.
  Then the following are equivalent:
  \begin{enumerate}[(99)]
  \item[(1)] $(\A,{\cle},{\to})$ is a (complete) Boolean algebra;
  \item[(2)] $\cc^{\A}=\top$;
  \item[(3)] $t^{\A}=\top$ for all closed $\lambda$-terms with~$\cc$.
  \end{enumerate}
\end{proposition}

\begin{proof}
  Let us assume that $(\A,{\cle},{\to})$ is a complete Heyting
  algebra.\smallbreak\noindent
  $(1)\limp(2)$.\quad
  In the case where $(\A,{\cle},{\to})$ is a Boolean algebra, Peirce's
  law is valid in~$\A$, so that $((\lnot a\to a)\to a)=\top$ for all
  $a\in\A$. Hence $\cc^{\A}=\top$, taking the meet.
  \smallbreak\noindent
  $(1)\limp(3)$.\quad Let us assume that $\cc^{\A}=\top$.
  Given a closed $\lambda$-term~$t$ with~$\cc$, we have
  $t=t_0\{x:=\cc\}$ for some pure $\lambda$-term~$t_0$ such
  that $\FV(t_0)\subseteq\{x\}$.
  From Prop.~\ref{p:HeytingLam}, we thus get
  $t^{\A}=(t_0\{x:=\cc^{\A}\})^{\A}\cge\cc^{\A}=\top$,
  hence $t^{\A}=\top$.
  \smallbreak\noindent
  $(3)\limp(1)$.\quad From $(3)$ it is clear that $\cc^{\A}=\top$,
  hence $((\lnot a\to a)\to a)=\top$ for all $a\in\A$.
  Therefore $(\lnot\lnot a\to a)=((\lnot a\to\bot)\to a)
  \cge((\lnot a\to a)\to a)=\top$, hence $(\lnot\lnot a\to a)=\top$
  for all $a\in\A$, which means that $(\A,{\cle},{\to})$ is a Boolean
  algebra.
\end{proof}

\subsection{The problem of consistency}
\label{ss:Consistency}

Although it is possible to interpret all closed $\lambda$-terms
(and even the control operator~$\cc$) in any implicative structure
$(\A,{\cle},{\to})$, the counter-examples given in
Examples~\ref{ex:DummyImpStruct2} and~\ref{ex:DummyImpStruct3} should
make clear to the reader that not all implicative structures are well
suited to interpret intuitionistic or classical logic.
In what follows, we shall say that:
\begin{definition}[Consistency]\label{d:ConsImpStruct}
  An implicative structure $(\A,{\cle},{\to})$ is:
  \begin{itemize}
  \item\emph{intuitionistically consistent} when
    $t^{\A}\neq\bot$ for all closed $\lambda$-terms;
  \item\emph{classically consistent} when
    $t^{\A}\neq\bot$ for all closed $\lambda$-terms with~$\cc$.
  \end{itemize}
\end{definition}

We have seen that complete Heyting/Boolean algebras are particular
cases of implicative structures.
From Prop.~\ref{p:HeytingLam} and~\ref{p:BooleanLam}, it is clear
that:
\begin{proposition}[Consistency of complete Heyting/Boolean
    algebras]\label{p:ConsHeytBool}
  All non-degenerate complete Heyting (resp.\ Boolean) algebras
  are intuitionistically (resp.\ classically) consistent, as
  implicative structures.
\end{proposition}

\subsubsection{The case of intuitionistic
  realizability}\label{sss:CaseRealizJ}

Let us recall~\cite{Oos08} that:
\begin{definition}[Partial combinatory algebra]\label{d:PCA}
  A \emph{partial combinatory algebra} (or \emph{PCA}, for short)
  is a partial applicative structure $(P,{\,\cdot\,})$
  (Section~\ref{sss:ImpStructRealizJ}) with two elements
  $\mathtt{k},\mathtt{s}\in P$ satisfying the following properties
  for all $x,y,z\in P$:
  \begin{enumerate}[(99)]
  \item[(1)] $(\mathtt{k}\cdot x){\downarrow}$,
    $(\mathtt{s}\cdot x){\downarrow}$ and
    $((\mathtt{s}\cdot x)\cdot y){\downarrow}$;
  \item[(2)] $(\mathtt{k}\cdot x)\cdot y\simeq x$;
  \item[(3)] $((\mathtt{s}\cdot x)\cdot y)\cdot z\simeq
    (x\cdot z)\cdot(y\cdot z)$.
  \end{enumerate}
  (As usual, the symbol  $\simeq$ indicates that either both sides of
  the equation are undefined, or that they are both defined and
  equal.)
\end{definition}

Let $(P,{\,\cdot\,},\mathtt{k},\mathtt{s})$ be a PCA.
In Section~\ref{sss:ImpStructRealizJ}, we have seen (Fact~\ref{f:PAS})
that the underlying partial applicative structure $(P,{\,\cdot\,})$
induces a quasi-implicative structure $(\Pow(P),{\subseteq},{\to})$
based on Kleene's implication.
Since we are only interested here in full implicative structures (in
which $(\top\to\top)=\top$), we shall now assume that the
operation of application $({\cdot}):P^2\to P$ is total, so that the
above axioms on~$\mathtt{k},\mathtt{s}\in P$ simplify to:
$$(\mathtt{k}\cdot x)\cdot y=x\qquad\text{and}\qquad
((\mathtt{s}\cdot x)\cdot y)\cdot z=(x\cdot z)\cdot(y\cdot z)
\eqno(\text{for all}~x,y,z\in P)$$
The quadruple $(P,{\,\cdot\,},\mathtt{k},\mathtt{s})$ is then called a
(total) \emph{combinatory algebra} (\emph{CA}).

We want to show that the implicative structure
$\A=(\Pow(P),{\subseteq},{\to})$ induced by any (total) combinatory
algebra $(P,{\,\cdot\,},\mathtt{k},\mathtt{s})$ is intuitionistically
consistent, thanks to the presence of the combinators $\mathtt{k}$
and~$\mathtt{s}$.
For that, we call a \emph{closed combinatory term} any closed
$\lambda$-term that is either $\mathbf{K}~({=}~\Lam{xy}{x})$,
either $\mathbf{S}~({=}~\Lam{xyz}{xz(yz)})$, or the application
$t_1t_2$ of two closed combinatory terms~$t_1$ and~$t_2$.
Each closed combinatory term~$t$ is naturally interpreted in the
set~$P$ by an element $t^P\in P$ that is recursively defined by:
$$\mathbf{K}^P:=\mathtt{k},\qquad
\mathbf{S}^P:=\mathtt{s}\qquad\text{and}\qquad
\mathbf(t_1t_2)^P:=t_1^P\cdot t_2^P\,.$$
We then easily check that:
\begin{lemma}\label{l:IntPinIntA}
  For each closed combinatory term~$t$, we have:\quad
  $t^P\in t^{\A}$.
\end{lemma}

\begin{proof*}
  By induction on~$t$, distinguishing the following cases:
  \begin{itemize}
  \item[$\bullet$] $t=\mathbf{K}$.\quad
    In this case, we have:
    $$\mathbf{K}^P~=~\mathtt{k}~\in~
    \bigcap_{\!\!\!\!\!a,b\in\Pow(P)\!\!\!\!\!}(a\to b\to a) 
    ~=~\mathbf{K}^{\A}\eqno(\text{by Prop.~\ref{p:CombTypes}})$$
  \item[$\bullet$] $t=\mathbf{S}$.\quad
    In this case, we have:
    $$\mathbf{S}^P~=~\mathtt{s}~\in~
    \bigcap_{\!\!\!\!\!\!\!a,b,c\in\Pow(P)\!\!\!\!\!\!\!}
    ((a\to b\to c)\to(a\to b)\to a\to c)
    ~=~\mathbf{S}^{\A}\eqno(\text{by Prop.~\ref{p:CombTypes}})$$
  \item[$\bullet$] $t=t_1t_2$, where $t_1$, $t_2$ are closed
    combinatory terms.\quad
    By IH, we have $t_1^P\in t_1^{\A}$ and $t_2^P\in t_2^{\A}$, hence
    $t^P=t_1^P\cdot t_2^P\in t_1^{\A}t_2^{\A}=t^{\A}$, by
    Coro.~\ref{c:TotalApp}.\hfill\usebox{\proofbox}
  \end{itemize}
\end{proof*}

From the above observation, we immediately get that:
\begin{proposition}[Consistency]
  The implicative structure $(\Pow(P),{\subseteq},{\to})$ induced
  by any (total) combinatory algebra
  $(P,{\,\cdot\,},\mathtt{k},\mathtt{s})$ is
  intuitionistically consistent.
\end{proposition}

\begin{proof}
  Let $t$ be a closed $\lambda$-term.
  From the theory of the $\lambda$-calculus~\cite[Theorem
    7.3.10(i)]{Bar84}, there is a closed combinatory
  term~$t_0$ such that $t_0\redd_{\beta}t$.
  We have $t_0^P\in t_0^{\A}$ (by Lemma~\ref{l:IntPinIntA}) and
  $t_0^{\A}\subseteq t^{\A}$ (by Prop.~\ref{p:BetaEta}), hence
  $t^{\A}\neq\varnothing~({=}~\bot)$.
\end{proof}

(The implicative structure $(\Pow(P),{\subseteq},{\to})$ is
not classically consistent, in general.)

\subsubsection{The case of classical
  realizability}\label{sss:CaseRealizK}

\begin{definition}[Abstract Krivine Structure]\label{d:AKS}
  An \emph{abstract Krivine structure} (or \emph{AKS}) is any
  structure of the form\quad 
  $\mathcal{K}=(\Lambda,\Pi,@,{\,\cdot\,},\mathtt{k}_{\_},
  \mathtt{K},\mathtt{S},\cc,\mathrm{PL},\Bot)$,\quad where:
  \begin{itemize}
  \item $\Lambda$ and $\Pi$ are nonempty sets, whose elements are
    respectively called the \emph{$\mathcal{K}$-terms} and
    the \emph{$\mathcal{K}$-stacks} of the AKS $\mathcal{K}$;
  \item $@:\Lambda\times\Lambda\to\Lambda$ (`application')
    is an operation that associates to each pair of
    $\mathcal{K}$-terms $t,u\in\Lambda$ a $\mathcal{K}$-term
    $@(t,u)\in\Lambda$, usually written $tu$ (by juxtaposition);
  \item $({\,\cdot\,}):\Lambda\times\Pi\to\Pi$ (`push') is an
    operation that associates to each $\mathcal{K}$-term $t\in\Lambda$
    and to each $\mathcal{K}$-stack $\pi\in\Pi$ a $\mathcal{K}$-stack
    $t\cdot\pi\in\Pi$;
  \item $\mathtt{k}_{\_}:\Pi\to\Lambda$ is a function that turns each
    $\mathcal{K}$-stack $\pi\in\Pi$ into a $\mathcal{K}$-term
    $\texttt{k}_{\pi}\in\Pi$, called the \emph{continuation}
    associated to~$\pi$;
  \item $\mathtt{K},\mathtt{S},\cc\in\Lambda$ are three distinguished
    $\mathcal{K}$-terms;
  \item $\mathrm{PL}\subseteq\Lambda$ is a set of $\mathcal{K}$-terms,
    called the set of \emph{proof-like $\mathcal{K}$-terms}, that
    contains the three $\mathcal{K}$-terms $\mathtt{K}$, $\mathtt{S}$
    and~$\cc$, and that is closed under application;
  \item $\Bot\subseteq\Lambda\times\Pi$ is a binary relation between
    $\mathcal{K}$-terms and $\mathcal{K}$-stacks, called the
    \emph{pole} of the AKS $\mathcal{K}$, that fulfills the following
    axioms
    $$\begin{array}{r@{~~}c@{~~}l@{\quad}l@{\quad}r@{~~}c@{~~}l}
      t&\Bot&u\cdot\pi &\text{implies}& tu&\Bot&\pi \\
      t&\Bot&\pi &\text{implies}&
      \mathtt{K}&\Bot&t\cdot u\cdot\pi\\
      t&\Bot&v\cdot uv\cdot\pi &\text{implies}&
      \mathtt{S}&\Bot&t\cdot u\cdot v\cdot\pi\\
      t&\Bot&\mathtt{k}_{\pi}\cdot\pi &\text{implies}&
      \cc&\Bot&t\cdot\pi \\
      t&\Bot&\pi &\text{implies}&
      \mathtt{k}_{\pi}&\Bot&t\cdot\pi'\\
    \end{array}$$
    for all $t,u,v\in\Lambda$ and $\pi,\pi'\in\Pi$.
  \end{itemize}
\end{definition}

\begin{remarks}
  (1)~~The above closure conditions on the pole
  $\Bot\subseteq\Lambda\times\Pi$ actually express that it is closed
  by \emph{anti-evaluation}, in the sense of the evaluation rules
  $$\begin{array}{r@{~}c@{~}lcr@{~}c@{~}l}
    tu&\star&\pi &\succ& t&\star&u\cdot\pi \\
    \mathtt{K}&\star&t\cdot u\cdot\pi &\succ& t&\star&\pi \\
    \mathtt{S}&\star&t\cdot u\cdot v\cdot\pi
    &\succ& t&\star&v\cdot uv\cdot\pi \\
    \cc&\star&t\cdot\pi&\succ&t&\star&\mathtt{k}_{\pi}\cdot\pi\\
    \mathtt{k}_{\pi}&\star&t\cdot\pi'&\succ&t&\star&\pi\\
  \end{array}$$
  (writing $t\star\pi=(t,\pi)$ the \emph{process} formed by a
  $\mathcal{K}$-term $t$ and a $\mathcal{K}$-stack~$\pi$).\par
  (2)~~The notion of AKS---which was introduced by
  Streicher~\cite{Str13}---is very close to the notion of
  \emph{realizability structure} such as introduced by
  Krivine~\cite{Kri11}, the main difference being that the latter
  notion introduces more primitive combinators, essentially to mimic
  the evaluation strategy of the $\lambda_c$-calculus~\cite{Kri09}.
  However, in what follows, we shall not need such a level of
  granularity, so that we shall stick to Streicher's definition.
\end{remarks}

In Section~\ref{sss:ImpStructRealizK}, we have seen
(Fact~\ref{f:KrivineImpAlg}) that the quadruple
$(\Lambda,\Pi,{\,\cdot\,},\Bot)$ underlying any abstract Krivine
structure\quad
$\mathcal{K}=(\Lambda,\Pi,@,{\,\cdot\,},\mathtt{k}_{\_},
\mathtt{K},\mathtt{S},\cc,\mathrm{PL},\Bot)$\quad
induces an implicative structure $\A=(\A,{\cle},{\to})$ that is
defined by:
\begin{itemize}
\item[$\bullet$] $\A~:=~\Pow(\Pi)$;
\item[$\bullet$] $a\cle b~:\liff~a\supseteq b$\hfill
  (for all $a,b\in\A$)
\item[$\bullet$] $a\to b~:=~a^{\Bot}\cdot b~=~
  \{t\cdot\pi:t\in a^{\Bot},~\pi\in b\}$\hfill
  (for all $a,b\in\A$)
\end{itemize}
where
$a^{\Bot}:=\{t\in\Lambda:\forall\pi\in a,~(t,\pi)\in\Bot\}
\in\Pow(\Lambda)$ is the orthogonal of the set $a\in\Pow(\Pi)$
w.r.t.\ the pole $\Bot\subseteq\Lambda\times\Pi$.

Note that since the ordering of subtyping $a\cle b$ is defined here as
the relation of \emph{inverse inclusion} $a\supseteq b$ (between two
sets of stacks $a,b\in\Pow(\Pi)$), the smallest element of the induced
implicative structure $\A=(\A,{\cle},{\to})$ is given by $\bot=\Pi$.

\begin{remark}
  In~\cite{Str13}, Streicher only considers sets of stacks
  $a\in\Pow(\Pi)$ such that $a^{\Bot\Bot}=a$, thus working with a
  smaller set of `truth values' $\A'$ given by:
  $$\A'~:=~\Pow_{\Bot}(\Pi)~=~\{a\in\Pow(\Pi):a^{\Bot\Bot}=a\}\,.$$
  Technically, such a restriction requires to alter the interpretation
  of implication, by adding another step of bi-orthogonal closure:
  $$a\to'b~:=~\bigl(a^{\Bot}\cdot b\bigr)^{\Bot\Bot}
  \eqno(\text{for all}~a,b\in\A')$$
  However, the resulting triple $(\A',{\cle},{\to'})$ is in general
  not an implicative structure, since it does not fulfill axiom~(2) of
  Def.~\ref{d:ImpStruct}%
  \footnote{As a consequence, the constructions presented
    in~\cite{Str13,FFGMM17} only fulfill half of the adjunction of
    Prop.~\ref{p:PropApp}~(5), the missing implication being recovered
    only up to a step of $\eta$-expansion, by inserting the
    combinator~$\mathtt{E}=\Lam{xy}{xy}$ appropriately
    (see~\cite{Str13,FFGMM17} for the details).}.
  For this reason, we shall follow Krivine by considering all sets of
  stacks as truth values in what follows. 
\end{remark}

The basic intuition underlying Krivine's realizability is that each
set of $\mathcal{K}$-stacks $a\in\Pow(\Pi)$ represents the set of
\emph{counter-realizers} (or \emph{attackers}) of a particular
formula, whereas its orthogonal $a^{\Bot}\in\Pow(\Lambda)$ represents
the set of \emph{realizers} (or \emph{defenders}) of the same formula%
\footnote{This is why sets of stacks are sometimes called
  \emph{falsity values}, as in~\cite{Miq10,Miq11}.}.
In this setting, the realizability relation is naturally defined by
$$t\Vdash a~~:\liff~~t\in a^{\Bot}
\eqno(\text{for all}~t\in\Lambda,~a\in\A)$$

However, when the pole $\Bot\subseteq\Lambda\times\Pi$ is not empty,
we can observe that:
\begin{fact}
  Given a fixed $(t_0,\pi_0)\in\Bot$, we have\quad
  $\mathtt{k}_{\pi_0}t_0\Vdash a$\quad for all $a\in\A$.
\end{fact}
so that \emph{any} element of the implicative structure is actually
realized by some $\mathcal{K}$-term (which does not even depend on the
considered element of~$\A$).
This is the reason why Krivine introduces an extra parameter, the set
of \emph{proof-like} ($\mathcal{K}$)-terms
$\mathrm{PL}\subseteq\Lambda$, whose elements are (by convention) the
realizers that are considered as valid certificates of the truth of a
formula.
(The terminology `proof-like' comes from the fact that all realizers
that come from actual proofs belong to the subset
$\mathrm{PL}\subseteq\Lambda$.)

Following Krivine, we say that a truth value $a\in\A$ is
\emph{realized} when it is realized by a proof-like term, that is:
$$\begin{array}{r@{\quad}r@{\quad}l}
  a~\text{realized}&:\liff&
  \exists t\in\mathrm{PL},~t\Vdash a\\
  &\liff&a^{\Bot}\cap\mathrm{PL}\neq\varnothing\\
\end{array}$$
More generally, we say that the abstract Krivine structure
$\mathcal{K}=(\Lambda,\Pi,\ldots,\mathrm{PL},\Bot)$ is
\emph{consistent} when the smallest truth value $\bot=\Pi$ is not
realized, that is:
$$\mathcal{K}~\text{consistent}~~:\liff~~
\Pi^{\Bot}\cap\textrm{PL}=\varnothing\,.$$

We now need to check that Krivine's notion of consistency is
consistent with the one that comes with implicative structures
(Def.~\ref{d:ConsImpStruct}).
For that, we call a \emph{closed classical combinatory term} any
closed $\lambda$-term with~$\cc$ that is either
$\mathbf{K}~({=}~\Lam{xy}{x})$, either
$\mathbf{S}~({=}~\Lam{xyz}{xz(yz)})$, either the constant~$\cc$,
or the application $t_1t_2$ of two closed classical combinatory
terms~$t_1$ and~$t_2$.
Each closed classical combinatory term~$t$ is naturally interpreted by
an element $t^{\Lambda}\in\Lambda$ that is recursively defined by:
$$\mathbf{K}^{\Lambda}:=\mathtt{K},\qquad
\mathbf{S}^{\Lambda}:=\mathtt{S},\qquad
\cc^{\Lambda}:=\cc\qquad\text{and}\qquad
\mathbf(t_1t_2)^{\Lambda}:=t_1^{\Lambda}t_2^{\Lambda}\,.$$
From the closure properties of the set $\mathrm{PL}$ of proof-like
terms, it is clear that $t^{\Lambda}\in\mathrm{PL}$ for each closed
classical combinatory term~$t$.
Moreover:
\begin{lemma}\label{l:IntPrealIntA}
  For each closed classical combinatory term~$t$, we have:\quad
  $t^{\Lambda}\Vdash t^{\A}$.
\end{lemma}

\begin{proof*}
  By induction on~$t$, distinguishing the following cases:
  \begin{itemize}
  \item[$\bullet$] $t=\mathbf{K},\mathbf{S},\cc$.\quad
    In this case, combining standard results of classical
    realizability~\cite{Kri11} with the properties of implicative
    structures, we get:
    $$\begin{array}{@{}r@{~~}c@{~~}l@{~~}c@{~~}l@{}}
      \mathbf{K}^{\Lambda}&=&\mathtt{K}&\Vdash&\ds
      \bigmeet_{\!\!\!\!a,b\in\A\!\!\!\!}(a\to b\to a)
      ~=~\mathbf{K}^{\A}\\
      \mathbf{S}^{\Lambda}&=&\mathtt{S}&\Vdash&\ds
      \bigmeet_{\!\!\!\!\!\!a,b,c\in\A\!\!\!\!\!\!}
      ((a\to b\to c)\to(a\to b)\to a\to c)
      ~=~\mathbf{S}^{\A}\\
      \cc^{\Lambda}&=&\cc&\Vdash&\ds
      \bigmeet_{\!\!\!\!a,b\in\A\!\!\!\!}(((a\to b)\to a)\to a)
      ~=~\cc^{\A}\\
    \end{array}\eqno\begin{tabular}{r@{}}
    (by Prop.~\ref{p:CombTypes})\\[12.5pt]
    (by Prop.~\ref{p:CombTypes})\\[12.5pt]
    (by definition)\\[12pt]
    \end{tabular}$$
  \item[$\bullet$] $t=t_1t_2$, where $t_1$, $t_2$ are closed
    classical combinatory terms.\quad
    In this case, we have $t_1^{\Lambda}\Vdash t_1^{\A}$ and
    $t_2^{\Lambda}\Vdash t_2^{\A}$ by IH.
    And since $t_1^{\A}\cle(t_2^{\A}\to t_1^{\A}t_2^{\A})$ (from
    Prop.~\ref{p:PropApp}~(3)), we also have 
    $t_1^{\Lambda}\Vdash t_2^{\A}\to t_1^{\A}t_2^{\A}$ (by subtyping),
    so that we get $t^{\Lambda}=t_1^{\Lambda}t_2^{\Lambda}\Vdash
    t_1^{\A}t_2^{\A}=t^{\A}$ (by modus ponens).\hfill
    \usebox{\proofbox}
  \end{itemize}
\end{proof*}

We can now conclude:
\begin{proposition}
  If an abstract Krivine structure
  $\mathcal{K}=(\Lambda,\Pi,\ldots,\mathrm{PL},\Bot)$ is consistent
  (in the sense that $\Pi^{\Bot}\cap\mathrm{PL}=\varnothing$), then
  the induced implicative structure $\A=(\Pow(\Pi),{\supseteq},{\to})$
  is classically consistent (in the sense of
  Def.~\ref{d:ConsImpStruct}).
\end{proposition}

\begin{proof}
  Let us assume that $\Pi^{\Bot}\cap\mathrm{PL}=\varnothing$.
  Given a closed $\lambda$-term~$t$ with~$\cc$, there exists
  a closed classical combinatory term~$t_0$ such that
  $t_0\redd_{\beta}t$. 
  So that we have $t_0^{\Lambda}\Vdash t_0^{\A}$ (by
  Lemma~\ref{l:IntPrealIntA}) and $t_0^{\A}\cle t^{\A}$ (by
  Prop.~\ref{p:BetaEta}), hence $t_0^{\Lambda}\Vdash t^{\A}$ (by
  subtyping).
  But this implies that $t^{\A}\neq\bot~({=}~\Pi)$,
  since
  $t^{\Lambda}\in(t^{\A})^{\Bot}\cap\mathrm{PL}\neq\varnothing$.
\end{proof}

Note that the converse implication does not hold in general.
The reason is that the criterion of consistency for the considered
abstract Krivine structure depends both on the pole~$\Bot$ and on the
conventional set $\mathrm{PL}$ of proof-like terms.
(In particular, it should be clear to the reader that the larger the
set $\mathrm{PL}$, the stronger the corresponding criterion of
consistency.) 
On the other hand, the construction of the induced implicative
structure $\A=(\Pow(\Pi),{\supseteq},{\to})$ does not depend on the
set $\mathrm{PL}$, so that the criterion of classical consistency of
Def.~\ref{d:ConsImpStruct}---which does not depend on~$\mathrm{PL}$
either---can only be regarded as a minimal criterion of consistency.

In order to reflect more faithfully Krivine's notion of consistency
at the level of the induced implicative structure, it is now time to
introduce the last ingredient of implicative algebras: the notion of
\emph{separator}.

\section{Separation}
\label{s:Separation}

\subsection{Separators and implicative algebras}

Let $\A=(\A,{\cle},{\to})$ be an implicative structure.

\begin{definition}[Separator]\label{d:Separator}
  We call a \emph{separator} of~$\A$ any subset $S\subseteq\A$
  that fulfills the following conditions for all $a,b\in\A$:
  \begin{enumerate}[(99)]
  \item[(1)] If $a\in S$ and $a\cle b$, then $b\in S$\hfill
    ($S$ is upwards closed)
  \item[(2)] $\mathbf{K}^{\A}=(\Lam{xy}{x})^{\A}\in S$ and
    $\mathbf{S}^{\A}=(\Lam{xyz}{xz(yz)})^{\A}\in S$\hfill
    ($S$ contains $\mathbf{K}$ and $\mathbf{S}$)
  \item[(3)] If $(a\to b)\in S$ and $a\in S$, then $b\in S$\hfill
    ($S$ is closed under modus ponens)
  \end{enumerate}
  A separator $S\subseteq\A$ is said to be:
  \begin{itemize}
  \item\emph{consistent} when $\bot\notin S$;
  \item\emph{classical} when $\cc^{\A}\in S$.
  \end{itemize}
\end{definition}

\begin{remark}
  In the presence of condition~(1) (upwards closure), condition~(3)
  (closure under modus ponens) is actually equivalent to:
  \begin{enumerate}[(99)]
  \item[$(3')$] If $a,b\in S$, then $ab\in S$\hfill
    (closure under application)
  \end{enumerate}
\end{remark}

\begin{proof*}
  Let $S\subseteq\A$ be an upwards closed subset of~$\A$.
  \begin{itemize}
  \item[$\bullet$] $(3)\limp(3')$\quad
    Suppose that $a,b\in S$.
    Since $a\cle(b\to ab)$ (from Prop.~\ref{p:PropApp}~(3)), we get
    $(b\to ab)\in S$ by upwards closure, hence
    $ab\in S$ by~(3).
  \item[$\bullet$] $(3')\limp(3)$\quad
    Suppose that $(a\to b),a\in S$.
    By $(3')$ we have $(a\to b)a\in S$, and since
    $(a\to b)a\cle b$ (from Prop.~\ref{p:PropApp}~(2)), we get
    $b\in S$ by upwards closure.\hfill
    \usebox{\proofbox}
  \end{itemize}
\end{proof*}

Intuitively, each separator $S\subseteq\A$ defines a particular
`criterion of truth' within the implicative structure
$\A=(\A,{\cle},{\to})$.
In implicative structures, separators play the very same role as
filters in Heyting algebras, and it is easy to check that:
\begin{proposition}[Separators in a complete Heyting
    algebra]\label{p:HeytingSepar}
  If $\A=(\A,{\cle},{\to})$ is a complete Heyting algebra, then
  a subset $S\subseteq\A$ is a separator (in the sense of implicative
  structures) if and only if~$S$ is a filter (in the sense
  of Heyting algebras).
\end{proposition}

\begin{proof}
  Indeed, when the implicative structure $\A=(\A,{\cle},{\to})$ is a
  complete Heyting algebra, the conditions (1), (2) and $(3')$
  defining separators simplify to:
  \begin{enumerate}[(99')]
  \item[(1)] If $a\in S$ and $a\cle b$, then $b\in S$\hfill
    (upwards closure)
  \item[(2)] $\top~({=}~\mathbf{K}^{\A}\!=\mathbf{S}^{\A})\in S$\hfill
    (from Prop.~\ref{p:HeytingLam})
  \item[$(3')$] If $a,b\in S$, then $a\meet b~({=}~ab)\in S$\hfill
    (from Coro.~\ref{c:HeytingApp})
  \end{enumerate}
  which is precisely the definition of the notion of a filter.
\end{proof}

However, separators are in general \emph{not} filters, since they are
not closed under binary meets (i.e.\ $a\in S$ and $b\in S$ do not
necessarily imply that $a\meet b\in S$).
Actually, one of the key ideas we shall develop in the rest of this
paper is that the difference between (intuitionistic or classical)
realizability and forcing (in the sense of Kripke or Cohen) lies
precisely in the difference between separators and filters.

\begin{proposition}\label{p:SeparLam}
  If $S\subseteq\A$ is a separator, then for all
  $\lambda$-terms~$t$ with free variables $x_1,\ldots,x_n$ and for all
  parameters $a_1,\ldots,a_n\in S$, we have:
  $$(t\{x_1:=a_1,\ldots,x_n:=a_n\})^{\A}\in S\,.$$
  In particular, for all closed $\lambda$-terms $t$, we have
  $t^{\A}\in S$.
\end{proposition}

\begin{proof}
  Let $t$ be a $\lambda$-term with free variables $x_1,\ldots,x_n$,
  and let $a_1,\ldots,a_n$ be parameters taken in~$S$.
  From the theory of the $\lambda$-calculus, there exists a closed
  combinatory term~$t_0$ such that
  $t_0\redd_{\beta}\Lam{x_1\cdots x_n}{t}$.
  It is clear that $t_0^{\A}a_1\cdots a_n\in S$
  from the conditions (2) and $(3')$ on the separator~$S$.
  Moreover, by Prop.~\ref{p:BetaEta} we have
  $$t_0^{\A}a_1\cdots a_n
  ~\cle~(\Lam{x_1\cdots x_n}{t})^{\A}a_1\cdots a_n
  ~\cle~(t\{x_1:=a_1,\ldots,x_n:=a_n\})^{\A}\,,$$
  so that we get $(t\{x_1:=a_1,\ldots,x_n:=a_n\})^{\A}\in S$,
  by upwards closure.
\end{proof}

\begin{definition}[Implicative algebra]\label{d:ImpAlg}
  We call an \emph{implicative algebra} any implicative structure
  $(\A,{\cle},{\to})$ equipped with a separator $S\subseteq\A$.
  An implicative algebra $(\A,{\cle},{\to},S)$ is said to be
  \emph{consistent} (resp.\ \emph{classical}) when the underlying
  separator $S\subseteq\A$ is \emph{consistent}
  (resp.\ \emph{classical}).
\end{definition}

\subsection{Examples}

\subsubsection{Complete Heyting algebras}
\label{sss:ImpAlgHeyting}

We have seen that a complete Heyting algebra $(H,{\cle})$ can be
seen as an implicative structure $(H,{\cle},{\to})$ where implication
is defined by:
$$a\to b~:=~\max\{c\in H:(c\meet a)\cle b\}
\eqno(\text{for all}~a,b\in H)$$
The complete Heyting algebra $(H,{\cle})$ can also be seen as an
implicative \emph{algebra}, by endowing it with the trivial separator
$S=\{\top\}$ (i.e.\ the smallest filter of~$H$).

\subsubsection{Implicative algebras of intuitionistic realizability}
\label{sss:ImpAlgRealizJ}

Let $(P,{\,\cdot\,},\mathtt{k},\mathtt{s})$ be a (total) combinatory
algebra.
In section~\ref{sss:CaseRealizJ}, we have seen that
such a structure induces an implicative structure
$(\Pow(P),{\subseteq}\,{\to})$ whose implication is defined by:
$$a\to b~:=~\{z\in P:\forall x\in a,~z\cdot x\in b\}
\eqno(\text{for all}~a,b\in\Pow(P))$$
The above implicative structure is naturally turned into an
implicative algebra by endowing it with the separator
$S=\Pow(P)\setminus\{\varnothing\}$ formed by all truth values that
contain at least a realizer.
In this case, the separator $S=\Pow(P)\setminus\{\varnothing\}$ is not
only consistent (in the sense of Def.~\ref{d:Separator}), but it is
also a \emph{maximal separator} (see Section~\ref{ss:SeparMax}
below).

\begin{remark}
  In an arbitrary implicative structure $(\A,{\cle},{\to})$, we
  can observe that the subset $\A\setminus\{\bot\}\subset\A$ is in
  general \emph{not} a separator.
  (Counter-example: consider the Boolean algebra with 4~elements.)
  The property that $\A\setminus\{\bot\}$ is a separator is thus a
  specific property of the implicative structures induced by (total)
  combinatory algebras, and the existence of such a separator that is
  trivially consistent explains why there is no need to introduce a
  notion of proof-like term in intuitionistic realizability.
\end{remark}

\subsubsection{Implicative algebras of classical realizability}
\label{sss:ImpAlgRealizK}

Let
$$\mathcal{K}~=~(\Lambda,\Pi,@,{\,\cdot\,},\mathtt{k}_{\_},
\mathtt{K},\mathtt{S},\cc,\mathrm{PL},\Bot)$$
be an abstract Krivine structure (Def.~\ref{d:AKS}
p.~\pageref{d:AKS}).
We have seen (Section~\ref{sss:CaseRealizK}) that such a structure
induces an implicative structure $(\A,{\cle},{\to})$ where:
\begin{itemize}
\item[$\bullet$] $\A~:=~\Pow(\Pi)$;
\item[$\bullet$] $a\cle b~:\liff~a\supseteq b$\hfill
  (for all $a,b\in\A$)
\item[$\bullet$] $a\to b~:=~a^{\Bot}\cdot b~=~
  \{t\cdot\pi:t\in a^{\Bot},~\pi\in b\}$\hfill(for all $a,b\in\A$)
\end{itemize}
Using the set $\mathrm{PL}$ of proof-like terms, we can now turn the
former implicative structure into an implicative algebra
$(\A,{\cle},{\to},S)$, letting:
$$S~:=~\{a\in\A:a^{\Bot}\cap\mathrm{PL}\neq\varnothing\}\,.$$

\begin{proposition}
  The subset
  $S=\{a\in\A:a^{\Bot}\cap\mathrm{PL}\neq\varnothing\}\subseteq\A$
  is a classical separator of the implicative structure
  $(\A,{\cle},{\to})$.
\end{proposition}

\begin{proof*}
  By construction, we have
  $S=\{a\in\A:\exists t\in\mathrm{PL},~t\Vdash a\}$.
  \begin{enumerate}[(99)]
  \item[(1)] Upwards closure: obvious, by subtyping.
  \item[(2)] We have seen in Section~\ref{sss:CaseRealizK} (Proof of
    Lemma~\ref{l:IntPrealIntA}) that
    $\mathtt{K}\Vdash\mathbf{K}^{\A}$,
    $\mathtt{S}\Vdash\mathbf{S}^{\A}$ and
    $\cc\Vdash\cc^{\A}$, and since
    $\mathtt{K},\mathtt{S},\cc\in\mathrm{PL}$, we get
    $\mathbf{K}^{\A},\mathbf{S}^{\A},\cc^{\A}\in S$.
  \item[(3)] Suppose that $(a\to b),a\in S$.
    From the definition of~$S$, we have $t\Vdash a\to b$ and
    $u\Vdash a$ for some $t,u\in\mathrm{PL}$, so that
    $tu\Vdash b$, where $tu\in\mathrm{PL}$.
    Hence $b\in S$.\hfill\usebox{\proofbox}
  \end{enumerate}
\end{proof*}

Moreover, it is obvious that:
\begin{proposition}[Consistency]
  The classical implicative algebra $(\A,{\cle},{\to},S)$ induced
  by the abstract Krivine structure
  $\mathcal{K}=(\Lambda,\Pi,\ldots,\mathrm{PL},\Bot)$ is consistent
  (in the sense of Def.~\ref{d:ImpAlg}) if and only if $\mathcal{K}$
  is consistent (in the sense that
  $\Pi^{\Bot}\cap\mathrm{PL}=\varnothing$).
\end{proposition}

\begin{proof}
  Indeed, we have $\bot\notin S$ iff
  $\bot^{\Bot}\cap\mathrm{PL}=\varnothing$, that is:
  iff $\Pi^{\Bot}\cap\mathrm{PL}=\varnothing$.
\end{proof}

\subsection{Generating separators}
\label{ss:SeparGen}

Let $\A=(\A,{\cle},{\to})$ be an implicative structure.
For each subset $X\subseteq\A$, we write:
\begin{itemize}
\item[$\bullet$] ${\uparrow}X=\{a\in\A:\exists a_0\in X,~a_0\cle a\}$
  the \emph{upwards closure} of $X$ in~$\A$;
\item[$\bullet$] $@(X)$ the \emph{applicative closure of~$X$}, defined
  as the smallest subset of~$\A$ containing~$X$ (as a subset) and
  closed under application;
\item[$\bullet$] $\Lambda(X)$ the \emph{$\lambda$-closure of~$X$},
  formed by all elements $a\in\A$ that can be written
  $a=(t\{x_1:=a_1,\ldots,x_n:=a_n\})^{\A}$ for some pure
  $\lambda$-term~$t$ with free variables $x_1,\ldots,x_n$ and for some
  parameters $a_1,\ldots,a_n\in X$.
\end{itemize}
Note that in general, the sets $@(X)$ and $\Lambda(X)$ are
not upwards closed, but we obviously have the inclusion
$@(X)\subseteq\Lambda(X)$.

\begin{proposition}[Generated separator]\label{p:SeparGen}
  Given any subset $X\subseteq\A$ we have:
  $${\uparrow}\Lambda(X)~=~
  {\uparrow}@(X\cup\{\mathbf{K}^{\A},\mathbf{S}^{\A}\})\,.$$
  By construction, the above set is the smallest separator of~$\A$
  that contains~$X$ as a subset; it is called the \emph{separator
    generated by~$X$}, and written $\Sep(X)$.
\end{proposition}

\begin{proof}
  The inclusion
  ${\uparrow}@(X\cup\{\mathbf{K}^{\A},\mathbf{S}^{\A}\})
  \subseteq{\uparrow}\Lambda(X)$ is obvious, and the converse
  inclusion follows from Prop.~\ref{p:BetaEta} using the fact each
  $\lambda$-term is the $\beta$-contracted of some combinatory term
  constructed from variables, $\mathbf{K}$, $\mathbf{S}$ and
  application.
  The set ${\uparrow}\Lambda(X)$ is clearly a separator (closure under
  application follows from Prop.~\ref{p:PropApp}~(1)), and from
  Prop.~\ref{p:SeparLam}, it is included in any separator
  containing~$X$ as a subset.
\end{proof}

An important property of first-order logic is the deduction lemma,
which states that an implication $\phi\limp\psi$ is provable in a
theory $\mathscr{T}$ if and only if the formula $\psi$ is provable in
the theory $\mathscr{T}+\phi$ that is obtained by
enriching~$\mathscr{T}$ with the axiom~$\phi$.
Viewing separators $S\subseteq\A$ as theories, this naturally suggests
the following semantic counterpart:
\begin{lemma}[Deduction in a separator]
  For each separator $S\subseteq\A$, we have
  $$(a\to b)\in S\qquad\text{iff}\qquad
  b\in\Sep(S\cup\{a\})\eqno(\text{for all}~a,b\in\A)$$
\end{lemma}

\begin{proof}
  Suppose that $(a\to b)\in S$.
  Then $(a\to b)\in\Sep(S\cup\{a\})$ (by inclusion), and
  since $a\in\Sep(S\cup\{a\})$ (by construction), we get
  $b\in\Sep(S\cup\{a\})$ (by modus ponens).
  Conversely, let us suppose that $b\in\Sep(S\cup\{a\})$.
  From the definition of the separator $\Sep(S\cup\{a\})$,
  this means that there are a $\lambda$-term~$t$ with free variables
  $x_1,\ldots,x_n$ and parameters $a_1,\ldots,a_n\in S\cup\{a\}$ such
  that $(t\{x_1:=a_1,\ldots,x_n:=a_n\})^{\A}\cle b$.
  Without loss of generality, we can assume that $a_1=a$ and
  $a_2,\ldots,a_n\in S$ (with $n\ge 1$).
  Letting
  $c:=(\Lam{x_1}{t\{x_2:=a_2,\ldots,x_n:=a_n\}})^{\A}$, we observe
  that $c\in S$, by Prop.~\ref{p:SeparLam}.
  Moreover, we have
  $ca\cle(t\{x_1:=a_1,x_2:=a_2,\ldots,x_n:=a_n\})^{\A}\cle b$
  by Prop~\ref{p:BetaEta}.
  And by adjunction, we deduce that $c\cle(a\to b)$, hence
  $(a\to b)\in S$.
\end{proof}

In what follows, we shall say that a separator $S\subseteq\A$ is
\emph{finitely generated} when $S=\Sep(X)$ for some finite
subset~$X\subseteq\A$.
Two important examples of finitely generated separators of an
implicative structure $\A=(\A,{\cle},{\to})$ are:
\begin{itemize}
\item the \emph{intuitionistic core} of~$\A$, defined by
  $\SJ(\A):=\Sep(\varnothing)$;
\item the \emph{classical core} of~$\A$, defined by
  $\SK(\A):=\Sep(\{\cc^{\A}\})$.
\end{itemize}
By definition, the set $\SJ(\A)$ (resp.\ $\SK(\A)$) is the smallest
separator (resp.\ the smallest classical separator) of~$\A$; and from
Prop.~\ref{p:SeparGen}, it is clear that the implicative structure
$\A=(\A,{\cle},{\to})$ is intuitionistically consistent
(resp.\ classically consistent) in the sense of
Def.~\ref{d:ConsImpStruct} if and only if $\bot\notin\SJ(\A)$
(resp.\ $\bot\notin\SK(\A)$).

\subsection{Interpreting first-order logic}
\label{ss:InterpFOL}

\subsubsection{Conjunction and disjunction}
Each implicative structure $\A=(\A,{\cle},{\to})$ describes a
particular logic from the interaction between implication $a\to b$
and universal quantification, seen as a meet w.r.t.\ the ordering
$a\cle b$ of subtyping.
In such a framework, conjunction (notation: $a\times b$) and
disjunction (notation: $a+b$) are naturally defined using the standard
encodings of minimal second-order logic~\cite{Gir72,Gir89}:
$$\begin{array}{r@{~~}c@{~~}l}
  a\times b&:=&\ds\bigmeet_{c\in\A}((a\to b\to c)\to c)\\
  a+b&:=&\ds\bigmeet_{c\in\A}((a\to c)\to(b\to c)\to c)\\
\end{array}$$
Finally, negation and logical equivalence are defined as expected,
letting $\lnot a:=(a\to\bot)$ and
$a\leftrightarrow b:=(a\to b)\times(b\to a)$.
We easily check that:
\begin{proposition}
  When $(\A,{\cle},{\to})$ is a complete Heyting algebra:
  $$a\times b~=~a\meet b\qquad\text{and}\qquad
  a+b~=~a\join b\eqno(\text{for all}~a,b\in\A)$$
\end{proposition}

(The proof is left as an exercise to the reader.)

\bigbreak
In the general case, the introduction and elimination rules of
conjunction and disjunction are naturally expressed as semantic typing
rules (see Section~\ref{ss:SemTyping}) using the very same proof-terms
as in Curry-style system~F~\cite{Lei83,BLRU94}:
\begin{proposition}[Typing rules for $\times$ and
    $+$]\label{p:ValidConjDisjRules}
  The semantic typing rules
  $$\begin{array}{c}
    \infer{\Gamma\vdash\Lam{z}{z\,t\,u}:a\times b}{
      \Gamma\vdash t:a & \Gamma\vdash u:b}\qquad
    \infer{\Gamma\vdash t\,(\Lam{xy}{x}):a}{
      \Gamma\vdash t:a\times b}\qquad
    \infer{\Gamma\vdash t\,(\Lam{xy}{y}):b}{
      \Gamma\vdash t:a\times b}\\
    \noalign{\medskip}
    \infer{\Gamma\vdash\Lam{zw}{z\,t}:a+b}{\Gamma\vdash t:a}\qquad
    \infer{\Gamma\vdash\Lam{zw}{w\,t}:a+b}{\Gamma\vdash t:b}\\
    \noalign{\medskip}
    \infer{\Gamma\vdash t\,(\Lam{x}{u})\,(\Lam{y}{v}):c}{
      \Gamma\vdash t:a+b &
      \Gamma,x:a\vdash u:c & \Gamma,y:b\vdash v:c}\\
  \end{array}$$
  are valid in any implicative structure.
\end{proposition}
(Recall that~~$\Gamma\vdash t:a$~~means:\ \
$\FV(t)\subseteq\dom(\Gamma)$ and $(t[\Gamma])^{\A}\cle a$.)

Following the spirit of Prop.~\ref{p:CombTypes}, we can notice that
via the interpretation $t\mapsto t^{\A}$ of pure $\lambda$-terms into
the implicative structure~$\A$ (Section~\ref{ss:AppLam}), the pairing
construct $\<t,u\>:=\Lam{z}{z\,t\,u}$ appears to be the same as
conjunction itself:
\begin{proposition}
  For all $a,b\in\A$:\quad
  $\<a,b\>^{\A}=(\Lam{z}{z\,a\,b})^{\A}=a\times b$.
\end{proposition}

\begin{proof}
  Same proof technique as for Prop.~\ref{p:CombTypes}.
\end{proof}

\subsubsection{Quantifiers}
In any implicative structure $\A=(\A,{\cle},{\to})$, the universal
quantification of a family of truth values $(a_i)_{i\in I}\in\A^I$ is
naturally defined as its meet:
$$\bigforall_{\!\!i\in I\!\!}a_i~:=~\bigmeet_{\!\!i\in I\!\!}a_i\,.$$
It is obvious that:
\begin{proposition}[Rules for $\forall$]\label{p:ValidForRules}
  The following semantic typing rules
  $$\infer{\Gamma\vdash t:\bigforall_{i\in I}a_i}{
    \Gamma\vdash t:a_i\quad(\text{for all}~i\in I)}\qquad\qquad
  \infer[(i_0\in I)]{\Gamma\vdash t:a_{i_0}}{
    \Gamma\vdash t:\bigforall_{i\in I}a_i}$$
  are valid in any implicative structure.
\end{proposition}

In such a framework, it would be quite natural to define existential
quantification dually, that is: as a join.
Alas, this interpretation does not fulfill (in general) the elimination
rule for~$\exists$---remember that joins only exist by accident.
As for conjunction and disjunction, we shall use the corresponding
encoding in second-order minimal logic~\cite{Gir72,Gir89}, letting:
$$\bigexists_{\!\!i\in I\!\!}a_i~:=~\bigmeet_{c\in\A}
\biggl(\bigmeet_{\!\!i\in I\!\!}(a_i\to c)~\to~c\biggr)\,.$$
Again, we easily check that:
\begin{proposition}
  When $(\A,{\cle},{\to})$ is a complete Heyting algebra:
  $$\bigexists_{i\in I}a_i~=~\bigjoin_{i\in I}a_i
  \eqno\hskip-30mm(\text{for all}~(a_i)_{i\in I}\in\A^I)$$
\end{proposition}
Coming back to the general case:
\begin{proposition}[Rules for $\exists$]\label{p:ValidExRules}
  The following semantic typing rules
  $$\infer[(i_0\in I)]{\Gamma\vdash\Lam{z}{z\,t}:
    \bigexists_{i\in I}a_i}{\Gamma\vdash t:a_{i_0}}\qquad
  \infer{\Gamma\vdash t\,(\Lam{x}{u}):c}{
    \Gamma\vdash t:\bigexists_{i\in I}a_i &\quad
    \Gamma,x:a_i\vdash u:c\quad(\text{for all}~i\in I)}$$
  are valid in any implicative structure.
\end{proposition}

\subsubsection{Leibniz equality}
Given any two objects $\alpha$ and~$\beta$, the \emph{identity}
of~$\alpha$ and~$\beta$ (in the sense of Leibniz) is expressed by the
truth value $\mathbf{id}^{\A}(\alpha,\beta)\in\A$ defined by:
$$\mathbf{id}^{\A}(\alpha,\beta)~:=~\begin{cases}
  \mathbf{I}^{\A}&\text{if}~\alpha=\beta\\
  \top\to\bot&\text{if}~\alpha\neq\beta\\
\end{cases}$$
It is a straightforward exercise to check that when~$\alpha$
and~$\beta$ belong to a given set~$M$, the above interpretation of
Leibniz equality amounts to the usual second-order encoding:
\begin{proposition}
  For all sets~$M$ and for all $\alpha,\beta\in M$, we have:
  $$\mathbf{id}^{\A}(\alpha,\beta)~=~
  \bigmeet_{p\in\A^{M}}(p(\alpha)\to p(\beta))\,.$$
\end{proposition}
Moreover:
\begin{proposition}[Rules for
    $\mathbf{id}^{\A}$]\label{p:ValidIdRules}
  Given a set $M$, a function $p:M\to\A$ and two objects
  $\alpha,\beta\in M$, the following semantic typing rules are valid:
  $$\infer{\Gamma\vdash\Lam{x}{x}:\alpha=\alpha}{}\qquad\qquad
  \infer{\Gamma\vdash t\,u:p(\beta)}{
    \Gamma\vdash t:\mathbf{id}^{\A}(\alpha,\beta) &
    \Gamma\vdash u:p(\alpha)}$$
\end{proposition}

\subsubsection{Interpreting a first-order language}
Let $\A=(\A,{\cle},{\to})$ be an implicative structure.
An $\A$-valued interpretation of a first-order language $\mathscr{L}$
is defined by:
\begin{itemize}
\item a domain of interpretation $M\neq\varnothing$;
\item an $M$-valued function $f^{M}:M^k\to M$ for each $k$-ary
  function symbol of $\mathscr{L}$;
\item a truth-value function $p^{\A}:M^k\to\A$ for each $k$-ary
  predicate symbol of $\mathscr{L}$.
\end{itemize}
As usual, we call a \emph{term with parameters in~$M$} (resp.\ a
\emph{formula with parameters in~$M$}) any first-order term
(resp.\ any formula) of the first-order language $\mathscr{L}$
enriched with constant symbols taken in~$M$.
Each closed term~$t$ with parameters in~$M$ is naturally interpreted
as the element $t^M\in M$ defined from the equations
$$a^M=a\quad(\text{if}~a~\text{is a parameter})\qquad\qquad
f(t_1,\ldots,t_k)^{M}=f^{M}(t_1^M,\ldots,t_k^M)$$
whereas each closed formula~$\phi$ with parameters in~$M$ is
interpreted as the truth value $\phi^{\A}\in\A$ defined from the
equations:
$$\begin{array}{r@{~~}c@{~~}l@{\qquad}r@{~~}c@{~~}l}
  (t_1=t_2)^{\A} &:=& \mathbf{id}^{\A}(t_1^M,t_2^M) &
  (p(t_1,\ldots,t_k))^{\A} &:=& p^{\A}(t_1^M,\ldots,t_k^M) \\[3pt]
  (\phi\limp\psi)^{\A} &:=& \phi^{\A}\to\psi^{\A} &
  (\lnot\phi)^{\A} &:=& \phi^{\A}\to\bot \\[3pt]
  (\phi\land\psi)^{\A} &:=& \phi^{\A}\times\psi^{\A} &
  (\phi\lor\psi)^{\A} &:=& \phi^{\A}+\psi^{\A} \\[3pt]
  (\forall x\,\phi(x))^{\A} &:=&
  \ds\bigforall_{\alpha\in M}(\phi(\alpha))^{\A} &
  (\exists x\,\phi(x))^{\A} &:=&
  \ds\bigexists_{\alpha\in M}(\phi(\alpha))^{\A} \\
\end{array}$$

\begin{proposition}[Soundness]\label{p:FOSoundness}
  If a closed formula~$\phi$ of the language~$\mathscr{L}$ is an
  intuitionistic tautology (resp.\ a classical tautology), then
  $$\phi^{\A}\in~\SJ(\A)\qquad(\text{resp.}~\phi^{\A}\in~\SK(\A))$$
  where $\SJ(\A)$ (resp.\ $\SK(\A)$) is the intuitionistic core
  (resp.\ the classical core) of~$\A$.
\end{proposition}

\begin{proof}
  By induction on the derivation~$d$ of the formula~$\phi$ (in
  natural deduction), we construct a closed $\lambda$-term~$t$
  (possibly containing the constant~$\cc$ when the derivation~$d$ is
  classical) such that $\vdash t:\phi^{\A}$, using the semantic typing
  rules given in Prop.~\ref{p:SemTypingRules},
  \ref{p:ValidConjDisjRules}, \ref{p:ValidForRules},
  \ref{p:ValidExRules} and \ref{p:ValidIdRules}.
  So that $t^{\A}\cle\phi^{\A}$.
  We conclude by Prop.~\ref{p:SeparLam}.
\end{proof}

\subsection{Entailment and the induced Heyting algebra}
\label{ss:InducedHA}

Let $(\A,{\cle},{\to})$ be an implicative structure.
Each separator $S\subseteq\A$ induces a binary relation of
\emph{entailment}, written $a\ent_Sb$ and defined by
$$a\ent_Sb~~:\liff~~(a\to b)\in S
\eqno(\text{for all}~a,b\in\A)$$

\begin{proposition}
  The relation $a\ent_Sb$ is a preorder on~$\A$.
\end{proposition}

\begin{proof}
  Reflexivity: given $a\in\A$, we have
  $\mathbf{I}^{\A}\cle(a\to a)\in S$.
  Transitivity: given $a,b,c\in\A$ such that $(a\to b)\in S$ and
  $(b\to c)\in S$, we observe that
  $\mathbf{B}^{\A}=(\Lam{xyz}{x(yz)})^{\A}\cle
  (b\to c)\to(a\to b)\to a\to c\in S$, hence
  $(a\to c)\in S$, by modus ponens.
\end{proof}

In what follows, we shall write $\A/S=(\A/S,{\le_S})$ the \emph{poset
  reflection} of the pre-ordered set $(\A,{\ent_S})$, where:
\begin{itemize}
\item[$\bullet$] $\A/S:=\A/{\tnent_S}$ is the quotient of~$\A$ by the
  equivalence relation $a\tnent_Sb$ induced by the preorder
  $a\ent_Sb$, which is defined by:
  $$a\tnent_Sb~~:\liff~~(a\to b)\in S\land(b\to a)\in S
  \eqno(\text{for all}~a,b\in\A)$$
\item[$\bullet$] $\alpha\le_S\beta$ is the order induced by the
  preorder $a\ent_Sb$ in the quotient set~$\A/S$, which is
  characterized by: 
  $$[a]\le_S[b]~~\liff~~a\ent_Sb
  \eqno(\text{for all}~a,b\in\A)$$
  writing $[a],[b]$ the equivalence classes of~$a,b\in\A$
  in the quotient~$\A/S$.
\end{itemize}

\begin{proposition}[Induced Heyting algebra]\label{p:InducedHA}
  For each separator $S\subseteq\A$, the poset reflection
  $H:=(\A/S,{\le_S})$ of the pre-ordered set $(\A,{\ent_S})$ is a
  Heyting algebra whose operations are given for all $a,b\in\A$ by:
  $$\begin{array}{r@{~~}c@{~~}l@{\qquad}r@{~~}c@{~~}l}
    {[a]\to_H[b]}&=&[a\to b] \\
    {[a]\land_H[b]}&=&[a\times b]&
    \top_H&=&[\top]~=~S \\
    {[a]\lor_H[b]}&=&[a+b] &
    \bot_H&=&[\bot]~=~\{c\in\A:(\lnot c)\in S\}\\
  \end{array}$$
  (writing $[a]$ the equivalence class of $a$).
  If, moreover, the separator $S\subseteq\A$ is classical, then the
  induced Heyting algebra~$H=(\A/S,{\le_S})$ is a Boolean algebra.
\end{proposition}

In what follows, the quotient poset $H:=(\A/S,{\le_S})$ is called the
\emph{Heyting algebra induced by the implicative algebra
  $(\A,{\cle},{\to},S)$}.

\begin{proof}
  Given $a,b\in\A$, we observe the following:
  \begin{itemize}
  \item[$\bullet$] For all $c\in\A$, we have
    $\mathbf{I}^{\A}\cle(\bot\to c)\in S$, hence $[\bot]\le_S[c]$.
  \item[$\bullet$] For all $c\in\A$, we have
    $(c\to\top)=\top\in S$, hence $[c]\le_S[\top]$.
  \item[$\bullet$]
    $(\Lam{z}{z\,(\Lam{xy}{x})})^{\A}\cle(a\times b\to a)\in S$
    and $(\Lam{z}{z\,(\Lam{xy}{y})})^{\A}\cle(a\times b\to b)$, hence
    $[a\times b]\le_S[a]$ and $[a\times b]\le_S[b]$.
    Conversely, if $c\in\A$ is such that $[c]\le_S[a]$ and
    $[c]\le_S[b]$, we have $(c\to a)\in S$ and $(c\to b)\in S$.
    From Prop.~\ref{p:SeparLam} and Prop.~\ref{p:PropApp}~(2), we get
    $(\Lam{zw}{w\,((c\to a)\,z)\,((c\to b)\,z)})^{\A}
    \cle(c\to a\times b)\in S$, hence $[c]\le_S[a\times b]$.
    Therefore: $[a\times b]=\inf_H([a],[b])=[a]\wedge_H[b]$.
  \item[$\bullet$] $(\Lam{xzw}{z\,x})^{\A}\cle(a\to a+b)\in S$ and
    $(\Lam{yzw}{w\,y})^{\A}\cle(b\to a+b)\in S$, hence
    $[a]\le_S[a+b]$ and $[b]\le_S[a+b]$.
    Conversely, if $c\in\A$ is such that $[a]\le_S[c]$ and
    $[b]\le_S[c]$, we have $(a\to c)\in S$ and $(b\to c)\in S$.
    From Prop.~\ref{p:SeparLam} we get
    $(\Lam{z}{z\,(a\to c)\,(b\to c)})^{\A}\cle(a+b\to c)\in S$,
    hence $[a+b]\le_S[c]$.
    Therefore: $[a+b]=\sup_H([a],[b])=[a]\vee_H[b]$.
  \item[$\bullet$] For all $c\in\A$, we have
    $(\Lam{wz}{z\,w})^{\A}\cle((c\to a\to b)\to c\times a\to b)\in S$
    and $(\Lam{wxy}{w\,\<x,y\>})^{\A}\cle
    ((c\times a\to b)\to c\to a\to b)\in S$.
    Hence the equivalence
    $(c\to a\to b)\in S$ iff $(c\times a\to b)\in S$, that is:
    $[c]\le_S[a\to b]$ iff $[c\times a]\le_S[b]$.
    Therefore: $[a\to b]=
    \max\{\gamma\in H:\gamma\wedge_H[a]\le_S[b]\}=[a]\to_H[b]$.
  \end{itemize}
  So that the poset $(\A/S,{\le_S})$ is a Heyting algebra.
  If, moreover, the separator $S\subseteq\A$ is classical, then
  we have $\cc^{\A}\cle(\lnot\lnot a\to a)\in S$ for all $a\in\A$,
  so that $\lnot_H\lnot_H[a]=[\lnot\lnot a]\le_S[a]$, which means that
  $(\A/S,{\le_S})$ is a Boolean algebra.
\end{proof}

\begin{remarks}\label{r:InducedHA}
  (1)~~In the particular case where $(\A,{\cle},{\to})$ is a complete
  Heyting algebra (Section~\ref{sss:ComplHA}), the separator
  $S\subseteq\A$ is a filter, and the above construction amounts to
  the usual construction of the quotient $\A/S$ in Heyting
  algebras.\par
  (2)~~Coming back to the general framework of implicative structures,
  it is clear that the induced Heyting algebra $H=(\A/S,{\le_S})$ is
  non-degenerate (i.e.\ $[\top]\neq[\bot]$) if and only if the
  separator $S\subseteq\A$ is consistent (i.e.\ $\bot\notin S$).\par
  (3)~~When the separator $S\subseteq\A$ is classical (i.e.\ when
  $\cc^{\A}\in S$), the induced Heyting algebra is a Boolean algebra.
  The converse implication does not hold in general, and we shall see
  a counter-example in Section~\ref{ss:SeparMax} below
  (Remark~\ref{r:SeparMaxNoCC}).\par
  (4)~~In general, the induced Heyting algebra $(\A/S,{\le_S})$ is
  not complete---so that it is not an implicative structure either.
  A simple counter-example is given by the complete Boolean algebra
  $\Pow(\omega)$ (which is also an implicative structure) equipped
  with the Fr{\'e}chet filter
  $F=\{a\in\Pow(\omega):a~\text{cofinite}\}$ (which is also a
  classical separator of~$\Pow(\omega)$), since the quotient
  Boolean algebra $\Pow(\omega)/F$ is not complete~\cite[Chap. 2,
    \S~5.5]{Kop89}.
\end{remarks}

\subsection{Ultraseparators}\label{ss:SeparMax}

Let $\A=(\A,{\cle},{\to})$ be an implicative structure.
Although the separators of~$\A$ are in general not filters, they can
be manipulated similarly to filters.
By analogy with the notion of ultrafilter, we define the notion of
ultraseparator:
\begin{definition}[Ultraseparator]
  We call an \emph{ultraseparator} of~$\A$ any separator
  $S\subseteq\A$ that is both consistent and maximal among consistent
  separators (w.r.t.\ $\subseteq$).
\end{definition}

From Zorn's lemma, it is clear that:
\begin{lemma}
  For each consistent separator $S_0\subseteq\A$, there exists an
  ultraseparator $S\subseteq\A$ such that $S_0\subseteq S$.
\end{lemma}

\begin{proposition}
  For each separator $S\subseteq\A$, the following are equivalent:
  \begin{enumerate}[(99)]
  \item[(1)] $S$ is an ultraseparator of~$\A$.
  \item[(2)] The induced Heyting algebra $(\A/S,{\le_S})$ is the
    2-element Boolean algebra.
  \end{enumerate}
\end{proposition}

\begin{proof}
  $(1)\limp(2)$\quad Assume that $S\subseteq\A$ is an ultraseparator.
  Since~$S$ is consistent, we have $\bot\notin S$ and thus
  $[\bot]\neq[\top]~({=}~S)$.
  Now, take $a_0\in\A$ such that $[a_0]\neq[\bot]$, and let
  $S'=\{a\in\A:[a_0]\le_S[a]\}=\{a\in\A:(a_0\to a)\in S\}$ be the
  preimage of the principal filter ${\uparrow}[a_0]\subseteq\A/S$ via
  the canonical surjection $[\,\cdot\,]:\A\to\A/S$.
  Clearly, the subset $S'\subseteq\A$ is a consistent separator such
  that $S\subseteq S'$ and $a_0\in S'$.
  By maximality, we have $S'=S$, so that $a_0\in S$ and thus
  $[a_0]=[\top]$.
  Therefore, $\A/S=\{[\bot],[\top]\}$ is the 2-element Heyting
  algebra, that is also a Boolean algebra.
  \smallbreak\noindent
  $(2)\limp(1)$\quad Let us assume that $\A/S$ is the 2-element
  Boolean algebra (so that $\A/S=\{[\bot],[\top]\}$), and consider a
  consistent separator $S'\subseteq\A$ such that $S\subseteq S'$.
  For all $a\in S'$, we have $\lnot a\notin S$
  (otherwise, we would have $a,\lnot a\in S'$, and thus
  $\bot\in S'$), hence $a\notin[\bot]$ and thus
  $a\in[\top]=S$.
  Therefore, $S'=S$.
\end{proof}

\begin{remark}\label{r:SeparMaxNoCC}
  It is important to notice that a maximal separator is not
  necessarily classical, although the induced Heyting algebra is
  always the trivial \emph{Boolean} algebra.
  Indeed, we have seen in Section~\ref{sss:ImpAlgRealizJ} that any
  total combinatory algebra $(P,{\,\cdot\,},\mathtt{k},\mathtt{s})$
  induces an implicative algebra
  $(\A,{\cle},{\to},S)=
  (\Pow(P),{\subseteq},{\to},\Pow(P)\setminus\{\varnothing\})$
  whose separator
  $S:=\Pow(P)\setminus\{\varnothing\}=\A\setminus\{\bot\}$ 
  is obviously an ultraseparator.
  But when the set~$P$ has more than one element, it is easy to check
  that
  $$\cc^{\A}~\cle~\bigmeet_{\!\!a\in\A\!\!}(\lnot\lnot a\to a)
  ~=~\bot~({=}~\varnothing)$$
  so that $\cc^{\A}=\bot\notin S$.
  On the other hand, the induced Heyting algebra $\A/S$ is the
  trivial Boolean algebra, which corresponds to the well-known fact
  that, in intuitionistic realizability, one of both formulas~$\phi$
  and $\lnot\phi$ is realized for each \emph{closed} 
  formula~$\phi$.
  So that all the closed instances of the law of excluded middle are
  actually realized.
  Of course, this does not imply that the law of excluded middle
  itself---that holds for all \emph{open} formulas---is (uniformly)
  realized.
  By the way, this example also shows that a non-classical
  separator~$S\subseteq\A$ may induce a Boolean algebra (see
  Remark~\ref{r:InducedHA}~(3)).
\end{remark}

\subsection{Separators, filters and non-deterministic
  choice}\label{ss:SeparFilter}

Like filters, separators are upwards closed and nonempty, but they
are not closed under binary meets in general.
In this section, we shall now study the particular case of separators
that happen to be filters.

\subsubsection{Non-deterministic choice}
Given an implicative structure $\A=(\A,{\cle},{\to})$, we let:
$$\Fork^{\A}
~:=~(\Lam{xy}{x})^{\A}\meet(\Lam{xy}{y})^{\A}
~=~\bigmeet_{\!\!\!\!a,b\in\A\!\!\!\!}(a\to b\to a\meet b)\,.$$
By construction, we have:
$$\Fork^{\A}a\,b~\cle~a\qquad\text{and}\qquad
\Fork^{\A}a\,b~\cle~b\eqno(\text{for all}~a,b\in\A)$$
so that we can think of $\Fork^{\A}$ as the \emph{non-deterministic
  choice operator} (in~$\A$), that takes two arguments $a,b\in\A$ and
returns~$a$ or~$b$ in an non-deterministic way%
\footnote{In classical realizability, it can be shown~\cite{GM15} that
  the universal realizers of the second-order formula
  $\forall\alpha\,\forall\beta\,(\alpha\to\beta\to\alpha\cap\beta)$
  (where $\alpha\cap\beta$ denotes the intersection of~$\alpha$
  and~$\beta$) are precisely the closed terms~$t$ with the
  non-deterministic computational rules\ \
  $t\star u\cdot v\cdot\pi\succ u\star\pi$\ \ and\ \
  $t\star u\cdot v\cdot\pi\succ v\star\pi$\ \
  for all closed terms~$u$, $v$ and for all stacks~$\pi$.
  Recall that Krivine's abstract machine~\cite{Kri09} can be extended
  with extra instructions at will (for instance: an instruction
  $\Fork$ with the aforementioned non-deterministic behavior), so
  that such realizers may potentially exist.}.

From the point of view of logic, recall that the meet $a\meet b$ of
two elements $a,b\in\A$ can be seen as a strong form of conjunction.
Indeed, it is clear that
$$(\Lam{xz}{z\,x\,x})^{\A}~\cle~(a\meet b\to a\times b)~\in~S$$
for all separators $S\subseteq\A$ and for all $a,b\in\A$, so that we
have $a\meet b\ent_Sa\times b$.
Seen as a type, the non-deterministic choice operator
$\Fork^{\A}=\bigmeet_{a,b}(a\to b\to a\meet b)$ precisely
expresses the converse implication, and we easily check that:
\begin{proposition}[Characterizing filters]\label{p:CharacFilter}
  For all separators $S\subseteq\A$, the following assertions are
  equivalent:
  \begin{enumerate}[(99)]
  \item[(1)] $\Fork^{\A}\in S$;
  \item[(2)] $[a\meet b]_{/S}=[a\times b]_{/S}$ for all $a,b\in\A$;
  \item[(3)] $S$ is a filter (w.r.t.\ the ordering $\cle$).
  \end{enumerate}
\end{proposition}

\begin{proof}
  $(1)\limp(2)$\quad
  For all $a,b\in\A$, it is clear that
  $[a\meet b]_{/S}\le_S[a\times b]_{/S}$.
  And from~(1), we get
  $\bigl(\Lam{z}{z\,\Fork^{\A}}\bigr)^{\A}
  \cle(a\times b\to a\meet b)\in S$,
  hence $[a\times b]_{/S}\le_S[a\meet b]_{/S}$.
  \smallbreak\noindent
  $(2)\limp(3)$\quad Let us assume that $a,b\in S$.
  We have $[a]_{/S}=[b]_{/S}=[\top]_{/S}$, so that by~(2) we get
  $[a\meet b]_{/S}=[a\times b]_{/S}
  =[\top\times\top]_{/S}=[\top]_{/S}$.
  Therefore $(a\meet b)\in S$.
  \smallbreak\noindent
  $(3)\limp(1)$\quad
  It is clear that $(\Lam{xy}{x})^{\A}\in S$ and
  $(\Lam{xy}{y})^{\A}\in S$, so that from (3) we get
  $\Fork^{\A}=(\Lam{xy}{x})^{\A}\meet(\Lam{xy}{y})^{\A}\in S$.
\end{proof}

\subsubsection{Non-deterministic choice and induction}
In second-order logic~\cite{Gir89,Kri93}, the predicate
$\Nat(x)$ expressing that a given individual~$x$ is a natural number%
\footnote{Here, we recognize Dedekind's construction of natural
  numbers, as the elements of a fixed Dedekind-infinite set that
  are reached by the induction principle (seen as a local property).}
is given by:
$$\Nat(x)~:=~
\forall Z\,(Z(0)\limp\forall y\,(Z(y)\limp Z(y+1))\limp Z(x))\,.$$
In intuitionistic realizability~\cite{Oos08,Kri93} as in classical
realizability~\cite{Kri09}, it is well-known that the (unrelativized)
induction principle\ \ $\textsc{Ind}:=\forall x~\Nat(x)$\ \
is not realized in general, even when individuals are interpreted by
natural numbers in the model.
(Technically, this is the reason why uniform quantifications over the
set of natural numbers need to be replaced by quantifications
relativized to the predicate $\Nat(x)$.)

In any implicative structure $\A=(\A,{\cle},{\to})$, the syntactic
predicate $\Nat(x)$ is naturally interpreted by the semantic predicate
$\Nat^{\A}:\omega\to\A$ defined by
$$\Nat^{\A}(n)~:=~\bigmeet_{\!a\in\A^{\omega}\!}
\Biggl(a_0\to
\bigmeet_{\!\!i\in\omega\!\!}\Bigl(a_i\to a_{i+1}\Bigr)\to
a_n\Biggr)\eqno(\text{for all}~n\in\omega)$$
while the (unrelativized) induction scheme is interpreted by the truth
value
$$\textsc{Ind}^{\A}~:=~
\bigmeet_{\!\!n\in\omega\!\!}\Nat^{\A}(n)\,.$$
The following proposition states that the unrelativized induction
scheme $\textsc{Ind}^{\A}$ and the non-deterministic choice
operator~$\Fork^{\A}$ are intuitionistically equivalent in~$\A$:
\begin{proposition}
  $\textsc{Ind}^{\A}\tnent_{\SJ(\A)}\Fork^{\A}$
  (where $\SJ(\A)$ is the intuitionistic core of~$\A$).
\end{proposition}

\begin{proof}
  $(\textsc{Ind}^{\A}\ent_{\SJ(\A)}\Fork^{\A})$\quad
  Given $a,b\in\A$, we let $c_0=a$ and $c_n=b$ for all $n\ge 1$.
  From an obvious argument of subtyping, we get
  $$\textsc{Ind}^{\A}~\cle~\bigmeet_{\!\!\!n\in\omega\!\!\!}
  \biggl(c_0\to
  \bigmeet_{\!\!\!i\in\omega\!\!\!}\Bigl(c_i\to c_{i+1}\Bigr)
  \to c_n\biggr)~=~
  a\to((a\to b)\meet(b\to b))\to a\meet b$$
  so that $(\Lam{nxy}{n\,x\,(\mathbf{K}\,y)})^{\A}\cle
  (\textsc{Ind}^{\A}\to a\to b\to a\meet b)$.
  Now taking the meet for all $a,b\in\A$, we thus get
  $(\Lam{nxy}{n\,x\,(\mathbf{K}\,y)})^{\A}\cle
  (\textsc{Ind}^{\A}\to\Fork^{\A})\in\SJ(\A)$.
  \smallbreak\noindent
  $(\Fork^{\A}\ent_{\SJ(\A)}\textsc{Ind}^{\A})$\quad
  Consider the following pure $\lambda$-terms:
  $$\begin{array}{rcl}
    \mathbf{zero}&:=&\Lam{xy}{x}\\
    \mathbf{succ}&:=&\Lam{nxy}{y\,(n\,x\,y)}\\
    \mathbf{Y}&:=&(\Lam{yf}{f\,(y\,y\,f)})\,
    (\Lam{yf}{f\,(y\,y\,f)})\\
    t[x]&:=&\mathbf{Y}\,(\Lam{r}{
      x\,\mathbf{zero}\,(\mathbf{succ}\,r)})\\
  \end{array}$$
  (here, $\mathbf{Y}$ is Turing's fixpoint combinator).
  From the typing rules of Prop.~\ref{p:SemTypingRules}, we easily
  check that $\mathbf{zero}^{\A}\cle\Nat(0)$ and
  $\mathbf{succ}^{\A}\cle\Nat(n)\to\Nat(n+1)$ for all $n\in\omega$.
  Now, consider the element
  $\Theta:=\bigl(t[\Fork^{\A}]\bigr)^{\A}\in\A$.
  From the reduction rule of $\mathbf{Y}$, we get
  $$\Theta~\cle~
  \Fork^{\A}\mathbf{zero}^{\A}(\mathbf{succ}^{\A}\Theta)
  ~\cle~\mathbf{zero}^{\A}\meet\mathbf{succ}^{\A}\Theta\,.$$
  By a straightforward induction on~$n$, we deduce that
  $\Theta\cle\Nat(n)$ for all $n\in\omega$, hence
  $\Theta\cle\textsc{Ind}^{\A}$.
  Therefore: $(\Lam{x}{t[x]})^{\A}\cle
  (\Fork^{\A}\to\Theta)\cle
  (\Fork^{\A}\to\textsc{Ind}^{\A})\in\SJ(\A)$.
\end{proof}

\subsubsection{Non-deterministic choice and the parallel-or}
A variant of the non-deterministic choice operator is the
\emph{parallel `or'}, that is defined by:
$$\Por^{\A}~:=~
(\bot\to\top\to\bot)\meet(\top\to\bot\to\bot)\,.$$
Intuitively, the parallel `or' is a function that takes two
arguments---one totally defined and the other one totally
undefined---and returns the most defined of both, independently from
the order in which both arguments were passed to the function.
(Recall that according to the definitional ordering
$a\sqsubseteq b:\liff a\cge b$, the element $\bot$ represents
the totally defined object whereas $\top$ represents the totally
undefined object.)

We observe that
$$\Fork^{\A}~=~\bigmeet_{\!\!\!\!a,b\in\A\!\!\!\!}(a\to b\to a\meet b)
~\cle~(\bot\to\top\to\bot)\meet(\top\to\bot\to\bot)\,,$$
which means that the parallel `or' $\Por^{\A}$ is a super-type of the
non-deterministic choice operator $\Fork^{\A}$.
However, both operators are classically equivalent:
\begin{proposition}\label{p:PorForkEquivK}
  $\Por^{\A}\tnent_{\SK(\A)}\Fork^{\A}$
  (where $\SK(\A)$ is the classical core of~$\A$).
\end{proposition}

\begin{proof}
  $(\Fork^{\A}\ent_{\SK(\A)}\Por^{\A})$\quad Obvious, by subtyping.
  \smallbreak\noindent
  $(\Por^{\A}\ent_{\SK(\A)}\Fork^{\A})$\quad
  Let $t:=\Lam{zxy}{\cc\,(\Lam{k}{z\,(k\,x)\,(k\,y)})}$.
  From the semantic typing rules of Prop.~\ref{p:SemTypingRules}
  (and from the type of~$\cc$) we easily check that
  $$t^{\A}\cle(\Por^{\A}\to a\to b\to a)\quad\text{and}\quad
  t^{\A}\cle(\Por^{\A}\to a\to b\to b)$$
  for all $a,b\in\A$, hence
  $t^{\A}\cle(\Por^{\A}\to\Fork^{\A})\in\SK(\A)$.
\end{proof}

\subsubsection{The case of finitely generated
  separators}\label{sss:FinGenSep}
In Prop.~\ref{p:CharacFilter} above, we have seen that a separator
$S\subseteq\A$ is a filter if and only if it contains the
non-deterministic choice operator $\Fork^{\A}$.
In the particular case where the separator $S\subseteq\A$ is finitely
generated (see Section~\ref{ss:SeparGen}), the situation is even more
dramatic:
\begin{proposition}\label{p:CharacPrincFilter}
  Given a separator $S\subseteq\A$, the following are equivalent:
  \begin{enumerate}[(99)]
  \item[(1)] $S$ is finitely generated and $\Fork^{\A}\in S$.
  \item[(2)] $S$ is a principal filter of~$\A$:\quad
    $S={\uparrow}\{\Theta\}$\quad for some $\Theta\in S$.
  \item[(3)] The induced Heyting algebra $(\A/S,{\le_S})$ is
    complete and the canonical surjection $[\,\cdot\,]_{/S}:\A\to\A/S$
    commutes with arbitrary meets:
    $$\biggl[\bigmeet_{\!\!i\in I\!\!}a_i\biggr]_{/S}~=~
    \bigwedge_{i\in I}[a_i]_{/S}
    \eqno(\text{for all}~(a_i)_{i\in I}\in\A^I)$$
  \end{enumerate}
\end{proposition}

\begin{proof}
  $(1)\limp(2)$\quad Let us assume that
  $S={\uparrow}@(\{g_1,\ldots,g_n\})$ for some $g_1,\ldots,g_n\in S$
  (see Section~\ref{ss:SeparGen}, Prop.~\ref{p:SeparGen}), and
  $\Fork^{\A}\in S$.
  From the latter assumption, we know (by Prop.~\ref{p:CharacFilter})
  that~$S$ is closed under all finite meets, so that for all $k\ge 1$,
  we have:
  $$\Fork_k^{\A}~:=~\bigmeet_{i=1}^k(\Lam{x_1\cdots x_k}{x_i})^{\A}
  ~=~\bigmeet_{\hskip-10pta_1,\ldots,a_k\in\A\hskip-10pt}
  (a_1\to\cdots\to a_k\to a_1\meet\cdots\meet a_k)~\in~S\,.$$
  Let $\Theta:=\bigl(\mathbf{Y}\,(\Lam{r}{
    \Fork_{n+1}^{\A}g_1\cdots g_n\,(r\,r)})\bigr)^{\A}$, where
  $\mathbf{Y}:=(\Lam{yf}{f\,(y\,y\,f)})\,(\Lam{yf}{f\,(y\,y\,f)})$
  is Turing's fixpoint combinator.
  Since $g_1,\ldots,g_n,\Fork_{n+1}^{\A}\in S$, it is clear that
  $\Theta\in S$.
  From the evaluation rule of $\mathbf{Y}$, we have\ \
  $\Theta~\cle~\Fork_{n+1}^{\A}g_1\cdots g_n\,(\Theta\Theta)
  ~\cle~g_1\meet\cdots\meet g_n\meet\Theta\Theta$,\ \
  hence $\Theta\cle g_i$ for all $i\in\{1,\ldots,n\}$ and
  $\Theta\cle\Theta\Theta$.
  By a straightforward induction, we deduce that
  $\Theta\cle a$ for all $a\in@(\{g_1,\ldots,g_n\})$
  (recall that the latter set is generated from $g_1,\ldots,g_n$ by
  application), and thus $\Theta\cle a$ for all 
  $a\in{\uparrow}@(\{g_1,\ldots,g_n\})=S$ (by upwards closure).
  Therefore: $\Theta=\min(S)$ and $S={\uparrow}\{\Theta\}$
  (since $S$ is upwards closed).
  \smallbreak\noindent
  $(2)\limp(3)$\quad
  Let us assume that $S={\uparrow}\{\Theta\}$ for some $\Theta\in S$.
  Let $(\alpha_i)_{i\in I}\in(\A/S)^I$ be a family of equivalence
  classes indexed by an arbitrary set~$I$, and
  $(a_i)_{i\in I}\in\prod_{i\in I}\alpha_i$ a system of
  representatives.
\COUIC{
  \footnote{\label{note:AC}
    In what follows, we shall silently use the axiom of choice
    (AC) whenever necessary.
    Note that here, we only need (AC) to extract a family of
    representatives $(a_i)_{i\in I}\in\prod_{i\in I}\alpha_i$ from a
    family of equivalence classes  $(\alpha_i)_{i\in I}\in(\A/S)^I$,
    which could be avoided by working with Heyting pre-algebras rather
    than with Heyting algebras.
    On the other hand, Heyting pre-algebras have their own technical
    complications, so that we shall stick to Heyting algebras, for the
    sake of readability.}.
}
  Since $\bigl(\bigmeet_{i\in I}a_i\bigr)\cle a_i$ for all $i\in I$,
  we have $\bigl[\bigmeet_{i\in I}a_i\bigr]_{/S}\le_S\alpha_i$ for
  all $i\in I$, hence $\bigl[\bigmeet_{i\in I}a_i\bigr]_{/S}$ is a
  lower bound of $(\alpha_i)_{i\in I}$ in $\A/S$.
  Now, let us assume that $\beta=[b]_{/S}$ is a lower bound of
  $(\alpha_i)_{i\in I}$ in~$\A/S$, which means that $(b\to a_i)\in S$
  for all $i\in I$.
  But since $S={\uparrow}\{\Theta\}$, we have
  $\Theta\cle(b\to a_i)$ for all $i\in I$, hence
  $\Theta\cle\bigl(b\to\bigmeet_{i\in I}a_i\bigr)$,
  so that $\beta=[b]_{/S}\le_S\bigl[\bigmeet_{i\in I}a_i\bigr]_{/S}$.
  Therefore, $\bigl[\bigmeet_{i\in I}a_i\bigr]_{/S}$ is the g.l.b.\ of
  the family $(\alpha_i)_{i\in I}=\bigl([a_i]_{/S}\bigr)_{i\in I}$ in
  $\A/S$.
  This proves that the induced Heyting algebra $(\A/S,{\le_S})$ is
  complete, as well as the desired commutation property.
  \smallbreak\noindent
  $(3)\limp(2)\limp(1)$\quad Let us assume that the Heyting algebra
  $(\A/S,{\le_S})$ is complete, and that the canonical surjection
  $[\,\cdot\,]_{/S}:\A\to\A/S$ commutes with arbitrary meets.
  Letting $\Theta:=\bigmeet S$, we observe that
  $$\bigl[\Theta\bigr]_{/S}~=~
  \biggl[\bigmeet_{a\in S}a\biggr]_{/S}~=~
  \bigwedge_{a\in S}[a]_{/S}~=~[\top]_{/S}\,,$$
  hence $\Theta\in S$.
  Therefore $\Theta=\min(S)$ and $S={\uparrow}\{\Theta\}$ (since
  $S$ is upwards closed), which shows that~$S$ is the principal filter
  generated by~$\Theta$.
  But this implies that $S$ is finitely generated (we obviously have
  $S=\Sep(\{\Theta\})$) and that $\Fork^{\A}\in S$ (by
  Prop.~\ref{p:CharacFilter}).
\end{proof}

\begin{remark}\label{r:RealizCollapse}
  From a categorical perspective, the situation described by
  Prop.~\ref{p:CharacPrincFilter} is particularly important, since it
  characterizes the collapse of realizability to forcing.
  Indeed, we shall see in Section~\ref{ss:CharacForcingTriposes}
  (Theorem~\ref{th:CharacForcingTriposes}) that the tripos induced by
  an implicative algebra $(\A,{\cle},{\to},S)$
  (Section~\ref{ss:ConstrImpTripos}) is isomorphic to a forcing tripos
  (induced by some complete Heyting algebra) if and only if the
  separator~$S\subseteq\A$ is a principal filter of~$\A$, that is: if
  and only if the separator~$S$ is finitely generated and contains the
  non-deterministic choice operator $\Fork^{\A}$.
\end{remark}

\subsection{On the interpretation of existential
  quantification as a join}\label{ss:ExJoin}

In Section~\ref{ss:InterpFOL}, we have seen that existential
quantifications cannot be interpreted by (infinitary) joins in the
general framework of implicative structures.
(We shall actually present a counter-example at the end of this
section.)
Using the material presented in Section~\ref{ss:SeparFilter} above,
we shall now study the particular class of implicative structures
where existential quantifications are naturally interpreted by joins.

Formally, we say that an implicative structure $\A=(\A,{\cle},{\to})$
is \emph{compatible with joins} when it fulfills the additional axiom
$$\bigmeet_{\!\!a\in A\!\!}(a\to b)~=~
\biggl(\bigjoin_{\!\!a\in A\!\!}a\biggr)\to b$$
for all subsets $A\subseteq\A$ and for all $b\in\A$.
(Note that the converse relation $\cge$ holds in any implicative
structure, so that only the direct relation $\cle$ matters.)

This axiom obviously holds in any complete Heyting (or Boolean)
algebra, as well as in any implicative structure induced by a total
combinatory algebra $(P,{\,\cdot\,},\mathtt{k},\mathtt{s})$ 
(Section~\ref{sss:CaseRealizJ}).
On the other hand, the implicative structures induced by classical
realizability (Section~\ref{sss:CaseRealizK}) are in general
\emph{not} compatible with joins, as we shall see below.

When an implicative structure $\A=(\A,{\cle},{\to})$ is compatible
with joins, the existential quantifier can be interpreted as a join
$$\bigexists_{i\in I}a_i~:=~\bigjoin_{i\in I}a_i$$
since the corresponding elimination rule is directly given by the
subtyping relation
$$\bigmeet_{\!\!i\in I\!\!}(a_i\to b)~\cle~
\biggl(\bigjoin_{\!\!i\in I\!\!}a_i\biggr)\to b\,.$$
In this situation, we can also observe many simplifications at the level
of the defined connectives $\times$ and $+$:
\begin{proposition}
  If an implicative structure $\A=(\A,{\cle},{\to})$ is
  compatible with joins, then for all $a\in\A$, we have:
  $$\begin{array}{r@{~~}c@{~~}l@{\qquad\qquad}r@{~~}c@{~~}l}
    \bot\to a&=&\top & \Por^{\A}&=&\top \\
    a\times\bot&=&\top\to\bot & a+\bot&=&(\Lam{xy}{x\,a})^{\A} \\
    \bot\times a&=&\top\to\bot & \bot+a&=&(\Lam{xy}{y\,a})^{\A} \\
  \end{array}$$
\end{proposition}

\begin{proof*}
  Indeed, we have:
  \begin{itemize}
  \item[$\bullet$] $\bot\to a=(\bigjoin\varnothing)\to a
    =\bigmeet\varnothing=\top$, from the compatibility with joins.
  \item[$\bullet$] $\Por^{\A}=(\bot\to\top\to\bot)\meet
    (\top\to\bot\to\bot)=\top\meet(\top\to\top)=\top$.
  \item[$\bullet$] $a\times\bot=\bigmeet_c((a\to\bot\to c)\to c)
    =\bigmeet_c(\top\to c)=\top\to\bot$.
  \item[$\bullet$] $\bot\times a=\bigmeet_c((\bot\to a\to c)\to c)
    =\bigmeet_c(\top\to c)=\top\to\bot$.
  \item[$\bullet$] By semantic typing, we have:
    \begin{center}
      $(\Lam{xy}{x\,a})^{\A}~\cle~
      \bigmeet_{c}((a\to c)\to(\bot\to c)\to c)~=~a+\bot\,.$
    \end{center}
    And conversely:
    \begin{center}
      $\begin{array}{@{}r@{~~}c@{~~}l}
        a+\bot&=&\ds\bigmeet_{c}((a\to c)\to(\bot\to c)\to c)
        ~=~\bigmeet_{c}((a\to c)\to\top\to c)\\
        &\cle&\ds\bigmeet_{d,e}((a\to da)\to e\to da)
        ~\cle~\bigmeet_{d,e}(d\to e\to da)~=~(\Lam{xy}{x\,a})^{\A}\\
      \end{array}$
    \end{center}
  \item[$\bullet$] The equality $\bot+a=(\Lam{xy}{y\,a})^{\A}$ is
    proved similarly.\hfill\usebox{\proofbox}
  \end{itemize}
\end{proof*}

In particular, we observe a trivialization of the parallel `or':
$\Por^{\A}=\top$, so that by Prop.~\ref{p:PorForkEquivK}, we get
$\Fork^{\A}\in\SK(\A)$.
Therefore, by Prop.~\ref{p:CharacFilter}, it is clear that:
\begin{proposition}\label{p:BadRealizK}
  If an implicative structure $\A=(\A,{\to},{\cle})$ is compatible
  with joins, then all its classical separators are filters.
\end{proposition}

Of course, this situation is highly undesirable in classical
realizability (see Remark~\ref{r:RealizCollapse} above), and this
explains why classical realizability is not and cannot be compatible
with joins in general (except in the degenerate case of forcing).

\begin{remark}[The model of threads]\label{r:ModelThreads}
  In~\cite{Kri12}, Krivine constructs a model of
  $\text{ZF}+\text{DC}$ from a particular abstract Krivine structure
  (see Section~\ref{sss:CaseRealizK}), called the \emph{model of
    threads}.
  This particular AKS is defined in such a way that it is consistent,
  while providing a proof-like term $\theta\in\mathrm{PL}$ that
  realizes the \emph{negation} of the parallel `or':
  \begin{center}
    $\theta~\Vdash~
    \lnot((\bot\to\top\to\bot)\meet(\top\to\bot\to\bot))$.
  \end{center}
  In the induced classical implicative algebra $(\A,{\cle},{\to},S)$
  (Section~\ref{sss:ImpAlgRealizK}), we thus have $\bot\notin S$
  and $\lnot\Por^{\A}\in S$.
  Hence $\Por^{\A}\notin S$ and thus $\Fork^{\A}\notin S$
  (by Prop.~\ref{p:PorForkEquivK}), so that~$S$ is not a filter
  (Prop.~\ref{p:CharacFilter}).
  From Prop.~\ref{p:BadRealizK} (by contraposition), it is then clear
  that the underlying implicative structure $(\A,{\cle},{\to})$ is not
  compatible with joins.
\end{remark}

\section{The implicative tripos}
\label{s:Tripos}

In Section~\ref{ss:InducedHA}, we have seen that any implicative
algebra $(\A,{\cle},{\to},S)$ induces a Heyting algebra
$(\A/S,{\le_S})$ that intuitively captures the corresponding
logic, at least at the propositional level.
In this section, we shall see that this construction more generally
gives rise to a ($\Set$-based) tripos, called an \emph{implicative
  tripos}.
For that, we first need to present some constructions on implicative
structures and on separators.

\subsection{Product of implicative structures}
\label{ss:ProdImpStruct}

Let $(\A_i)_{i\in I}=(\A_i,{\cle_i},{\to_i})_{i\in I}$ be a family of
implicative structures indexed by an arbitrary set~$I$.
The Cartesian product $\A:=\prod_{i\in I}\A_i$ is naturally equipped
with the ordering $({\cle})\subseteq\A^2$ and the implication
$({\to}):\A^2\to\A$ that are defined componentwise:
$$\begin{array}{r@{~~}c@{~~}l}
  (a_i)_{i\in I}\cle(b_i)_{i\in I}
  &:\liff& \forall i\in I,~a_i\cle_ib_i\\[3pt]
  (a_i)_{i\in I}\to(b_i)_{i\in I}
  &:=& (a_i\to_ib_i)_{i\in I}\\
\end{array}\eqno\begin{tabular}{r@{}}
(product ordering)\\[3pt]
(product implication)\\
\end{tabular}$$
It is straightforward to check that:
\begin{proposition} The triple $(\A,{\cle},{\to})$ is an
  implicative structure.
\end{proposition}

In the product implicative structure
$(\A,{\cle},{\to})=\prod_{i\in I}\A_i$, the defined constructions
$\lnot a$ (negation), $a\times b$ (conjunction), $a+b$ (disjunction),
$ab$ (application), $\cc^{\A}$ (Peirce's law) and $\Fork^{\A}$
(non-deterministic choice) are naturally characterized componentwise:
\begin{proposition}\label{p:ProdOps}
  For all $a,b\in\A=\prod_{i\in I}\A_i$, we have:
  $$\begin{array}{r@{~~}c@{~~}l@{\qquad}r@{~~}c@{~~}l@{\qquad}r@{~~}c@{~~}l}
    \lnot a&=&(\lnot a_i)_{i\in I} &
    a\times b&=&(a_i\times b_i)_{i\in I} &
    a+b&=&(a_i+b_i)_{i\in I} \\
    ab&=&(a_ib_i)_{i\in I} &
    \cc^{\A}&=&\bigl(\cc^{\A_i}\bigr)_{i\in I} &
    \Fork^{\A}&=&\bigl(\Fork^{\A_i}\bigr)_{i\in I} \\
  \end{array}$$
\end{proposition}

\begin{proof}
  Given $a,b\in\A$, we have:
  $$\begin{array}{r@{~~}c@{~~}l}
    a\times b&=&\ds\bigmeet_{\!\!c\in\A\!\!}((a\to b\to c)\to c)
    ~=~\bigmeet_{\!\!c\in\A\!\!}
    \Bigl((a_i\to b_i\to c_i)\to c_i\Bigr)_{i\in I}\\
    &=&\ds\biggl(\bigmeet_{\!c\in\A_i\!}
    ((a_i\to b_i\to c)\to c)\biggr)_{i\in I}
    ~=~\bigl(a_i\times b_i\bigr)_{i\in I}\\
  \end{array}$$
  $$\begin{array}{r@{~~}c@{~~}l}
    ab&=&\ds\bigmeet\bigl\{c\in\A:a\cle(b\to c)\bigr\}
    ~=~\bigmeet\prod_{i\in I}
    \bigl\{c\in\A_i:a_i\cle(b_i\to c)\bigr\}\\
    &=&\ds\Bigl(\bigmeet
    \bigl\{c\in\A_i:a_i\cle(b_i\to c)\bigr\}\Bigr)_{i\in I}
    ~=~\bigl(a_ib_i\bigr)_{i\in I}\\
  \end{array}$$
  The other equalities are proved similarly.
\end{proof}

\begin{proposition}\label{p:ProdLam}
  For all pure $\lambda$-terms $t(x_1,\ldots,x_k)$ with
  free variables $x_1,\ldots,x_k$ and for all parameters
  $a_1,\ldots,a_k\in\A=\prod_{i\in I}\A_i$, we have:
  $$t(a_1,\ldots,a_k)^{\A}~=~\Bigl(
  t\bigl(a_{1,i},\ldots,a_{k,i}\bigr)^{\A_i}
  \Bigr)_{i\in I}$$
\end{proposition}

\begin{proof*}
  By structural induction on the term $t(x_1,\ldots,x_k)$.
  The case of a variable is obvious, the case of an application
  follows from the equality $ab=(a_ib_i)_{i\in I}$, so that we only
  treat the case where
  $t(x_1,\ldots,x_k)=\Lam{x_0}{t_0(x_0,x_1,\ldots,x_k)}$.
  In this case, we have:
  $$\begin{array}[b]{r@{~~}c@{~~}l}
    t(a_1,\ldots,a_k)^{\A}
    &=&\ds\bigl(\Lam{x_0}{t_0(x_0,a_1,\ldots,a_k)}\bigr)^{\A}\\
    &=&\ds\bigmeet_{\!\!a_0\in\A\!\!}
    (a_0\to t_0(a_0,a_1,\ldots,a_k)^{\A})\\
    &=&\ds\bigmeet_{\!\!a_0\in\A\!\!}\Bigl(a_{0,i}\to_i
    t_0\bigl(a_{0,i},a_{1,i},\ldots,a_{k,i}\bigr)^{\A_i}\Bigr)_{i\in{I}}
    \rlap{\hskip 26mm(by IH)}\\
    &=&\ds\biggl(\bigmeet_{\!\!a_0\in\A_i\!\!}\Bigl(a_0\to_i
    t_0\bigl(a_0,a_{1,i},\ldots,a_{k,i}\bigr)^{\A_i}
    \Bigr)\biggr)_{i\in I}\\
    &=&\ds\Bigl(\bigl(\Lam{x_0}{
      t_0(x_0,a_{1,i},\ldots,a_{k,i})}\bigr)^{\A_i}\Bigr)_{i\in I}
    ~=~\Bigl(t\bigl(a_{1,i},\ldots,a_{k,i}\bigr)^{\A_i}\Bigr)_{i\in{I}}\\
  \end{array}\eqno\usebox{\proofbox}$$
\end{proof*}

\subsubsection{Product of separators}\label{sss:ProdSep}
Given a family of separators $(S_i\subseteq\A_i)_{i\in I}$, it is
clear that the Cartesian product $S=\prod_{i\in I}S_i$ is also a
separator of $\A=\prod_{i\in I}\A_i$.
In the product separator $S=\prod_{i\in I}S_i$, the relation of
entailment $a\ent_Sb$ and the corresponding equivalence $a\tnent_Sb$
are characterized by:
$$\begin{array}{rcl}
  a\ent_Sb&\liff&\forall i\in I,~a_i\ent_{S_i}b_i\\
  a\tnent_Sb&\liff&\forall i\in I,~a_i\tnent_{S_i}b_i\\
\end{array}\eqno\hskip-30mm(\text{for all}~a,b\in\A)$$
For each index $i\in I$, the corresponding projection
$\pi_i:\A\to\A_i$ factors into a map
$$\begin{array}{r@{~~}c@{~~}l}
  \tilde{\pi}_i~~:~~\A/S&\to&\A_i/S_i\\
  {[a]}_{/S}&\mapsto&[a_i]_{/S_i}\\
\end{array}$$
that is obviously a morphism of Heyting algebras (from
Prop.~\ref{p:InducedHA} and~\ref{p:ProdOps}).
In this situation, we immediately get the factorization\ \
$\A/S~\cong~\prod_{i\in I}(\A_i/S_i)$,\ \ since:
\begin{proposition}\label{p:Factorization}
  The map
  $$\<\tilde{\pi}_i\>_{i\in I}~:~
  \A/S~\to~\prod_{i\in I}(\A_i/S_i)$$
  is an isomorphism of Heyting algebras.
\end{proposition}

\begin{proof}
  For all $a,b\in\A$, we have
  $$[a]\le_S[b]~~\liff~~(a\to b)\in S~~\liff~
  (\forall i\,{\in}\,I)~(a_i\to b_i)\in S_i~~\liff~~
  (\forall i\,{\in}\,I)~[a_i]\le_{S_i}[b_i]$$
  which proves that the map
  $\<\tilde{\pi}_i\>_{i\in I}:\A/S\to\prod_{i\in I}(\A_i/S_i)$ is an
  embedding of the poset $(\A/S,{\le_S})$ into the product poset
  $\prod_{i\in I}(\A_i/S_i,{\le_{S_i}})$.
  Moreover, the map $\<\tilde{\pi}_i\>_{i\in I}$ is clearly surjective
  (from the axiom of choice);
\COUIC{
  \footnote{See the discussion of footnote~\ref{note:AC}
    p.~\pageref{note:AC}.})
}
  therefore, it is an isomorphism of posets, and thus an isomorphism
  of Heyting algebras.
\end{proof}

\subsection{The uniform power separator}
\label{ss:UnifPowSep}

Let $\A=(\A,{\cle},{\to})$ be a fixed implicative structure.
For each set $I$, we write
$$\A^I~=~(\A^I,{\cle}^I,{\to}^I)~:=~\prod_{i\in I}(\A,{\cle},{\to})$$
the corresponding power implicative structure, which is a particular
case of the product presented in Section~\ref{ss:ProdImpStruct}
above.
Each separator $S\subseteq\A$ induces two separators in~$\A^I$:
\begin{itemize}
\item The \emph{power separator}
  $S^I:=\prod_{i\in I}S\subseteq\A^I$.
\item The \emph{uniform power separator} $S[I]\subseteq\A$, that is
  defined by:
  $$S[I]~:=~\{a\in\A:\exists s\in S,~\forall i\in I,~s\cle a_i\}
  ~=~{\uparrow}\mathrm{img}(\delta_I)\,,$$
  where $\delta_I:\A\to\A^I$ is defined by $\delta(a)=(i\mapsto a)$
  for all $a\in\A$.
\end{itemize}
From the definition, it is clear that $S[I]\subseteq S^I\subseteq\A$.
The converse inclusion $S^I\subseteq S[I]$ does not hold in general,
and we easily check that:
\begin{proposition}\label{p:CharacSepEqual}
  For all separators $S\subseteq\A$, the following are equivalent:
  \begin{enumerate}[(99)]
  \item[(1)] $S[I]=S^I$.
  \item[(2)] $S$ is closed under all $I$-indexed meets.
  \end{enumerate}
\end{proposition}

\begin{proof}
  $(1)\limp(2)$\quad
  Let $(a_i)_{i\in I}$ be an $I$-indexed family of elements of~$S$,
  that is: an element of $S^I$.
  By $(1)$ we have $(a_i)_{i\in I}\in S[I]$, so that there is
  $s\in S$ such that $s\cle a_i$ for all $i\in I$.
  Therefore $s\cle\bigl(\bigmeet_{i\in I}a_i\bigr)\in S$
  (by upwards closure).
  \smallbreak\noindent
  $(2)\limp(1)$\quad
  Let $(a_i)_{i\in I}\in S^I$.
  By $(2)$ we have $s:=\bigl(\bigmeet_{i\in I}a_i\bigr)\in S$,
  and since $s\cle a_i$ for all $i\in I$, we get that
  $(a_i)_{i\in I}\in S[I]$ (by definition).
  Hence $S^I=S[I]$.
\end{proof}

Thanks to the notion of uniform separator, we can also characterize
the intuitionistic and classical cores (Section~\ref{ss:SeparGen}) of
the power implicative structure $\A^I$:
\begin{proposition}
  $\SJ(\A^I)=\SJ(\A)[I]$ and
  $\SK(\A^I)=\SK(\A)[I]$.
\end{proposition}

\begin{proof}
  Recall that:\quad
  $\begin{array}[t]{r@{~~}c@{~~}l}
    \SJ(\A)&=&{\uparrow}\bigl\{(t)^{\A}:
    t~\text{closed $\lambda$-term}\bigr\}\\[3pt]
    \SJ(\A^I)&=&{\uparrow}\bigl\{(t)^{\A^I}:
    t~\text{closed $\lambda$-term}\bigr\}\\[3pt]
    \SJ(\A)[I]&=&\bigl\{a\in\A^I:
    \exists s\in\SJ(\A),~\forall i\in I,~s\cle a_i\bigr\}\,.\\
  \end{array}$\smallbreak\noindent
  Since $\SJ(\A^I)$ is the smallest separator of~$\A^I$, we have
  $\SJ(\A^I)\subseteq\SJ(\A)[I]$.
  Conversely, take $a\in\SJ(\A)[I]$.
  By definition, there is $s\in\SJ(\A)$ such that $s\cle a_i$ for all
  $i\in I$.
  And since $s\in\SJ(\A)$, there is a closed $\lambda$-term $t$ such
  that $(t)^{\A}\cle s$, hence $(t)^{\A}\cle a_i$ for all $i\in I$.
  From Prop.~\ref{p:ProdLam}, we deduce that
  $(t)^{\A^I}=\bigl((t)^{\A}\bigr)_{i\in I}\cle(a_i)_{i\in I}$
  (in~$\A^I$), hence $(a_i)_{i\in I}\in\SJ(\A^I)$.
  The equality $\SK(\A^I)=\SK(\A)[I]$ is proved similarly, using
  closed $\lambda$-terms with~$\cc$ instead of pure
  $\lambda$-terms.
\end{proof}

In the rest of this section, we shall see that, given a separator
$S\subseteq\A$, the correspondence\ \ $I\mapsto\A^I/S[I]$\ \ (from
unstructured sets to Heyting algebras) is functorial, and
actually constitutes a tripos.

\subsection{Triposes}
\label{ss:Tripos}

\subsubsection{The category of Heyting algebras}
\label{sss:CatHA}

Given two Heyting algebras~$H$ and~$H'$, a function $F:H\to H'$ is
called a \emph{morphism of Heyting algebras} when
$$\begin{array}{r@{~~}c@{~~}l@{\qquad}r@{~~}c@{~~}l}
  F(a\land_Hb)&=&F(a)\land_{H'}F(b) & F(\top_H)&=&\top_{H'} \\
  F(a\lor_Hb)&=&F(a)\lor_{H'}F(b) & F(\bot_H)&=&\bot_{H'} \\
  F(a\to_Hb)&=&F(a)\to_{H'}F(b) \\
\end{array}\eqno(\text{for all}~a,b\in H)$$
(In other words, a morphism of Heyting algebras is a morphism of
bounded lattices that also preserves Heyting's implication.
Note that such a function is always monotonic.)

The category of Heyting algebras (notation: \HA) is the category whose
objects are the Heyting algebras and whose arrows are the morphisms of
Heyting algebras; it is a (non-full) subcategory of the category of
posets (notation: \Pos).
This category also enjoys some specific properties that will be useful
in the following:
\begin{enumerate}[(99)]
\item[(1)] An arrow is an isomorphism in~\HA{} if and only if it is an
  isomorphism in~\Pos.
\item[(2)] Any injective morphism of Heyting algebras $F:H\to H'$ is
  also an embedding of posets, in the sense that:\quad 
  $a\le b$~~iff~~ $F(a)\le F(b)$\quad (for all $a,b\in H$).
\item[(3)] Any bijective morphism of Heyting algebras is also an
  isomorphism.
\end{enumerate}

\subsubsection{$\Set$-based triposes}

In this section, we recall the definition of $\Set$-based triposes,
such as initially formulated by Hyland, Johnstone and Pitts
in~\cite{HJP80}. 
For the general definition of triposes---where the base category
$\Set$ can be replaced by an arbitrary Cartesian category---, see for
instance~\cite{Pit81,Pit01}.

\begin{definition}[$\Set$-based tripos]\label{d:Tripos}
  A \emph{$\Set$-based tripos} is a functor $\P:\Set^{\op}\to\HA$ that
  fulfills the following three conditions:
  \begin{enumerate}[(99)]
  \item[(1)] For each function $f:I\to J$, the corresponding map
    $\P{f}:\P{J}\to\P{I}$ has left and right adjoints in~$\Pos$,
    that are monotonic maps $\exists{f},\forall{f}:\P{I}\to\P{J}$ such
    that
    $$\begin{array}{rcl}
      \exists{f}(p)\le q&\liff&p\le\P{f}(q) \\[3pt]
      q\le\forall{f}(p)&\liff&\P{f}(q)\le p \\
    \end{array}\eqno(\text{for all}~p\in\P{I},~q\in\P{J})$$
  \item[(2)] Beck-Chevalley condition. Each pullback square in $\Set$
    (on the left-hand side) induces the following two commutative
    diagrams in~$\Pos$ (on the right-hand side):
    $$\begin{array}{@{}c@{}}
      \xymatrix{
        I\pullback{6pt}\ar[r]^{f_1}\ar[d]_{f_2}& I_1\ar[d]^{g_1} \\
        I_2\ar[r]_{g_2} & J \\
      }\\
    \end{array}\qquad{\limp}\qquad
    \begin{array}{c@{\qquad}c}
      \xymatrix{
        \P{I}\ar[r]^{\exists{f_1}}& \P{I_1} \\
        \P{I_2}\ar[u]^{\P{f_2}}\ar[r]_{\exists{g_2}} &
        \P{J}\ar[u]_{\P{g_1}} \\
      } &
      \xymatrix{
        \P{I}\ar[r]^{\forall{f_1}}& \P{I_1} \\
        \P{I_2}\ar[u]^{\P{f_2}}\ar[r]_{\forall{g_2}} &
        \P{J}\ar[u]_{\P{g_1}} \\
      }\\
    \end{array}$$
    That is:\quad
    $\exists{f_1}\circ\P{f_2}~=~\P{g_1}\circ\exists{g_2}$\quad
    and\quad $\forall{f_1}\circ\P{f_2}~=~\P{g_1}\circ\forall{g_2}$.
  \item[(3)] The functor $\P:\Set^{\op}\to\HA$ has a \emph{generic
    predicate}, that is: a predicate $\Tr\in\P\Prop$ (for
    some set~$\Prop$) such that for all sets~$I$, the following map
    is surjective:
    $$\begin{array}{r@{~{}~}c@{~{}~}l}
      \Prop^I&\to&\P{I} \\
      f&\mapsto&\P{f}(\Tr) \\
    \end{array}$$
  \end{enumerate}
\end{definition}

\begin{remarks}[Intuitive meaning of the
    definition]\label{r:TriposIntuitions}
  Intuitively, each $\Set$-based tripos $\P:\Set^{\op}\to\HA$
  describes a particular model of intuitionistic higher-order logic,
  in which higher-order types are modeled by sets.
  In this framework:
  \smallbreak\noindent
  (1)~~The functor $\P:\Set^{\op}\to\C$ associates to each `type'
  $I\in\Set$ the poset $\P{I}$ of all \emph{predicates} over~$I$.
  The ordering on $\P{I}$ represents \emph{inclusion} of predicates
  (in the sense of the considered model), whereas equality represents
  \emph{extensional equality} (or \emph{logical equivalence}).
  In what follows, it is convenient to think that predicates
  $p,q,\ldots\in\P{I}$ represent abstract formulas $p(x),q(x),\ldots$
  depending on a variable~$x:I$, so that
  $$\begin{array}[b]{r@{\qquad}c@{\qquad}l}
    p\le q&\text{means}&(\forall x\in I)(p(x)\limp q(x))\\[3pt]
    p=q&\text{means}&(\forall x\in I)(p(x)\liff q(x))\,.\\
  \end{array}\leqno\text{whereas}$$
  The fact that $\P{I}$ is a Heyting algebra simply expresses that the
  predicates over~$I$ can be composed using all the connectives of
  intuitionistic logic, and that these operations fulfill the laws of
  intuitionistic propositional logic.
  \smallbreak\noindent
  (2)~~The functoriality of~$\P$ expresses that each function
  $f:I\to J$ induces a \emph{substitution map} $\P{f}:\P{J}\to\P{I}$,
  that intuitively associates to each predicate $q\in\P{J}$ its
  ``preimage'' $\P{f}(q)=\text{``}q\circ f\text{''}\in\P{I}$.
  Again, if we think that the predicate $q\in\P{J}$ represents an
  abstract formula $q(y)$ depending on a variable $y:J$, then the
  predicate $\P{f}(q)$ represents the substituted formula
  $q(y)\{y:=f(x)\}\equiv q(f(x))$ (that now depends on $x:I$).
  The fact that the substitution map $\P{f}:\P{J}\to\P{I}$ is a
  morphism of Heyting algebras expresses that substitution commutes
  with all the logical connectives.
  \smallbreak\noindent
  (3)~~Given a function $f:I\to J$, the left and right adjoints
  $\exists{f},\forall{f}:\P{I}\to\P{J}$ represent existential and
  universal quantifications along the function $f:I\to J$.
  By this, we mean that if a predicate $p\in\P{I}$ represents a
  formula $p(x)$ (depending on $x:I$), then
  $$\begin{array}[b]{r@{\qquad}c@{\qquad}l}
    \exists{f}(p)&\text{represents the formula}&
    (\exists x:I)(f(x)=y\land p(x))\\[3pt]
    \forall{f}(p)&\text{represents the formula}&
    (\forall x:I)(f(x)=y\limp p(x))\\
  \end{array}\leqno\text{whereas}$$
  (where both right-hand side formulas depend on $y:J$).
  Both `quantified' predicates $\exists{f}(p),\forall{f}(p)\in\P{J}$
  are characterized by the adjunctions
  $$\begin{array}[b]{r@{\qquad}c@{\qquad}l}
    \exists{f}(p)\le q&\text{iff}&p\le\P{f}(q) \\[3pt]
    q\le\forall{f}(p)&\text{iff}&\P{f}(q)\le p \\
  \end{array}\leqno\text{and}$$
  (for all $q\in\P{J}$), which express the logical equivalences
  $$\begin{array}[b]{r@{\quad}c@{\quad}l}
    (\forall y:J)[(\exists x:I)(f(x)=y\land p(x))~\limp~q(y)]
    &\liff&(\forall x:I)[p(x)~\limp~q(f(x))] \\[3pt]
    (\forall y:J)[q(y)~\limp~(\forall x:I)(f(x)=y\limp p(x))]
    &\liff&(\forall x:I)[q(f(x))~\limp~p(x)]\,. \\
  \end{array}\leqno\text{and}$$
  Note that unlike the substitution map $\P{f}:\P{J}\to\P{I}$, the two
  adjoints $\exists{f},\forall{f}:\P{I}\to\P{J}$ are only monotonic
  maps (i.e.\ arrows in $\Pos$); they are not morphisms of Heyting
  algebras in general.
  (Quantifiers do not commute will all connectives!)
  Nevertheless, left adjoints commute with joins (and $\bot$),
  whereas right adjoints commute with meets (and $\top$):
  $$\begin{array}[m]{r@{~~}c@{~~}l@{\qquad\quad}r@{~~}c@{~~}l}
    \exists{f}(p_1\lor p_2)&=&
    \exists{f}(p_1)\lor\exists{f}(p_2)&
    \exists{f}(\bot)&=&\bot\\
    \forall{f}(p_1\land p_2)&=&
    \forall{f}(p_1)\land\forall{f}(p_2)&
    \forall{f}(\top)&=&\top\\
  \end{array}\eqno(\text{for all}~p_1,p_2\in\P{I})$$
  Using left adjoints, we can also define the equality predicate
  $$({=}_I)~:=~\exists{\delta_I}(\top_I)~\in~\P(I\times I)
  \eqno(\text{for each}~I\in\Set)$$
  writing $\delta_I:I\to I\times I$ the duplication function
  and~$\top_I$ the top element of $\P{I}$.
  From what precedes, it should be clear to the reader that this
  predicate represents the formula
  $(\exists x\in I)(\delta(x)=(x_1,x_2)\land\top)$ (depending on
  $x_1,x_2:I$), that is equivalent to $x_1=x_2$.
  \smallbreak\noindent
  (4)~~The Beck-Chevalley condition expresses a property of
  commutation between substitution and quantifications.
  It is typically used with pullback squares of the form
  $$\xymatrix@C=36pt{
    I\times K\pullback{6pt}
    \ar[r]^{\pi_{I,K}}\ar[d]_{f\times\id_K}&I\ar[d]^{f}\\
    J\times K\ar[r]_{\pi_{J,K}}& J\\
  }$$
  where the adjoints
  $\exists\pi_{I,K},\forall\pi_{I,K}:\P(I\times K)\to\P{I}$ and
  $\exists\pi_{J,K},\forall\pi_{J,K}:\P(J\times K)\to\P{J}$
  represent `pure' quantifications over an abstract variable $z:K$
  (in the contexts $x:I$ and $y:J$, respectively).
  In this case, the induced equalities
  $$\begin{array}{c@{\quad}c@{\quad}c}
    \exists\pi_{I,K}\circ\P(f\times\id_K)=\P{f}\circ\exists\pi_{J,K}&
    \text{and}&
    \forall\pi_{I,K}\circ\P(f\times\id_K)=\P{f}\circ\forall\pi_{J,K}\\
    \noalign{\medskip}
    \xymatrix@C=48pt{
      \P(I\times K)\ar[r]^{\exists\pi_{I,K}}&\P{I}\\
      \P(J\times K)\ar[r]_{\exists\pi_{J,K}}\ar[u]^{\P(f\times\id_K)}&
      \P{J}\ar[u]_{\P{f}}\\
    } &&
    \xymatrix@C=48pt{
      \P(I\times K)\ar[r]^{\forall\pi_{I,K}}&\P{I}\\
      \P(J\times K)\ar[r]_{\forall\pi_{J,K}}\ar[u]^{\P(f\times\id_K)}&
      \P{J}\ar[u]_{\P{f}}\\
    }\\
  \end{array}$$
  express for each predicate $q\in\P(J\times K)$ the logical
  equivalences
  $$\begin{array}[b]{l}
    (\forall x:I)[
      (\exists z:K)(q(y,z)\{y:=f(x),z:=z\})~\liff~
      ((\exists z:K)q(y,z))\{y:=f(x)\}] \\[3pt]
    (\forall x:I)[
      (\forall z:K)(q(y,z)\{y:=f(x),z:=z\})~\liff~
      ((\forall z:K)q(y,z))\{y:=f(x)\}] \\
  \end{array}\leqno\text{and}$$
  describing the behavior of substitution w.r.t. quantifiers.
  \smallbreak\noindent
  (5)~~The Beck-Chevalley condition requires that the diagrams
  $$\begin{array}{c}
    \xymatrix{
      \P{I}\ar[r]^{\exists{f_1}}& \P{I_1} \\
      \P{I_2}\ar[u]^{\P{f_2}}\ar[r]_{\exists{g_2}} &
      \P{J}\ar[u]_{\P{g_1}} \\
    }
  \end{array}\qquad\text{and}\qquad\begin{array}{c}
    \xymatrix{
      \P{I}\ar[r]^{\forall{f_1}}& \P{I_1} \\
      \P{I_2}\ar[u]^{\P{f_2}}\ar[r]_{\forall{g_2}} &
      \P{J}\ar[u]_{\P{g_1}} \\
    }
  \end{array}$$
  commute for all pullback squares\quad
  $\begin{array}{c}
    \xymatrix@C=18pt@R=18pt{
      I\pullback{6pt}\ar[r]^{f_1}\ar[d]_{f_2}&I_1\ar[d]^{g_1}\\
      I_2\ar[r]_{g_2}&J\\
    }
  \end{array}$\quad in~$\Set$.\\
  However, both commutation properties are equivalent up to the
  symmetry w.r.t.\ the diagonal (by exchanging the indices~$1$ and~$2$
  in the above pullback square).
  By this, we mean that the $\exists$-diagram associated to the
  initial pullback square commutes if and only if the
  $\forall$-diagram associated to the symmetric pullback square
  (obtained by exchanging the indices~$1$ and~$2$) commutes, that is:
  $$\exists{f_1}\circ\P{f_2}=\P{g_1}\circ\exists{g_2}
  \qquad\text{iff}\qquad
  \forall{f_2}\circ\P{f_1}=\P{g_2}\circ\forall{g_1}\,.$$
  (This equivalence is easily derived from the adjunctions defining
  the monotonic maps $\exists{f}$ and $\forall{f}$.)
  So that in order to prove the Beck-Chevalley condition, we only need
  to check that all $\exists$-diagrams commute, or that all
  $\forall$-diagrams commute.
  \smallbreak\noindent
  (6)~~Finally, the set $\Prop$ represents the \emph{type of
    propositions}, whereas the generic predicate $\Tr\in\P\Prop$
  represents the formula asserting that a given proposition is true.
  Thanks to this predicate, we can turn any \emph{functional
    proposition} into a predicate via the map
  $$\begin{array}{r@{~~}c@{~~}l}
    \Prop^I&\to&\P{I} \\
    f&\mapsto&\P{f}(\Tr) \\
  \end{array}\eqno(I\in\Set)$$
  We require that this map is surjective for all sets~$I$, thus
  ensuring that each predicate $p\in\P{I}$ is represented by (at
  least) a functional proposition $f\in\Prop^I$.
\end{remarks}

\begin{remark}[Non-uniqueness of the generic predicate]
  It is important to observe that in a $\Set$-based tripos~$\P$, the
  generic predicate is never unique.
  \smallbreak\noindent
  (1)~~Indeed, given a generic predicate $\Tr\in\P\Prop$ and a
  surjection $h:\Prop'\to\Prop$, we can always construct another
  generic predicate $\Tr'\in\P\Prop'$, letting
  $\Tr'=\P{h}(\Tr)$%
  \footnote{To prove that $\Tr'\in\P\Prop'$ is another generic
    predicate of the tripos~$\P$, we actually need to pick a right
    inverse of $h:\Prop'\to\Prop$, which exists by~(AC).
    Without (AC), the same argument works by replacing `surjective'
    with `having a right inverse'.}.
  \smallbreak\noindent
  (2)~~More generally, if $\Tr\in\P\Prop$ and $\Tr'\in\P\Prop'$ are
  two generic predicates of the same tripos~$\P$, then there always
  exist two conversion maps $h:\Prop'\to\Prop$ and $h':\Prop\to\Prop'$
  such that $\Tr'=\P{h}(\Tr)$ and $\Tr=\P{h'}(\Tr)$.
  Intuitively, the sets $\Prop$ and $\Prop'$ represent distinct
  implementations of the type of propositions (they do not need to
  have the same cardinality), whereas the conversion functions
  $h:\Prop'\to\Prop$ and~$h':\Prop\to\Prop'$ implement the
  corresponding changes in representation.
\end{remark}

\begin{example}[Forcing tripos]\label{ex:ForcingTripos}
  Given a complete Heyting algebra $(H,{\le})$, the functor
  $\P:\Set^{\op}\to\HA$ defined for all $I,J\in\Set$ and $f:I\to J$
  by
  $$\P{I}~:=~H^I\qquad\text{and}\qquad
  \P{f}~:=~(h\mapsto h\circ f)~:~H^J\to H^I$$
  is a $\Set$-based tripos, in which left and right adjoints
  $\exists{f},\forall{f}:\P{I}\to\P{J}$ are given by
  $$\exists{f}(p)~:=~
  \biggl(\bigvee_{i\in f^{-1}(j)}\!\!\!\!\!p_i\biggr)_{j\in J}
  \qquad\text{and}\qquad
  \forall{f}(p)~:=~
  \biggl(\bigwedge_{i\in f^{-1}(j)}\!\!\!\!\!p_i\biggr)_{j\in J}$$
  (for all $f:I\to J$ and $p\in\P{I}=H^I$), and whose generic
  predicate $(\Prop,\Tr)$ is given by
  $$\Prop~:=~H\qquad\text{and}\qquad
  \Tr~:=~\id_{H}~\in~\P{\Prop}\,.$$
  Such a tripos is called a \emph{Heyting tripos}, or a
  \emph{forcing tripos}.
\end{example}

\subsection{Construction of the implicative tripos}
\label{ss:ConstrImpTripos}

\begin{theorem}[Implicative tripos]\label{th:ImpTripos}
  Let $\A=(\A,{\cle},{\to},S)$ be an implicative algebra.
  For each set $I$, we write $\P{I}=\A^I/S[I]$.
  Then:
  \begin{enumerate}[(99)]
  \item[(1)] The correspondence $I\mapsto\P{I}$ induces a
    (contravariant) functor $\P:\Set^{\op}\to\HA$
  \item[(2)] The functor $\P:\Set^{\op}\to\HA$ is a $\Set$-based
    tripos.
  \end{enumerate}
\end{theorem}

\begin{proof*}
  It is clear that for each set~$I$, the poset
  $(\A^I/S[I],{\le_{S[I]}})$ is a Heyting algebra, namely: the Heyting
  algebra induced by the implicative algebra
  $(\A^I,{\cle^I},{\to^I},S[I])$.
  \begin{description}
  \item[Functoriality]\quad
    Let $I,J\in\Set$.
    Each function $f:I\to J$ induces a reindexing map
    $\A^f:\A^J\to\A^I$ defined by $\A^f(a)=a\circ f$ for all
    $a\in\A^J$.
    Now, let us consider two families $a,b\in\A^J$ such that
    $a\tnent_{S[J]}b$, that is: such that
    $\bigmeet_{j\in J}(a_j\leftrightarrow b_j)\in S$. 
    Since
    $\bigmeet_{j\in J}(a_j\leftrightarrow b_j)
    \cle\bigmeet_{i\in I}(a_{f(i)}\leftrightarrow b_{f(j)})$, we deduce
    that $\bigmeet_{i\in I}(a_{f(i)}\leftrightarrow b_{f(i)})\in S$,
    so that $\A^f(a)\tnent_{S[I]}\A^f(b)$.
    Therefore, through the quotients $\P{J}=\A^J/S[J]$ and
    $\P{I}=\A^I/S[I]$, the reindexing map $\A^f:\A^J\to\A^I$
    factors into a map $\P{f}:\P{J}\to\P{I}$.
    We now need to check that the map $\P{f}:\P{J}\to\P{I}$ is a
    morphism of Heyting algebras.
    For that, given predicates $p=[a]_{/S[J]}\in\P{J}$ and
    $q=[b]_{/S[J]}\in\P{J}$, we observe that:
    $$\begin{array}{r@{~~}c@{~~}l}
      \P{f}(p\land q)&=&
      \P{f}\bigl(\bigl[a\times^J b\bigr]_{/S[J]}\bigr)
      ~=~\P{f}\bigl(\bigl[(a_j\times b_j)_{j\in J}\bigr]_{/S[J]}\bigr)\\
      &=&\bigl[(a_{f(i)}\times b_{f(i)})_{i\in I}\bigr]_{/S[I]}
      ~=~\bigl[(a_{f(i)})_{i\in I}\times^I
        (b_{f(i)})_{i\in I}\bigr]_{/S[I]} \\
      &=&\bigl[(a_{f(i)})_{i\in I}\bigr]_{/S[I]}\land
      \bigl[(b_{f(i)})_{i\in I}\bigr]_{/S[I]}
      ~=~\P{f}(p)\land\P{f}(q)\\
    \end{array}$$
    (The case of the other connectives $\lor$, $\to$, $\bot$ and
    $\top$ is similar.)
    The contravariant functoriality of the correspondence 
    $f\mapsto\P{f}$ is obvious from the definition.
  \item[Existence of right adjoints]\quad
    Let $f:I\to J$.
    For each family $a\in\A^I$, we let
    $$\forall_f^0(a)~=~
    \biggl(\bigmeet_{f(i)=j}\!\!\!\!a_i\biggr)_{j\in J}
    \eqno({\in}~\A^J)$$
    We observe that for all $a,b\in\A^I$ and $s\in S$,
    $$s\cle\bigmeet_{i\in I}(a_i\to b_i)
    \qquad\text{implies}\qquad
    s\cle\bigmeet_{j\in J}(\forall_f^0(a)_j\to\forall_f^0(b)_j)\,.$$
    $$\begin{array}{r@{\qquad}c@{\qquad}l}
      a\ent_{S[I]}b&\text{implies}&
      \forall_f^0(a)\ent_{S[J]}\forall_f^0(b)\,,\\
      a\tnent_{S[I]}b&\text{implies}&
      \forall_f^0(a)\tnent_{S[J]}\forall_f^0(b)\,.\\
    \end{array}\leqno\begin{tabular}{@{}l}
    Therefore\\and thus\\
    \end{tabular}$$
    For each predicate $p=[a]_{/S[I]}\in\P{I}$, we can now let
    $\forall{f}(p)=\bigl[\forall_f^0(a)\bigr]_{/S[J]}\in\P{J}$.
    Given $p=[a]_{/S[I]}\in\P{I}$ and $q=[b]_{/S[J]}\in\P{J}$,
    it remains to check that:
    $$\begin{array}{rcl}
      \P{f}(q)\le p
      &\text{iff}&\ds\bigmeet_{i\in I}(b_{f(i)}\to a_i)\in S
      \quad\text{iff}\quad\bigmeet_{j\in J}
      \bigmeet_{f(i)=j}\!\!(b_j\to a_i)\in S\\
      \noalign{\smallskip}
      &\text{iff}&\ds\bigmeet_{j\in J}\Bigl(
      b_j\to\!\!\!\bigmeet_{f(i)=j}\!\!\!\!a_i\Bigr)\in S
      \quad\text{iff}\quad\bigmeet_{j\in J}
      \bigl(b_j\to\forall^0_f(a)_j\bigr)\in S\\
      \noalign{\smallskip}
      &\text{iff}&q\le\forall{f}(p)\\
    \end{array}$$
  \item[Existence of left adjoints]\quad
    Let $f:I\to J$.
    For each family $a\in\A^I$, we let
    $$\exists_f^0(a)~=~
    \biggl(\bigexists_{f(i)=j}\!\!\!a_i\biggr)_{j\in J}
    ~=~\Biggl(\bigmeet_{c\in\A}\biggl(
    \bigmeet_{f(i)=j}\!\!\!\!(a_i\to c)~\to~c\biggr)\Biggr)_{j\in J}
    \eqno({\in}~\A^J)$$
    We observe that for all $a,b\in\A^I$ and $s\in S$,
    $$s\cle\bigmeet_{i\in I}(a_i\to b_i)
    \qquad\text{implies}\qquad
    s'\cle\bigmeet_{j\in J}(\exists_f^0(a)_j\to\exists_f^0(b)_j)\,,$$
    where $s':=(\Lam{xy}{x\,(\Lam{z}{y\,(s\,z)})})^{\A}\in S$.
    $$\begin{array}{r@{\qquad}c@{\qquad}l}
      a\ent_{S[I]}b&\text{implies}&
      \exists_f^0(a)\ent_{S[J]}\exists_f^0(b)\,,\\
      a\tnent_{S[I]}b&\text{implies}&
      \exists_f^0(a)\tnent_{S[J]}\exists_f^0(b)\,.\\
    \end{array}\leqno\begin{tabular}{@{}l}
    Therefore\\and thus\\
    \end{tabular}$$
    For each predicate $p=[a]_{/S[I]}\in\P{I}$, we can now let
    $\exists{f}(p)=\bigl[\exists_f^0(a)\bigr]_{/S[J]}\in\P{J}$.
    Given $p=[a]_{/S[I]}\in\P{I}$ and $q=[b]_{/S[J]}\in\P{J}$,
    it remains to check that:
    $$\begin{array}{r@{\quad}c@{\quad}l}
      p\le\P{f}(q)
      &\text{iff}&\ds\bigmeet_{i\in I}(a_i\to b_{f(i)})\in S
      \quad\text{iff}\quad\bigmeet_{j\in J}
      \bigmeet_{f(i)=j}\!\!(a_i\to b_j)\in S\\
      \noalign{\smallskip}
      &\text{iff}&\ds\bigmeet_{j\in J}\biggl(
      \biggl(\bigexists_{f(i)=j}\!\!\!a_i\biggr)\to b_j\biggr)\in S
      \quad\text{iff}\quad\bigmeet_{j\in J}
      \bigl(\exists^0_f(a)_j\to b_j\bigr)\in S\\
      \noalign{\smallskip}
      &\text{iff}&\exists{f}(p)\le q\\
    \end{array}$$
    (Here, the third `iff' follows from the soundness of the
    elimination rule of~$\exists$.)
  \item[Beck-Chevalley condition]\quad
    Let us now check that the Beck-Chevalley condition holds for
    the functor $\P:\Set^{\op}\to\HA$.
    For that, we consider an arbitrary pullback square in the
    category~$\Set$
    $$\xymatrix{
      I\pullback{6pt}\ar[r]^{f_1}\ar[d]_{f_2}& I_1\ar[d]^{g_1} \\
      I_2\ar[r]_{g_2} & J \\
    }$$
    and we want to show that the following two diagrams commute
    (in $\Pos$):
    $$\xymatrix{
      \P{I}\ar[r]^{\exists{f_1}}& \P{I_1} \\
      \P{I_2}\ar[u]^{\P{f_2}}\ar[r]_{\exists{g_2}} &
      \P{J}\ar[u]_{\P{g_1}} \\
    }\qquad\xymatrix{
      \P{I}\ar[r]^{\forall{f_1}}& \P{I_1} \\
      \P{I_2}\ar[u]^{\P{f_2}}\ar[r]_{\forall{g_2}} &
      \P{J}\ar[u]_{\P{g_1}} \\
    }$$
    Since both commutation properties are equivalent up to the
    symmetry w.r.t.\ the diagonal
    (Remarks~\ref{r:TriposIntuitions}~(5)), we only need to prove the
    second commutation property.
    And since the correspondence $f\mapsto\forall{f}$ is functorial,
    we can assume without loss of generality that
    \begin{itemize}
    \item[$\bullet$]
      $I=\{(i_1,i_2)\in I_1\times I_2:g_1(i_1)=g_2(i_2)\}$
    \item[$\bullet$] $f_1(i_1,i_2)=i_1$ for all $(i_1,i_2)\in I$
    \item[$\bullet$] $f_2(i_1,i_2)=i_2$ for all $(i_1,i_2)\in I$
    \end{itemize}
    using the fact that each pullback diagram in~$\Set$ is of this
    form, up to a bijection.
    For all $p=[a]=\bigl[(a_i)_{i\in I_2}\bigr]\in\P{I_2}$, we check
    that:
    $$\begin{array}[b]{@{}r@{~~}c@{~~}l@{}}
      (\forall{f_1}\circ\P{f_2})(p)
      &=&\ds\forall{f_1}\bigl(\bigl[
        (a_{f_2(i_1,i_2)})_{(i_1,i_2)\in I}
        \bigr]\bigr)
      ~=~\forall{f_1}\bigl(\bigl[
        (a_{i_2})_{(i_1,i_2)\in I}
        \bigr]\bigr)\\
      \noalign{\medskip}
      &=&\ds\biggl[\biggl(\!
        \bigmeet_{\substack{(i_1,i_2)\in I\\f_1(i_1,i_2)=i'_1}}
        \!\!\!\!\!\!\!\!a_{i_2}\biggr)_{i'_1\in I_1}\biggr]
      ~=~\biggl[\biggl(\!\!\!
        \bigmeet_{\substack{i_2\in I_2\\g_2(i_2)=g_1(i_1)}}
        \!\!\!\!\!\!\!\!\!\!a_{i_2}\biggr)_{i_1\in I_1}\biggr]\\
      \noalign{\medskip}
      &=&\ds\Bigl[\Bigl(
        \bigl(\forall^0_{g_2}(a)\bigr)_{g_1(i_1)}
        \Bigr)_{i_1\in I_1}\Bigr]
      ~=~\P{g_1}\bigl(\bigl[\forall^0_{g_2}(a)\bigr]\bigr)\\
      \noalign{\medskip}
      &=&(\P{g_1}\circ\forall{g_2})(p)\\
    \end{array}$$
  \item[The generic predicate]\quad
    Let us now take $\Prop:=\A$ and
    $\Tr:=[\id_{\A}]_{/S[\A]}\in\P{\Prop}$.
    Given a set $I\in\Set$ and a predicate
    $p=\bigl[(a_i)_{i\in I}\bigr]_{/S[I]}\in\P{I}$,
    we take $f:=(a_i)_{i\in I}:I\to\A$ and check that:
    $$\P{f}(\Tr)~=~
    \P{f}\bigl(\bigl[(a)_{a\in\A}\bigr]_{/S[\A]}\bigr)
    ~=~\bigl[(a_i)_{i\in I}\bigr]_{/S[I]}~=~p\,.
    \eqno\usebox{\proofbox}$$
  \end{description}
\end{proof*}

\begin{example}[Particular case of a complete Heyting algebra]%
  \label{ex:HeytingImpTripos}
  In the particular case where the implicative algebra
  $(\A,{\cle},{\to},S)$ is a complete Heyting algebra (which means
  that $\to$ is Heyting's implication whereas the separator is
  trivial: $S=\{\top\}$), we can observe that for each set~$I$, the
  equivalence relation $\tnent_{S[I]}$ over $\A^I$ is  discrete (each
  equivalence class has one element), so that we can drop the
  quotient:
  $$\P{I}~=~\A^I/S[I]~\sim~\A^I\,.$$
  Up to this technical detail, the implicative tripos associated to
  the implicative algebra~$(\A,{\cle},{\to},S)$ is thus the very same
  as the forcing tripos associated to the underlying complete Heyting
  algebra $(\A,{\cle})$ (cf Example~\ref{ex:ForcingTripos}).
\end{example}

\subsection{Characterizing forcing
  triposes}\label{ss:CharacForcingTriposes}
Example~\ref{ex:HeytingImpTripos} shows that forcing triposes are
particular cases of implicative triposes.
However, it turns out that many implicative algebras that are not
complete Heyting algebras nevertheless induce a tripos that is
isomorphic to a forcing tripos.
The aim of this section is to characterize them, by proving the
following:
\begin{theorem}[Characterizing forcing
    triposes]\label{th:CharacForcingTriposes}
  Let $\P:\Set^{\op}\to\HA$ be the tripos induced by an implicative
  algebra $(\A,{\cle},{\to},S)$.
  Then the following are equivalent:
  \begin{enumerate}[(99)]
  \item[(1)] $\P$ is isomorphic to a forcing tripos.
  \item[(2)] The separator $S\subseteq\A$ is a principal filter
    of~$\A$.
  \item[(3)] The separator $S\subseteq\A$ is finitely generated
    and $\Fork^{\A}\in S$.
  \end{enumerate}
\end{theorem}

Before proving the theorem, let us recall that:
\begin{definition}
  Two $\Set$-based triposes $\P,\P':\Set^{\op}\to\HA$ are
  \emph{isomorphic} when there exists a natural isomorphism
  $\phi:\P\limp\P'$, that is: a family of isomorphisms
  $\phi_I:\P{I}\simto\P'{I}$ (indexed by $I\in\Set$)
  such that the following diagram commutes:
  $$\begin{array}{c}
    \xymatrix{
      I\ar[d]_f & \P{I}\ar[r]^{\phi_I}_{\sim} & \P'{I} \\
      J & \P{J}\ar[u]^{\P{f}}\ar[r]_{\phi_J}^{\sim} &
      \P'{J}\ar[u]_{\P'{f}} \\
    }
  \end{array}\eqno(\text{for all}~f:I\to J)$$
  (Note that here, the notion of isomorphism can be taken
  indifferently in the sense of $\HA$ or $\Pos$, since
  $\phi_I:\P{I}\to\P'{I}$ is an iso in~$\HA$ if and only if it is an
  iso in~$\Pos$.)
\end{definition}

\begin{remarks}
  The above definition does not take care of generic predicates,
  since any natural isomorphism $\phi:\P\limp\P'$ automatically maps
  each generic predicate $\Tr\in\P\Prop$ (for the tripos~$\P$) into a
  generic predicate $\Tr':=\phi_{\Prop}(\Tr)\in\P'\Prop$ (for the
  tripos~$\P'$).
\end{remarks}

\subsubsection{The fundamental diagram}
Given an implicative algebra $\A=(\A,{\cle},{\to},S)$ and a set~$I$,
we have seen (Section~\ref{ss:UnifPowSep}) that the separator
$S\subseteq\A$ induces two separators
$$S[I]~\subseteq~S^I~\subseteq~\A^I$$
in the power implicative structure~$\A^I$, where
$$S[I]~:=~\bigl\{(a_i)_{i\in I}\in\A^I~:~
\exists s\in S,~\forall i\in I,~s\cle a_i\bigr\}
\eqno(\text{uniform power separator})$$
We thus get the following (commutative) diagram
$$\xymatrix@R=48pt@C=48pt{
  \A^I\ar@{->>}[r]^{[\cdot]_{/S[I]}}\ar@{->>}[d]_{[\cdot]_{/S^I}}&
  \A^I/S[I]\ar@{->>}[dl]^{\widetilde{\mathrm{id}}}\ar@{->>}[d]^{\rho_I}
  \rlap{$~{=}~\P{I}$} &
  \text{\footnotesize$[(a_i)_{i\in I}]_{/S[I]}$}\ar@{|->}[d]\\
  \A^I/S^I\ar@{->>}[r]_{\alpha_I}^{\sim}&
  (\A/S)^I\rlap{$~{=}~(\P1)^I$}&
  \text{\footnotesize$([a_i]_{/S})_{i\in I}$}\\
}$$
where:
\begin{itemize}
\item $[\cdot]_{/S[I]}:\A^I\to A^I/S[I]~({=}~\P{I})$ is the quotient
  map associated to $A^I/S[I]$;
\item $[\cdot]_{/S^I}:\A^I\to A^I/S^I$ is the quotient map associated
  to $A^I/S^I$;
\item $\tilde{\id}:\A^I/S[I]\to\A^I/S^I$ is the (surjective) map that
  factors the identity of $\A^I$ through the quotients
  $\A^I/S[I]$ and $\A^I/S^I$ (remember that $S[I]\subseteq S^I$);
\item $\alpha_I=\<\tilde{\pi}_i\>_{i\in I}:
  \A^I/S^I\to(\A/S)^I$ is the canonical isomorphism
  (Prop.~\ref{p:Factorization}) between the Heyting algebras
  $\A^I/S^I$ and $(\A/S)^I~({=}~(\P1)^I)$;
\item $\rho_I:\A^I/S[I]\to(\A/S)^I$ is the (surjective) map that is
  defined by $\rho_I:=\alpha_I\circ\tilde{\id}$, so that for all
  $(a_i)_{i\in I}\in\A^I$, we have
  $$\rho_I\bigl(\bigl[(a_i)_{i\in I}\bigr]_{/S[I]}\bigr)
  ~=~\bigl([a_i]_{/S}\bigr)_{i\in I}\,.$$
\end{itemize}

\begin{proposition}\label{p:FundDiag}
  The following are equivalent:
  \begin{enumerate}[(99)]
  \item[(1)] The map $\rho_I:\P{I}\to(\P1)^I$ is injective.
  \item[(2)] The map $\rho_I:\P{I}\to(\P1)^I$ is an isomorphism
    of Heyting algebras.
  \item[(3)] Both separators $S[I]$ and $S^I$ coincide: $S[I]=S^I$.
  \item[(4)] The separator $S\subseteq\A$
    is closed under all $I$-indexed meets.
  \end{enumerate}
\end{proposition}

\begin{proof}
  $(1)\liff(2)$\quad
  Recall that a morphism of Heyting algebras is an isomorphism
  (in~$\HA$) if and only if the underlying map (in~$\Set$) is
  bijective.
  But since $\rho_I$ is a surjective morphism of Heyting algebras, it
  is clear that $\rho_I$ is an isomorphism (in~$\HA$) if and only the
  underlying map (in $\Set$) is injective.
  \smallbreak\noindent
  $(2)\liff(3)$\quad It is clear that $\rho_I$ is an iso iff
  $\tilde{\id}$ is an iso, that is: iff $S[I]=S^I$.
  \smallbreak\noindent
  $(3)\liff(4)$\quad See Prop.~\ref{p:CharacSepEqual}
  p.~\pageref{p:CharacSepEqual}.
\end{proof}

We can now present the
\begin{proof}[Proof of Theorem~\ref{th:CharacForcingTriposes}]\quad
  We have already proved that $(2)\liff(3)$
  (Prop.~\ref{p:CharacPrincFilter}, Section~\ref{sss:FinGenSep}),
  so that it only remains to prove that $(1)\liff(2)$.
  \smallbreak\noindent
  $(2)\limp(1)$\quad When $S\subseteq\A$ is a principal filter,
  the Heyting algebra $H:=\P1=\A/S$ is complete
  (Prop.~\ref{p:CharacPrincFilter}).
  Moreover, since $S$ is closed under arbitrary meets, the arrow
  $\rho_I:\P{I}\to(\P1)^I$ is an isomorphism (Prop.~\ref{p:FundDiag})
  for all sets~$I$.
  It is also clearly natural in~$I$, so that the family
  $(\rho_I)_{I\in\Set}$ is an isomorphism between the implicative
  tripos~$\P$ and the forcing tripos $I\mapsto H^I$ (where
  $H=\P1=\A/S$).
  \smallbreak\noindent
  $(1)\limp(2)$\quad Let us now assume that there is a complete
  Heyting algebra~$H$ together with a natural isomorphism
  $\phi_I:\P{I}\simto H^I$ (in~$I$).
  In particular, we have $\phi_1:\P1\inonto H^1=H$, so that
  $\A/S=\P1\sim H$ is a complete Heyting algebra.
  Now, fix a set~$I$, and write
  $c_i:=\{0\mapsto i\}:1\to I$ for each element $i\in I$.
  Via the two (contravariant) functors
  $\P,H^{(\text{--})}:\Set^{\op}\to\HA$, we easily check that
  the arrow $c_i:1\to I$ is mapped to:
  $$\begin{array}[b]{rclcl}
    \P{c_i}&=&\rho_i&:&\P{I}\to\P1\\[3pt]
    H^{c_i}&=&\pi'_i&:&H^I\to H\\
  \end{array}\leqno\text{and}$$
  where $\rho_i$ is the $i$th component of the surjection
  $\rho_I:\P{I}\onto(\P1)^I$ and where $\pi'_i$ is the $i$th
  projection from $H^I$ to $H$.
  We then observe that the two diagrams
  $$\xymatrix{
    \A/S\ar[r]^{\phi_1}_{\sim} & H \\
    \A^I/S[I]\ar[u]^{\P{c_i}=\rho_i}\ar[r]_{\phi_I}^{\sim} &
    H^I\ar[u]_{\pi'_i=H^{c_i}} \\
  }\qquad\xymatrix{
    (\A/S)^I\ar[r]^{\phi_1^I}_{\sim} & H^I \\
    \A^I/S[I]\ar[u]^{\rho_I=\<\rho_i\>_{i\in I}}
    \ar[r]_{\phi_I}^{\sim} &
    H^I\ar[u]_{\id_{H^I}=\<\pi'_i\>_{i\in I}}^{\sim} \\
  }$$
  are commutative.
  Indeed, the first commutation property comes from the naturality
  of~$\phi$, and the second commutation property follows from the
  first commutation property, by gluing the arrows $\rho_i$ and
  $\pi'_i$ for all indices $i\in I$.
  From the second commutation property, it is then clear that the
  arrow $\rho_I:\P{I}\to(\P1)^I$ is an isomorphism for all sets $I$,
  so that by Prop.~\ref{p:FundDiag}, the separator $S\subseteq\A$ is
  closed under arbitrary meets, which precisely means that it is a
  principal filter of~$\A$.
\end{proof}

\begin{remarks}
  Intuitively, Theorem~\ref{th:CharacForcingTriposes} expresses that
  forcing is the same as non-deterministic realizability (both in
  intuitionistic and classical logic).
\end{remarks}

\subsection{The case of classical
  realizability}\label{ss:TriposRealizK}
In Sections~\ref{sss:ImpStructRealizK} and~\ref{sss:ImpAlgRealizK}, we
have seen that each Abstract Krivine Structure (AKS)
$\mathcal{K}=(\Lambda,\Pi,\ldots)$ can be turned into a classical
implicative algebra $\A_{\mathcal{K}}=(\Pow(\Pi),\ldots)$.
By Theorem~\ref{th:ImpTripos}, the classical implicative algebra
$\A_{\mathcal{K}}$ induces in turn a (classical) tripos, which we
shall call the \emph{classical realizability tripos} induced by the
AKS $\mathcal{K}$.

\begin{remark}
  In~\cite{Str13}, Streicher shows how to construct a classical tripos
  (which he calls a \emph{Krivine tripos}) from an AKS, using a very
  similar construction.
  Streicher's construction is further refined in~\cite{FFGMM17}, which
  already introduces some of the main ideas underlying implicative
  algebras.
  Technically, the main difference between Streicher's construction
  and ours is that Streicher works with a smaller algebra
  $\A'_{\mathcal{K}}$ of truth values, that only contains the sets of
  stacks that are closed under bi-orthogonal:
  $$\A'_{\mathcal{K}}~=~\Pow_{\Bot}(\Pi)~=~\{S\in\Pow(\Pi)~:~
  S=S^{{\Bot}{\Bot}}\}\,.$$
  Although Streicher's algebra $\A'_{\mathcal{K}}$ is not an
  implicative algebra (it is a \emph{classical ordered combinatory
    algebra}, following the terminology of~\cite{FFGMM17}), it
  nevertheless gives rise to a classical tripos, using a construction
  that is very similar to ours.
  In~\cite{FM17}, it is shown that Streicher's tripos is actually
  isomorphic to the implicative tripos that is constructed from the
  implicative algebra $\A_{\mathcal{K}}$.
\end{remark}

The following theorem states that AKSs generate the very same class of
triposes as classical implicative algebras, so that both structures
(abstract Krivine structures and classical implicative algebras) have
actually the very same logical expressiveness:
\begin{theorem}[Universality of AKS]\label{th:UnivAKS}
  For each classical implicative algebra~$\A$, there exists an AKS
  $\mathcal{K}$ that induces the same tripos, in the sense that the
  classical realizability tripos induced by~$\mathcal{K}$ is
  isomorphic to the implicative tripos induced by~$\A$.
\end{theorem}

The proof of Theorem~\ref{th:UnivAKS} is a consequence of the
following lemma:
\begin{lemma}[Reduction of implicative algebras]\label{l:ReducImpAlg}
  Let $\A=(\A,{\cle_{\A}},{\to_{\A}},S_{\A})$ and
  $\B=(\B,{\cle_{\B}},{\to_{\B}},S_{\B})$ be
  two implicative algebras.
  If there exists a surjective map $\psi:\B\to\A$ (a `reduction
  from~$\B$ onto~$\A$') such that
  \begin{enumerate}[(99)]
  \item[(1)] $\psi\bigl(\bigmeet_{i\in I}b_i\bigr)
    =\bigmeet_{i\in I}\psi(b_i)$\hfill
    (for all $I\in\Set$ and $b\in\B^I$)
  \item[(2)] $\psi(b\to_{\B}b')=\psi(b)\to_{\A}\psi(b')$\hfill
    (for all $b,b'\in\B$)
  \item[(3)] $b\in S_{\!\B}$ iff $\psi(b)\in S_{\!\A}$\hfill
    (for all $b\in\B$)
  \end{enumerate}
  then the corresponding triposes
  $\P_{\!\A},\P_{\!\B}:\Set^{\op}\to\HA$ are isomorphic.
\end{lemma}

\begin{proof}
  For each set $I$, we consider the map $\psi^I:\B^I\to\A^I$ defined
  by $\psi^I(b)=\psi\circ b$ for all $b\in\B^I$.
  Given two points $b,b'\in\B^I$, we observe that:
  $$\begin{array}{rcl}
    b\ent_{S_{\B}[I]}b'
    &\text{iff}&\bigmeet_{i\in I}(b_i\to_{\B}b'_i)\in S_{\B} \\
    &\text{iff}&\psi\bigl(\bigmeet_{i\in I}(b_i\to_{\B}b'_i)\bigr)
    \in S_{\A}\\
    &\text{iff}&\bigmeet_{i\in I}(\psi(b_i)\to_{\A}\psi(b'_i))
    \in S_{\A}\\
    &\text{iff}&\psi^I(b)\ent_{S_{\A}[I]}\psi^I(b') \\
  \end{array}$$
  From this, we deduce that:
  \begin{enumerate}[(99)]
  \item[(1)] The map $\psi^I:\B^I\to\A^I$ is compatible with the
    preorders $\ent_{S_{\B}[I]}$ (on~$\B^I$) and
    $\ent_{S_{\A}[I]}$ (on $\A^I$), and thus factors into a
    monotonic map $\tilde{\psi}_I:\P_{\!\B}I\to\P_{\!\A}I$
    through the quotients $\P_{\!\B}I=\B^I/S_{\B}[I]$ and
    $\P_{\!\A}I=\A^I/S_{\A}[I]$.
  \item[(2)] The monotonic map
    $\tilde{\psi}_I:\P_{\!\B}I\to\P_{\!\A}I$ is an embedding of
    partial orderings, in the sense that $p\le p'$ iff 
    $\tilde{\psi}_I(p)\le\tilde{\psi}_I(p')$ for all
    $p,p'\in\P_{\!\B}I$.
  \end{enumerate}
  Moreover, since $\psi:\B\to\A$ is surjective, the maps
  $\psi^I:\B^I\to\A^I$ and
  $\tilde{\psi}_I:\P_{\!\B}I\to\P_{\!\A}I$ are surjective too, so
  that the latter is actually an isomorphism in~$\Pos$, and thus an
  isomorphism in~$\HA$.
  The naturality of $\tilde{\psi}_I:\P_{\!\B}I\to\P_{\!\A}I$ (in~$I$)
  follows from the naturality of $\psi^I:\A^I\to\B^I$ (in~$I$), which
  is obvious by construction.
\end{proof}

\begin{proof}[Proof of Theorem~\ref{th:UnivAKS}]\quad
  Let $\A=(\A,{\cle},{\to},S)$ be a classical implicative algebra.
  Following~\cite{FFGMM17}, we define
  $\mathcal{K}=(\Lambda,\Pi,@,{\,\cdot\,},\mathtt{k}_{\_},
  \mathtt{K},\mathtt{S},\cc,\mathrm{PL},\Bot)$ by
  \begin{itemize}
  \item[$\bullet$] $\Lambda=\Pi:=\A$
  \item[$\bullet$] $a@b:=ab$,\quad $a\cdot b:=a\to b$\quad and\quad
    $\mathtt{k}_{a}:=a\to\bot$\hfill (for all $a,b\in\A$)
  \item[$\bullet$] $\mathtt{K}:=\mathbf{K}^{\A}$,\quad
    $\mathtt{S}:=\mathbf{S}^{\A}$\quad and\quad
    $\cc:=\mathbf{cc}^{\A}$
  \item[$\bullet$] $\mathrm{PL}:=S$\quad and\quad
    $\Bot:=({\cle_{\A}})=\{(a,b)\in\A^2:a\cle b\}$
  \end{itemize}
  It is a routine exercise to check that the above structure is an
  AKS.
  Note that in this AKS, the orthogonal
  $\beta^{\Bot}\subseteq\Lambda$ of a set of stacks
  $\beta\subseteq\Pi$ is characterized by
  $$\beta^{\Bot}~=~\{a\in\A~:~\forall b\in\beta,~a\cle b\}
  ~=~{\bigm\downarrow}\bigl\{\bigmeet\!\beta\bigr\}$$
  From the results of Sections~\ref{sss:ImpStructRealizK}
  and~\ref{sss:ImpAlgRealizK}, the AKS $\mathcal{K}$ induces in turn a
  classical implicative algebra
  $\B=(\B,{\cle_{\B}},{\to_{\B}},S_{\B})$ that is defined by:
  \begin{itemize}
  \item[$\bullet$] $\B~:=~\Pow(\Pi)~=~\Pow(\A)$
  \item[$\bullet$] $\beta\cle_{\B}\beta'~:\liff~
    \beta\supseteq\beta'$\hfill (for all $\beta,\beta'\in\B$)
  \item[$\bullet$] $\beta\to_{\B}\beta'~:=~\beta^{\Bot}\cdot\beta'
    ~=~\bigl\{a\to a'~:~a\cle\bigmeet\!\beta,~a'\in\beta'\bigr\}$\hfill
    (for all $\beta,\beta'\in\B$)
  \item[$\bullet$] $S_{\B}~:=~\{\beta\in\B:
    \beta^{\Bot}\cap\mathrm{PL}\neq\varnothing\}
    ~=~\bigl\{\beta\in\Pow(\A):\bigmeet\!\beta\in S_{\A}\bigr\}$
  \end{itemize}
  Let us now define $\psi:\B\to\A$ by $\psi(\beta)=\bigmeet\!\beta$
  for all $\beta\in\B~({=}~\Pow(\A))$.
  We easily check that $\psi:\B\to\A$ is a reduction from the
  implicative algebra $\B$ onto the implicative algebra~$\A$ (in the
  sense of Lemma~\ref{l:ReducImpAlg}), so that by
  Lemma~\ref{l:ReducImpAlg}, the triposes induced by~$\A$ and $\B$
  are isomorphic.
\end{proof}

\subsection{The case of intuitionistic
  realizability}\label{ss:TriposRealizJ}

In Sections~\ref{sss:ImpStructRealizJ} and~\ref{sss:CaseRealizJ}, we
have seen that each PCA $P=(P,{\,\cdot\,},\mathtt{k},\mathtt{s})$
induces a quasi-implicative structure
$\A_P=(\Pow(P),{\subseteq},{\to})$ based on Kleene's implication.
In intuitionistic realizability~\cite{HJP80,Pit01,Oos08}, this
quasi-implicative structure $\A_P$ is the logical seed from which
the corresponding realizability tripos is constructed.
Indeed, recall that the intuitionistic realizability tripos
$\P:\Set^{\op}\to\HA$ induced by a PCA
$P=(P,{\,\cdot\,},\mathtt{k},\mathtt{s})$ is defined by
$$\P{I}~:=~\Pow(P)^I/{\tnent_I}
\eqno(\text{for all}~I\in\Set)$$
where $\tnent_I$ is the equivalence relation associated to the
preorder of entailment $\ent_I$ on the set $\Pow(P)^I~({=}~\A_P^I)$,
which is defined by
$$a\ent_Ia'\qquad\text{iff}\qquad
\bigcap_{i\in I}(a_i\to a'_i)\neq\varnothing
\eqno(\text{for all}~I\in\Set~\text{and}~a,a'\in\Pow(P)^I)$$

In the particular case where the PCA
$P=(P,{\,\cdot\,},\mathtt{k},\mathtt{s})$ is a (total) combinatory 
algebra, the induced quasi-implicative structure $\A_P$ turns out to
be a full implicative structure (Fact~\ref{f:PAS}~(2)
p.~\pageref{f:PAS}), and it should be clear to the reader that the
above construction coincides with the construction of the implicative
tripos induced by the implicative algebra
$(\A_P,\Pow(P)\setminus\{\varnothing\})$, which is obtained by
endowing the implicative structure $\A_P$ with the separator formed by
all nonempty truth values.
In other words:
\begin{proposition}
  For each (total) combinatory algebra
  $P=(P,{\,\cdot\,},\mathtt{k},\mathtt{s})$, the corresponding
  intuitionistic realizability tripos is an implicative tripos,
  namely: the implicative tripos induced by the implicative algebra
  $(\A_P,\Pow(P)\setminus\{\varnothing\})$ induced by~$P$.
\end{proposition}

However, there are many interesting intuitionistic realizability
triposes (for instance Hyland's effective tripos) that are induced by
PCAs whose application is not total.
To see how such realizability triposes fit in our picture, it is now
time to make a detour towards the notion of \emph{quasi-}implicative
algebra and the corresponding tripos construction.

\subsubsection{Quasi-implicative algebras}
\label{sss:QuasiImpAlg}

Quasi-implicative structures (Remark~\ref{r:ImpStruct}~(2)) differ
from (full) implicative structures in that the commutation property
$$a\to\bigmeet_{b\in B}\!\!b~=~\bigmeet_{b\in B}\!(a\to b)$$
only holds for the nonempty subsets $B\subseteq\A$, so that in general
we have $(\top\to\top)\neq\top$.

In practice, quasi-implicative structures are manipulated
essentially the same way as (full) implicative structures, the main
difference being that, in the absence of the equation
$(\top\to\top)=\top$, the operation of application $(a,b)\mapsto ab$
(Def.~\ref{d:App}) and the interpretation $t\mapsto t^{\A}$ of pure
$\lambda$-terms (Section~\ref{ss:AppLam}) are now \emph{partial}
functions.
Formally:
\begin{definition}[Interpretation of $\lambda$-terms in a
    quasi-implicative structure]
  Let $(\A,{\cle},{\to})$ be a quasi-implicative structure.
  \begin{enumerate}[(99)]
  \item[(1)] Given $a,b\in\A$, we let
    $U_{a,b}=\{c\in\A:a\cle(b\to c)\}$.
    When $U_{a,b}\neq\varnothing$, application is defined as
    $ab:=\bigmeet U_{a,b}$; otherwise, the notation $ab$ is
    undefined.
  \item[(2)] Given a \emph{partial} function $f:\A\pto\A$, the
    abstraction $\boldsymbol{\lambda}f$ is (always) defined by
    $$\boldsymbol{\lambda}f:=\mskip-22mu
    \bigmeet_{a\in\dom(f)}\mskip-20mu(a\to f(a))\,.$$
  \item[(3)] The \emph{partial} function $t\mapsto t^{\A}$ (from the
    set of closed $\lambda$-terms with parameters in~$\A$ to~$\A$)
    is defined by the equations 
    $$a^{\A}:=\,a,\qquad
    (tu)^{\A}:=\,t^{\A}u^{\A},\qquad
    (\Lam{x}{t})^{\A}:=\,\boldsymbol{\lambda}
    (a\mapsto (t\{x:=a\})^{\A})$$
    whenever the right-hand side is well defined.
  \item[(4)] The above interpretation naturally extends to
    $\lambda$-terms with $\cc$, letting\\
    $\cc^{\A}:=\bigmeet_{a,b\in\A}(((a\to b)\to a)\to a)$
    (as before).
  \end{enumerate}
\end{definition}
(The reader is invited to check that when the equation
$(\top\to\top)=\top$ holds, the above functions are total and coincide
with the ones defined in Section~\ref{ss:AppLam}.)

In the broader context of quasi-implicative structures:
\begin{itemize}
\item All the semantic typing rules of Prop.~\ref{p:SemTypingRules}
  p.~\pageref{p:SemTypingRules} remain valid (except the
  $\top$-introduc\-tion rule), provided we adapt the definition of
  semantic typing to partiality, by requiring that well-typed terms
  have a well-defined interpretation:
  $$\Gamma\vdash t:a\quad:\equiv\quad
  \FV(t)\subseteq\dom(\Gamma),~~
  (t[\Gamma])^{\A}~\text{well-defined}~~\text{and}~~
  (t[\Gamma])^{\A}\cle a\,.$$
\item The identities of Prop.~\ref{p:CombTypes}
  p.~\pageref{p:CombTypes} (about combinators
  $\mathbf{K}^{\A}$, $\mathbf{S}^{\A}$, etc.) remain valid.
\item Separators are defined the same way as for implicative
  structures (Def.~\ref{d:Separator} p.~\pageref{d:Separator}).
\item A \emph{quasi-implicative algebra} is a quasi-implicative
  structure $(\A,{\cle},{\to})$ equipped with a separator
  $S\subseteq\A$.
  As before, each quasi-implicative algebra $(\A,{\cle},{\to},S)$
  induces a Heyting algebra written $\A/S$, that is defined as the
  poset reflection of the preordered set $(\A,\ent_S)$, where $\ent_S$
  is defined by\ \ $a\ent_Sb~:\equiv~(a\to b)\in S$\ \ for all
  $a,b\in\A$.
\end{itemize}

Given a quasi-implicative algebra $(\A,{\cle},{\to},S)$, we more
generally associate to each set~$I$ the poset $\P{I}:=\A^I/S[I]$,
where $S[I]$ is the corresponding uniform power separator (same
definition as before).
It is then a routine exercise to check that:
\begin{proposition}[Quasi-implicative tripos]
  \begin{enumerate}[(99)]
  \item[(1)] The correspondence $I\mapsto\P{I}$ induces a
    (contravariant) functor $\P:\Set^{\op}\to\HA$.
  \item[(2)] The functor $\P:\Set^{\op}\to\HA$ is a $\Set$-based
    tripos.
  \end{enumerate}
\end{proposition}

From what precedes, it is now clear that:
\begin{proposition}
  Given a PCA $P=(P,{\,\cdot\,},\mathtt{k},\mathtt{s})$:
  \begin{enumerate}[(99)]
  \item[(1)] The quadruple
    $(\Pow(P),{\subseteq},{\to},\Pow(P)\setminus\{\varnothing\})$
    is a quasi-implicative algebra.
  \item[(2)] The tripos induced by the quasi-implicative algebra
    $(\Pow(P),{\subseteq},{\to},\Pow(P)\setminus\{\varnothing\})$
    is the intuitionistic realizability tripos induced by
    the PCA $P=(P,{\,\cdot\,},\mathtt{k},\mathtt{s})$.
  \end{enumerate}
\end{proposition}

At this point, the reader might wonder why we focused our study on the
notion of implicative algebra rather than on the more general notion
of quasi-implicative algebra.
The reason is that from the point of view of logic, quasi-implicative
algebras bring no expressiveness with respect to implicative algebras,
due to the existence of a simple completion mechanism that turns any
quasi-implicative algebra into a full implicative algebra, without
changing the underlying tripos.

\subsubsection{Completion of a quasi-implicative algebra}
\label{sss:Completion}
Given a quasi-implicative structure $\A=(\A,{\cle_{\A}},{\to_{\A}})$,
we consider the triple $\B=(\B,{\cle_{\B}},{\to_{\B}})$ that is
defined by:
\begin{itemize}
\item[$\bullet$] $\B~:=~\A\cup\{\top_{\!\B}\}$,
  where $\top_{\!\B}$ is a new element;
\item[$\bullet$] $b\cle_{\B}b'$\ \ iff\ \
  $b,b'\in\A$ and $b\cle_{\A}b'$, or $b'=\top_{\!\B}$\hfill
  (for all $b,b'\in\B$)
\item[$\bullet$] $b\to_{\B}b'~:=~\begin{cases}
  b\to_{\A}b'&\text{if}~b,b'\in\A\\
  \top_{\!\A}\to b'&\text{if}~b=\top_{\!\B}~\text{and}~b'\in\A\\
  \top_{\!\B}&\text{if}~b'=\top_{\!\B}\\
  \end{cases}$\hfill(for all $b,b'\in\B$)
\end{itemize}

\begin{fact}
  The triple $\B=(\B,{\cle_{\B}},{\to_{\B}})$ is a full implicative
  structure.
\end{fact}

In what follows, we shall say that the implicative structure~$\B$ is
the \emph{completion} of the quasi-implicative structure~$\A$%
\footnote{Here, the terminology of completion is a bit abusive, since
  $\B$ always extends~$\A$ with one point, even when~$\A$ is already
  a full implicative algebra.}.
Intuitively, this completion mechanism simply consists to add to the
source quasi-implicative structure $\A$ a new top element
$\top_{\!\B}$ (or, from the point of view of definitional ordering: a
new bottom element) that fulfills the equation
$(\top_{\!\B}\to\top_{\!\B})=\top_{\!\B}$ by construction.

Writing $\phi:\A\to\B$ the inclusion of~$\A$ into~$\B$, we easily
check that:
\begin{lemma}\label{l:PropPhi}
  \begin{enumerate}[(99)]
  \item[(1)] $a\cle_{\A}a'$\ \ iff\ \ $\phi(a)\cle_{\B}\phi(a')$\hfill
    (for all $a,a'\in\A$)
    \smallbreak
  \item[(2)] $\ds\phi\Bigl(\bigmeet_{i\in I}a_i\Bigr)=
    \bigmeet_{i\in I}\phi(a_i)$\hfill
    (for all $I\neq\varnothing$ and $a\in\A^I$)
    \smallbreak
  \item[(3)] $\phi(a\to_{\A}a')=\phi(a)\to_{\B}\phi(a')$\hfill
    (for all $a,a'\in\A$)
    \smallbreak
  \item[(4)] $\phi(\mathbf{K}^{\A})=\mathbf{K}^{\B}$\ \ and\ \
    $\phi(\mathbf{S}^{\A})=\mathbf{S}^{\B}$
  \end{enumerate}
\end{lemma}

\begin{proof}
  Items~(1), (2) and (3) are obvious from the definition of
  $\cle_{\B}$ and $\to_{\B}$.
  (Note that (2) only holds when $I\neq\varnothing$.)
  To prove (4), we observe that
  $$\begin{array}{rcl}
    \mathbf{K}^{\B}
    &=&\ds\bigmeet_{\!\!\!a,b\in\B\!\!\!}(a\to_{\B}b\to_{\B}a) \\
    &=&\ds\bigmeet_{\!\!\!a,b\in\A\!\!\!}(a\to_{\B}b\to_{\B}a)
    \quad{\meet}\quad
    \bigmeet_{\!\!\!a\in\A\!\!\!}(a\to_{\B}\top_{\!\B}\to_{\B}a)
    \quad{\meet}\\
    &&\ds\bigmeet_{\!\!\!b\in\A\!\!\!}
    (\top_{\!\B}\to_{\B}b\to_{\B}\top_{\!\B})
    \quad{\meet}\quad
    (\top_{\!\B}\to_{\B}\top_{\!\B}\to_{\B}\top_{\!\B})\\
    &=&\ds\bigmeet_{\!\!\!a,b\in\A\!\!\!}(a\to_{\A}b\to_{\A}a)
    \quad{\meet}\quad
    \bigmeet_{\!\!\!a\in\A\!\!\!}(a\to_{\A}\top_{\!\A}\to_{\A}a)
    \quad{\meet}\quad\top_{\B}\quad{\meet}\quad\top_{\B}\\
    &=&\ds\bigmeet_{\!\!\!a,b\in\A\!\!\!}(a\to_{\A}b\to_{\A}a)
    ~=~\phi(\mathbf{K}^{\A}) \\
  \end{array}$$
  The equality $\phi(\mathbf{S}^{\A})=\mathbf{S}^{\B}$ is proved
  similarly.
\end{proof}

From the above lemma, we immediately deduce that:
\begin{proposition}\label{p:ExtSepar}
  If $S_{\A}$ is a separator of the quasi-implicative structure~$\A$,
  then the set $S_{\B}:=S_{\A}\cup\{\top_{\!\B}\}$ is a separator of
  the implicative structure $\B$.
\end{proposition}

Now, given a quasi-implicative algebra
$\A=(\A,{\cle_{\A}},{\to_{\A}},S_{\A})$, we can define its
completion as the full implicative algebra
$\B=(\B,{\cle_{\B}},{\to_{\B}},S_{\B})$ where
\begin{itemize}
\item $(\B,{\cle_{\B}},{\to_{\B}})$ is the completion of the
  quasi-implicative structure $(\A,{\cle_{\A}},{\to_{\A}})$;
\item $S_{\B}:=S_{\A}\cup\{\top_{\!\B}\}$ is the extension of the
  separator $S_{\A}\subseteq\A$ into $\B$ (Prop.~\ref{p:ExtSepar}).
\end{itemize}
Writing $\P_{\!\A}:\Set^{\op}\to\HA$ and $\P_{\!\B}:\Set^{\op}\to\HA$
the triposes induced by~$\A$ and~$\B$, respectively, it now remains to
check that:
\begin{theorem}
  The triposes $\P_{\A},\P_{\B}:\Set^{\op}\to\HA$ are isomorphic.
\end{theorem}

\begin{proof}
  For each set~$I$, we observe that the inclusion map $\phi:\A\to\B$
  induces an inclusion map $\phi^I:\A^I\to\B^I$ (given by
  $\phi^I(a)=\phi\circ a$ for all $a\in\A^I$).
  Given $a,a'\in\A^I$, let us now check that:
  $$a\ent_{S_{\A}[I]}a'\quad\text{iff}\quad
  \phi^I(a)\ent_{S_{\B}[I]}\phi^I(a')\,.$$
  Indeed, the equivalence is clear when $I=\varnothing$, since
  $a=a'$ (for $\A^I$ is a singleton).
  And in the case where $I\neq\varnothing$, we have
  $$\begin{array}[t]{rcl}
    a\ent_{S_{\A}[I]}a'
    &\text{iff}&\ds\bigmeet_{i\in I}(a_i\to_{\A}a'_i)\in S_{\A}\\
    &\text{iff}&\ds\phi\Bigl(
    \bigmeet_{i\in I}(a_i\to_{\A}a'_i)\Bigr)\in S_{\B}\\
    &\text{iff}&\ds\bigmeet_{i\in I}
    (\phi(a_i)\to_{\B}\phi(a'_i))\in S_{\B}\\
    &\text{iff}&\phi^I(a)\ent_{S_{\B}[I]}\phi^I(a')\\
  \end{array}\eqno\begin{array}[t]{r@{}}
  \\[12pt](\text{since}~S_{\A}=S_{\B}\cap\A=\phi^{-1}(S_{\B}))\\[11pt]
  (\text{by Lemma~\ref{l:PropPhi}~(2), (3)})\\[11pt]
  \end{array}$$
  From the above equivalence, it is clear that:
  \begin{enumerate}[(99)]
  \item[(1)] The map $\phi^I:\A^I\to\B^I$ is compatible with the
    preorders $\ent_{S_{\A}[I]}$ (on~$\A^I$) and $\ent_{S_{\B}[I]}$
    (on~$\B^I$), and thus factors into a monotonic map 
    $\tilde{\phi}_I:\P_{\!\A}I\to\P_{\!\B}I$ through the quotients
    $\P_{\!\A}I:=\A^I/S_{\A}[I]$ and 
    $\P_{\!\B}I:=\B^I/S_{\B}[I]$.
  \item[(2)] The monotonic map
    $\tilde{\phi}_I:\P_{\!\A}I\to\P_{\!\B}I$ is an embedding of
    partial orderings, in the sense that $p\le p'$ iff 
    $\tilde{\phi}_I(p)\le\tilde{\phi}_I(p')$ for all
    $p,p'\in\P_{\!\A}I$.
  \item[(3)] The embedding
    $\tilde{\phi}_I:\P_{\!\A}I\to\P_{\!\B}I$
    is natural in~$I\in\Set$.
  \end{enumerate}
  To conclude that the embedding
  $\tilde{\phi}_I:\P_{\!\A}I\to\P_{\!\B}I$ is an isomorphism in
  $\Pos$---and thus an isomorphism in $\HA$---, it only remains to
  prove that it is surjective.
  For that, we consider the map $\psi:\B\to\B$ that is defined by
  $$\psi(b)~:=~\bigmeet_{c\in\B}\!((b\to c)\to c)
  \eqno(\text{for all}~b\in\B)$$
  as well as the family of maps $\psi^I:\B^I\to\B^I$ ($I\in\Set$)
  defined by $\psi^I(b):=\psi\circ b$ for all $I\in\Set$ and
  $b\in\B^I$.
  Now, given $I\in\Set$ and $b\in\B^I$, we observe that
  \begin{enumerate}[(99)]
  \item[(1)] $\psi^I(b)\tnent_{S_{\B}[I]}b$, since
    $$(\Lam{xz}{zx})^{\B}~\cle~
    \bigmeet_{i\in I}(b_i\to\psi(b_i))\quad\text{and}\quad
    (\Lam{y}{y\,\mathbf{I}})^{\B}~\cle~
    \bigmeet_{i\in I}(\psi(b_i)\to b_i)\,.$$
  \item[(2)] $\psi^I(b)\in\A^I$, since for all $i\in I$, we have
    $$\psi^I(b)_i~=~\psi(b_i)~\cle~(b\to_{\B}\bot)\to_{\B}\bot
    ~\cle~\bot\to_{\B}\bot~=~\bot\to_{\A}\bot~\cle~\top_{\A}\,.$$
  \end{enumerate}
  Therefore:\ \
  $[b]_{/S_{\B}[I]}=[\psi^I(b)]_{/S_{\B}[I]}=
  \tilde{\phi}_I\bigl([\psi^I(b)]_{/S_{\A}[I]}\bigr)$.\ \
  Hence $\tilde{\phi}_I$ is surjective.
\end{proof}

\bigbreak
From the above discussion, we can now conclude that:
\begin{theorem}
  For each PCA $P=(P,{\,\cdot\,},\mathtt{k},\mathtt{s})$,
  the intuitionistic realizability tripos induced by $P$ is
  isomorphic to an implicative tripos, namely: to the implicative
  tripos induced by the completion of the quasi-implicative algebra
  $(\Pow(P),{\subseteq},{\to},\Pow(P)\setminus\{\varnothing\})$.
\end{theorem}

\bibliographystyle{apalike}
\bibliography{paper}

\begin{thebibliography}{}

\bibitem[Barendregt, 1984]{Bar84}
Barendregt, H. (1984).
\newblock {\em The Lambda Calculus: Its Syntax and Semantics}, volume 103 of
  {\em Studies in Logic and The Foundations of Mathematics}.
\newblock North-Holland.

\bibitem[Barendregt et~al., 1983]{BCD83}
Barendregt, H., Coppo, M., and Dezani{-}Ciancaglini, M. (1983).
\newblock A filter lambda model and the completeness of type assignment.
\newblock {\em J. Symb. Log.}, 48(4):931--940.

\bibitem[Cohen, 1963]{Coh63}
Cohen, P.~J. (1963).
\newblock The independence of the continuum hypothesis.
\newblock {\em Proceedings of the National Academy of Sciences of the United
  States of America}, 50(6):1143--1148.

\bibitem[Cohen, 1964]{Coh64}
Cohen, P.~J. (1964).
\newblock The independence of the continuum hypothesis {II}.
\newblock {\em Proceedings of the National Academy of Sciences of the United
  States of America}, 51(1):105--110.

\bibitem[Coppo et~al., 1980]{CDV80}
Coppo, M., {Dezani-Ciancaglini}, M., and Venneri, B. (1980).
\newblock Principal type schemes and lambda-calculus semantics.
\newblock In Hindley, R. and Seldin, G., editors, {\em To H. B. Curry. Essays
  on Combinatory Logic, Lambda-Calculus and Formalism}, pages 480--490.
  Academic Press, London.

\bibitem[{Ferrer} and {Malherbe}, 2017]{FM17}
{Ferrer}, W. and {Malherbe}, O. (2017).
\newblock {The category of implicative algebras and realizability}.
\newblock {\em ArXiv e-prints}.

\bibitem[{Ferrer Santos} et~al., 2017]{FFGMM17}
{Ferrer Santos}, W., Frey, J., Guillermo, M., Malherbe, O., and Miquel, A.
  (2017).
\newblock Ordered combinatory algebras and realizability.
\newblock {\em Mathematical Structures in Computer Science}.

\bibitem[Friedman, 1973]{Fri73}
Friedman, H. (1973).
\newblock Some applications of {K}leene's methods for intuitionistic systems.
\newblock In {\em Cambridge Summer School in Mathematical Logic}, volume 337 of
  {\em Springer Lecture Notes in Mathematics}, pages 113--170. Springer-Verlag.

\bibitem[Girard, 1987]{Gir87}
Girard, J. (1987).
\newblock Linear logic.
\newblock {\em Theor. Comput. Sci.}, 50:1--102.

\bibitem[Girard, 1972]{Gir72}
Girard, J.-Y. (1972).
\newblock {\em Interpr{\'e}tation fonctionnelle et {\'e}limination des coupures
  de l'arith\-m{\'e}\-tique d'ordre sup{\'e}rieur}.
\newblock Doctorat d'{{\'E}}tat, Universit{\'e} Paris VII.

\bibitem[Girard et~al., 1989]{Gir89}
Girard, J.-Y., Lafont, Y., and Taylor, P. (1989).
\newblock {\em Proofs and Types}.
\newblock Cambridge University Press.

\bibitem[Griffin, 1990]{Gri90}
Griffin, T. (1990).
\newblock A formulae-as-types notion of control.
\newblock In {\em Principles Of Programming Languages (POPL'90)}, pages 47--58.

\bibitem[Guillermo and Miquel, 2015]{GM15}
Guillermo, M. and Miquel, A. (2015).
\newblock Specifying {P}eirce's law.
\newblock {\em Mathematical Structures in Computer Science}.

\bibitem[Hyland et~al., 1980]{HJP80}
Hyland, J. M.~E., Johnstone, P.~T., and Pitts, A.~M. (1980).
\newblock Tripos theory.
\newblock In {\em Math. Proc. Cambridge Philos. Soc.}, volume~88, pages
  205--232.

\bibitem[Jech, 2002]{Jec02}
Jech, T. (2002).
\newblock {\em Set theory, third millennium edition (revised and expanded)}.
\newblock Springer.

\bibitem[Kleene, 1945]{Kle45}
Kleene, S.~C. (1945).
\newblock On the interpretation of intuitionistic number theory.
\newblock {\em Journal of Symbolic Logic}, 10:109--124.

\bibitem[Koppelberg, 1989]{Kop89}
Koppelberg, S. (1989).
\newblock {\em Handbook of Boolean algebras, Vol. 1}.
\newblock North-Holland.

\bibitem[Krivine, 1993]{Kri93}
Krivine, J.~L. (1993).
\newblock {\em Lambda-calculus, types and models}.
\newblock Ellis Horwood.
\newblock Out of print, now available at the author's web page at:\\
  \texttt{https://www.irif.univ-%
  paris-diderot.fr/\string~krivine/articles/Lambda.pdf}.

\bibitem[Krivine, 2001]{Kri01}
Krivine, J.-L. (2001).
\newblock Typed lambda-calculus in classical {Z}ermelo-{F}raenkel set theory.
\newblock {\em Arch. Math. Log.}, 40(3):189--205.

\bibitem[Krivine, 2003]{Kri03}
Krivine, J.-L. (2003).
\newblock Dependent choice, `quote' and the clock.
\newblock {\em Theor. Comput. Sci.}, 308(1-3):259--276.

\bibitem[Krivine, 2009]{Kri09}
Krivine, J.-L. (2009).
\newblock Realizability in classical logic.
\newblock In {\em Interactive models of computation and program behaviour},
  volume~27 of {\em Panoramas et synth{\`e}ses}, pages 197--229.
  Soci{\'e}t{\'e} Math{\'e}matique de France.

\bibitem[Krivine, 2011]{Kri11}
Krivine, J.-L. (2011).
\newblock Realizability algebras : a program to well order {R}.
\newblock {\em Logical Methods in Computer Science}, 7:1--47.

\bibitem[Krivine, 2012]{Kri12}
Krivine, J.-L. (2012).
\newblock Realizability algebras {II} : new models of {ZF} + {DC}.
\newblock {\em Logical Methods for Computer Science}, 8(1:10):1--28.

\bibitem[Leivant, 1983]{Lei83}
Leivant, D. (1983).
\newblock Polymorphic type inference.
\newblock In {\em Proceedings of the 10th ACM Symposium on Principles of
  Programming Languages}, pages 88--98.

\bibitem[McCarty, 1984]{McCPhd}
McCarty, D. (1984).
\newblock {\em Realizability and Recursive Mathematics}.
\newblock PhD thesis, Oxford University.

\bibitem[Miquel, 2000]{Miq00}
Miquel, A. (2000).
\newblock A model for impredicative type systems, universes, intersection types
  and subtyping.
\newblock In {\em LICS}, pages 18--29.

\bibitem[Miquel, 2010]{Miq10}
Miquel, A. (2010).
\newblock Existential witness extraction in classical realizability and via a
  negative translation.
\newblock {\em Logical Methods for Computer Science}.

\bibitem[Miquel, 2011]{Miq11}
Miquel, A. (2011).
\newblock Forcing as a program transformation.
\newblock In {\em LICS}, pages 197--206. IEEE Computer Society.

\bibitem[Myhill, 1973]{Myh73}
Myhill, J. (1973).
\newblock Some properties of intuitionistic {Z}ermelo-{F}raenkel set theory.
\newblock {\em Lecture Notes in Mathematics}, 337:206--231.

\bibitem[Parigot, 1997]{Par97}
Parigot, M. (1997).
\newblock Proofs of strong normalisation for second order classical natural
  deduction.
\newblock {\em Journal of Symbolic Logic}, 62(4):1461--1479.

\bibitem[Pitts, 1981]{Pit81}
Pitts, A.~M. (1981).
\newblock {\em The theory of triposes}.
\newblock PhD thesis, University of Cambridge.

\bibitem[Pitts, 2001]{Pit01}
Pitts, A.~M. (2001).
\newblock Tripos theory in retrospect.
\newblock {\em Mathematical Structures in Computer Science}.

\bibitem[{Ronchi della Rocca} and Venneri, 1984]{RRV84}
{Ronchi della Rocca}, S. and Venneri, B. (1984).
\newblock Principal type schemes for an extended type theory.
\newblock {\em Theoretical Computer Science}, 28:151--169.

\bibitem[Ruyer, 2006]{Ruy06}
Ruyer, F. (2006).
\newblock {\em Preuves, types et sous-types}.
\newblock Th{\`e}se de doctorat, Universit{\'e} Savoie Mont Blanc.

\bibitem[Streicher, 2013]{Str13}
Streicher, T. (2013).
\newblock Krivine's classical realisability from a categorical perspective.
\newblock {\em Mathematical Structures in Computer Science}.

\bibitem[Tait, 1967]{Tai67}
Tait, W. (1967).
\newblock Intensional interpretation of functionals of finite type {I}.
\newblock {\em Journal of Symbolic Logic}, 32(2).

\bibitem[van Bakel et~al., 1994]{BLRU94}
van Bakel, S., Liquori, L., {Ronchi della Rocca}, S., and Urzyczyn, P. (1994).
\newblock Comparing cubes.
\newblock In Nerode, A. and Matiyasevich, Y.~V., editors, {\em Proceedings of
  LFCS'94. Third International Symposium on Logical Foundations of Computer
  Science}, volume 813 of {\em Lecture Notes in Computer Science}, pages
  353--365. Springer-Verlag.

\bibitem[van Oosten, 2002]{Oos02}
van Oosten, J. (2002).
\newblock Realizability: a historical essay.
\newblock {\em Mathematical Structures in Computer Science}, 12(3):239--263.

\bibitem[van Oosten, 2008]{Oos08}
van Oosten, J. (2008).
\newblock {\em Realizability, an Introduction to its Categorical Side}.
\newblock Elsevier.

\bibitem[Werner, 1994]{WerPhD}
Werner, B. (1994).
\newblock {\em Une th{\'e}orie des Constructions Inductives}.
\newblock PhD thesis, Universit{\'e} Paris VII.

\end{thebibliography}

\end{document}